\documentclass[reqno,11pt]{amsart}
\usepackage{amsmath}
\usepackage{amssymb}
\usepackage[left=0.9in,right=0.9in]{geometry}
\usepackage{xcolor}
\definecolor{light-gray}{gray}{0.7}

\newtheorem{theorem}{Theorem}
\newtheorem{lemma}[theorem]{Lemma}

\newtheorem{proposition}[theorem]{Proposition}
\newtheorem{definition}[theorem]{Definition}
\newtheorem{remark}[theorem]{Remark}

\newcommand{\nocontentsline}[3]{}
\newcommand{\tocless}[2]{\bgroup\let\addcontentsline=\nocontentsline#1{#2}\egroup}

\usepackage{amsmath}
\usepackage{amssymb}
\usepackage{amsfonts}
\usepackage{enumerate}
\usepackage{comment}
\usepackage{braket}

\usepackage{palatino}
\usepackage[T1]{fontenc}
\usepackage[mathscr]{eucal}

\usepackage{acronym}
\usepackage{latexsym}
\usepackage{paralist}
\usepackage{wasysym}
\usepackage{xspace}

\usepackage[font=small,labelfont=bf]{caption}
\usepackage{subfigure}
\usepackage{tikz}
\usetikzlibrary{calc,patterns}

\usepackage{hyperref}
\hypersetup{
colorlinks=true,
linktocpage=true,
pdfstartview=FitH,
breaklinks=true,
pdfpagemode=UseNone,
pageanchor=true,
pdfpagemode=UseOutlines,
plainpages=false,
bookmarksnumbered,
bookmarksopen=false,
bookmarksopenlevel=1,
hypertexnames=true,
pdfhighlight=/O,
urlcolor=blue!60!black,linkcolor=blue!70!black,citecolor=green!70!black, 
pdftitle={},
pdfauthor={},
pdfsubject={},
pdfkeywords={},
pdfcreator={pdfLaTeX},
pdfproducer={LaTeX with hyperref}
}

\numberwithin{equation}{section}  
\numberwithin{theorem}{section} 

\usepackage[sort&compress,capitalize,nameinlink]{cleveref}
\crefname{app}{Appendix}{Appendices}

\crefrangeformat{equation}{\upshape(#3#1#4)\textendash(#5#2#6)}

\usepackage[textwidth=30mm]{todonotes}

\usepackage{soul}
\setstcolor{red}
\sethlcolor{SkyBlue}

\newcommand{\h}{\mathbf{h}}
\newcommand{\w}{\mathbf{w}}

\newcommand{\scalar}[2]{\langle #1 , #2\rangle}

\newcommand{\si}{\sigma} 
\renewcommand{\l}{\lambda} 
\newcommand{\e}{\varepsilon} 
\renewcommand{\a}{\alpha} 
\renewcommand{\b}{\beta} 
\newcommand{\ent}{{\rm ENT} }

\newcommand{\tc}{\, |\, }

\newcommand{\ind}{\mathbf{1}}

\newcommand{\pimin}{\pi_{\rm min}}
\newcommand{\pimax}{\pi_{\rm max}}
\newcommand{\g}{\gamma} 
\newcommand{\bd}{\mathbf d}
\newcommand{\bP}{\mathbf P}
\newcommand{\bE}{\mathbf E}
\newcommand{\D}{\Delta}
\renewcommand{\d}{\delta}

\newcommand{\tx}{{\textsc{tx}}}
\renewcommand{\t}{\tau}

\newcommand{\tv}{\texttt{TV}}
\newcommand{\tent}{T_\ent}
\newcommand{\muin}{\mu_{\rm in}}
\newcommand{\G}{\Gamma}
\newcommand{\maxtwo}[2]{\max_{\substack{#1 \\ #2}}} 

\newcommand{\cA}{\ensuremath{\mathcal A}} 
\newcommand{\cB}{\ensuremath{\mathcal B}} 
\newcommand{\cC}{\ensuremath{\mathcal C}} 
\newcommand{\cD}{\ensuremath{\mathcal D}} 
\newcommand{\cE}{\ensuremath{\mathcal E}} 
\newcommand{\cF}{\ensuremath{\mathcal F}} 
\newcommand{\cG}{\ensuremath{\mathcal G}} 
\newcommand{\cH}{\ensuremath{\mathcal H}} 
\newcommand{\cI}{\ensuremath{\mathcal I}}

\newcommand{\cL}{\ensuremath{\mathcal L}}

\newcommand{\cQ}{\ensuremath{\mathcal Q}}

\newcommand{\cT}{\ensuremath{\mathcal T}} 
 
\newcommand{\cV}{\ensuremath{\mathcal V}} 
\newcommand{\cW}{\ensuremath{\mathcal W}} 
\newcommand{\cX}{\ensuremath{\mathcal X}} 
\newcommand{\cY}{\ensuremath{\mathcal Y}}

\newcommand{\bbC}{{\ensuremath{\mathbb C}} } 
 
\newcommand{\bbE}{{\ensuremath{\mathbb E}} }

\newcommand{\bbN}{{\ensuremath{\mathbb N}} } 
 
\newcommand{\bbP}{{\ensuremath{\mathbb P}} } 
 
\newcommand{\bbR}{{\ensuremath{\mathbb R}} }

\newcommand{\E}{\ensuremath{\mathbb{E}}}

\renewcommand{\P}{\ensuremath{\mathbb{P}}}

\def\({\left(}
\def\){\right)}
\def\[{\left[}
\def\]{\right]}
\newacro{NE}{Nash equilibrium}
\newacroplural{NE}[NE]{Nash equilibria}
\newacro{PNE}{pure Nash equilibrium}
\newacroplural{PNE}[PNE]{Nash equilibria}
\newacro{PFNE}{prior-free Nash equilibrium}
\newacroplural{PFNE}[PFNE]{prior-free Nash equilibria}
\newacro{WE}{Wardrop equilibrium}
\newacroplural{WE}[WE]{Wardrop equilibria}
\newacro{SO}{socially optimum}

\newacro{KKT}{Karush\textendash Kuhn\textendash Tucker}
\newacro{OD}[O/D]{origin-destination}
\newacro{PoA}{price of anarchy}
\newacro{PoS}{price of stability}
\newacro{PoCS}{price of correlated stability}
\newacro{BPR}{bureau of public roads}
\newacro{FIP}{finite improvement property}

\newacro{BPG}{buck-passing game}
\newacro{SBPG}{stochastic buck-passing game}
\newacro{MBPG}{mixed extension of the buck-passing game}

\begin{document}
\title
{Stationary Distribution and Cover Time of sparse directed configuration models}
\author[P.~Caputo]{Pietro Caputo$^{\#}$}
\address{$^{\#}$ Dipartimento di Matematica e Fisica, Universit\`a di Roma Tre, Largo S. Leonardo Murialdo 1, 00146 Roma, Italy.}
\email{caputo@mat.uniroma3.it}
\author[M.~Quattropani]{Matteo Quattropani$^{\flat}$}
\address{$^{\flat}$ Dipartimento di Matematica e Fisica, Universit\`a di Roma Tre, Largo S. Leonardo Murialdo 1, 00146 Roma, Italy.}
\email{matteo.quattropani@uniroma3.it}
\maketitle              
\begin{abstract}
We consider sparse digraphs generated by the configuration model with given in-degree and out-degree sequences. 
We establish that with high probability the cover time is linear up to a poly-logarithmic correction. For 
 a large class of degree sequences we determine the exponent $\g\geq1$ of the logarithm and show that the cover time grows as $n\log^{\g}(n)$, where $n$ is the number of vertices. The results are obtained by analysing 
 the extremal values of the stationary distribution of the digraph. 
In particular, we show that the stationary distribution $\pi$ is uniform up to a poly-logarithmic factor,  and that for 
 a large class of degree sequences the minimal values of $\pi$ have the form $\frac1n\log ^{1-\g}(n)$, while the maximal values of $\pi$ behave as  $\frac1n\log ^{1-\kappa}(n)$ for some 
other exponent $\kappa\in[0,1]$. In passing, we prove tight bounds on the diameter of the digraphs and show that the latter coincides with the typical distance between two vertices. 
\end{abstract}
%

\thispagestyle{empty}

\section{Introduction}
The problem of determining the cover time of a graph is a central one in combinatorics and probability  \cite{aleliunas1979random,aldous1983time,kahn1989cover,aldous1989introduction,lovasz1993random,feige1995tight,feige1995tight2}. In recent years, the cover
time of random graphs has been extensively studied \cite{jonasson1998cover,CF1,cooper2007cover,CF2,CF3}.
All these works consider undirected graphs, with the notable exception of the paper \cite{CF2} by Cooper and Frieze, where the authors compute the cover time of directed Erd\H{o}s-Renyi random graphs in the regime of strong connectivity, that is with a logarithmically diverging average degree. The main difficulty in the directed case is that, in contrast with the undirected case, the graph's stationary distribution is an unknown random variable.  

In this paper we address the problem of determining the cover time of sparse random digraphs with  bounded  degrees. More specifically, we consider random digraphs $G$ with given in- and out-degree sequences, generated via the configuration model. For the sake of this introductory discussion let us look at the special case where all vertices have either in-degree 2 and out-degree 3 or in-degree 3 and out-degree 2, with the two types evenly represented in the vertex set $V(G)$. We refer to this as the $(2,3)(3,2)$ case. With high probability $G$ is strongly connected and we may 
ask how long the random walk on $G$ takes to cover all the nodes. The expectation of this quantity, maximized over the initial point of the walk defines $T_{\rm cov}(G)$, the cover time of $G$. We will show that with high probability as the number of vertices $n$ tends to infinity one has
  \begin{equation}\label{tcovo}
T_{\rm cov}(G)\asymp n
\log^\g (n)
\end{equation}
 where $\g = \frac{\log 3}{\log 2}\approx 1.58$, and $a_n\asymp b_n$ stands for $C^{-1}\leq a_n/b_n \leq C$ for some constant $C>0$.   The constant $\g$ can be understood in connection with the statistics of the extremal  values of the stationary distribution $\pi$ of $G$. Indeed, following the theory developed by 
Cooper and Frieze, if the graphs satisfy suitable requirements, 
then the problem of determining the cover time can be reformulated in terms of the control of the minimal values of $\pi$. In particular, we will see that the hitting time of a vertex $x\in V(G)$ effectively behaves as an exponential random variable with parameter $\pi(x)$, and that to some extent these random variables are weakly dependent. This supports the heuristic picture that represents the cover time as the expected value of $n$ independent exponential random variables, each with parameter $\pi(x)$, $x\in V(G)$. Controlling the stationary distribution is however a rather challenging task, especially if the digraphs have bounded degrees. 

Recently, Bordenave, Caputo and Salez \cite{BCS1} analyzed the mixing time of sparse random digraphs with given degree sequences and their work provides some important information on the distribution of the values of $\pi$. In particular, in the $(2,3)\-(3,2)$ case, the empirical distribution of the values $\{n\pi(x),\, x\in V(G)\}$ converges as $n\to\infty$ to the probability law $\mu$ on $[0,\infty)$ of the random variable $X$ given by
 \begin{equation}\label{eq:receq}
X=\tfrac25\sum_{k=1}^N Z_k\,,
\end{equation}
where $N$ is the random variable with $N=2$ with probability $\frac12$ and $N=3$  with probability $\frac12$, and the $Z_k$ are independent and identically distributed mean-one random variables uniquely determined by the recursive distributional equation 
 \begin{equation}\label{eq:receq2}
Z_1 \stackrel{d}{=}  
\tfrac{1}{M}\sum_{k=1}^{5-M}Z_k,
\end{equation}
where $M$ is the random variable with $M=2$ with probability $2/5$ and $M=3$  with probability $3/5$,  independent of the $Z_k$'s, and $\stackrel{d}{=}$ denotes equality in distribution.

This gives convergence of the distribution of the bulk values of $\pi$, that is of the values of $\pi$ on the scale $1/n$. What enters in the cover time analysis are however the extremal values, notably the minimal ones, and thus what is needed is a local convergence result towards the left tail of $\mu$, which cannot be extracted from the analysis in \cite{BCS1}. To obtain a heuristic guess of the size of the minimal values of $\pi$ at large but finite $n$ one may pretend that  
the values of $n\pi$ are $n$ i.i.d.\ samples from $\mu$. This would imply that $\pimin$, the minimal value of $\pi$ is such that 
$n\pimin\sim \e(n)$ where $\e(n)$ is a sequence for which $n\mu([0,\e(n)])\sim 1$, if $ \mu([0,x])$ denotes the mass given by $\mu$ to the interval $[0,x]$.

Recursive distributional equations of the form \eqref{eq:receq2} are well studied, and many properties of the distribution $\mu$ can be derived. In particular, it has been shown by Liu \cite{Liu1} that the left tail of $\mu$ is of the form 
$$  
\log \mu([0,x])\asymp - x^{-\a}\,,\qquad x\to 0^+, 
$$
where $\a=1/(\g-1)$, with the coefficient $\g$ taking the value $\g = \frac{\log 3}{\log 2}$ in the $(2,3)(3,2)$ case.  Thus, returning to our heuristic reasoning, one has that the minimal value of $\pi$ should 
satisfy 
  \begin{equation}\label{pimin1}
n\pimin \asymp \log ^{1-\g}(n).
\end{equation}
Moreover, this argument 
also predicts that with high probability there should be at least $n^\b$ vertices $x\in V(G)$, for some constant $\b>0$,
such that $n\pi(x)$ is as small as $O(\log^{1-\g}( n))$. 

A similar heuristic argument, this time based on the analysis of the right tail of $\mu$, see  \cite{Liu2,Liu3}, predicts that $\pimax$, the maximal value of $\pi$, should satisfy 
  \begin{equation}\label{pimaxo}
n\pimax \asymp \log ^{1-\kappa}(n),
\end{equation}
where $\kappa$ takes the value $\kappa = \frac{\log 2}{\log 3}\approx 0.63$ in the $(2,3)(3,2)$ case. 

Our main results below will confirm these heuristic predictions. 
The proof involves the 
analysis of the statistics of 
the in-neighbourhoods of a node. Roughly speaking, it will be seen that the smallest values of $\pi$ are achieved at vertices $x\in V(G)$ whose in-neighbourhood at distance $\log_2\log n$ is a directed tree composed entirely of vertices with in-degree 2 and out-degree 3, while the the maximal values of $\pi$ are achieved at $x\in V(G)$ whose in-neighbourhood at distance $\log_3\log n$ is a directed tree composed entirely of vertices with in-degree 3 and out-degree 2. 
Once the results \eqref{pimin1} and \eqref{pimaxo} are established, the cover time asymptotic \eqref{tcovo} will follow from an appropriate implementation of the Cooper-Frieze approach. 

We conclude this preliminary discussion by comparing our estimates \eqref{pimin1} and \eqref{pimaxo} with related results for different random graph models. 
The asymptotic of extremal values of $\pi$ has been determined in \cite{CF2} for the directed Erd\H{o}s-Renyi random graphs with logarithmically diverging average degree. There, the authors show that $n\pimin$ and $n\pimax$ are essentially of order 1, which can be interpreted as a concentration property 
enforced by the divergence of the degrees. On the other hand, for uniformly random out-regular digraphs, that is with constant out-degrees but random in-degrees, the recent paper \cite{AbBP} shows that the stationary distribution restricted to the strongly connected component satisfies $n\pimin=n^{-\eta+o(1)}$, where $\eta$ is a 
computable constant, and $n\pimax=n^{o(1)}$. Indeed, in this model in contrast with our setting one can have in-neighborhoods made by long and thin filaments which determine a power law deviation from uniformity.    

We now turn to a more systematic exposition of our results.

\subsection{Model and statement of results}\label{se:results}
Set $[n]=\{1,\dots,n\}$, and 
 for each integer $n$, fix two sequences $\bd^+=(d_x^+)_{x\in[n]}$ and $\bd^-=(d_x^-)_{x\in[n]}$ of positive integers such that 
 \begin{equation}\label{eq:degs}
 m=\sum_{x=1}^nd_x^+=\sum_{x=1}^nd_x^-.
 \end{equation} 
 The {\em directed configuration model} DCM($\bd^\pm$) is the distribution of the random digraph $G$ with vertex set $V(G)=[n]$ obtained by the following procedure: 1) equip each node $x$ with $d_x^+$ tails and $d_x^-$ heads; 2) pick uniformly at random one of the $m!$ bijective maps  
 from the set of all tails into the set of all heads, call it $\omega$; 3) for all $x,y\in [n]$, add a directed edge $(x,y)$ every time a tail from $x$ is mapped into a head from $y$ through $\omega$. The resulting digraph $G$ may have self-loops and multiple edges, however it is classical that by conditioning on the event that there are no multiple edges and no self-loops $G$ has the uniform distribution among simple digraphs with in degree sequence $\bd^-$ and out degree sequence $\bd^+$.  
 
 Structural properties of digraphs obtained in this way have been studied in  ~\cite{CF0}. Here we consider the sparse case corresponding to bounded degree sequences and, in order to avoid non irreducibility issues, we shall assume that all degrees are at least $2$. Thus, from now on it will always be assumed that 
 \begin{equation}\label{degs}
\d_\pm=\min_{x\in[n]}d_x^\pm\geq 2\,
\qquad\quad
\D_\pm=\max_{x\in[n]}d_x^\pm=O(1). 
\end{equation}
Under the first assumption  it is known that DCM($\bd^\pm$) is strongly connected with high probability. Under the second assumption, it is known that DCM($\bd^\pm$) has a uniformly (in $n$) positive probability of having no self-loops nor multiple edges. In particular, any property that holds with high probability for DCM($\bd^\pm$) will also hold with high probability for a uniformly random simple digraph with degrees given by $\bd^-$ and $\bd^+$ respectively. Here and throughout the rest of the paper we say that a property holds with high probability (w.h.p. for short) if the probability of the corresponding event converges to $1$ as $n\to\infty$. 

The (directed) distance $d(x,y)$ from $x$ to $y$ is the minimal number of edges that need to be traversed to reach $y$ from $x$. The diameter is the maximal distance between two distinct vertices, i.e.\ 
\begin{equation}\label{eq:diam}
{\rm diam}(G)=\max_{x\not=y}d(x,y).
\end{equation}


We begin by showing 
that the diameter ${\rm diam}(G)$ concentrates around the value $c\log n$ within a $O(\log\log n)$ window, where $c$ is given by $c=1/\log \nu$ and $\nu$ is the parameter defined by
\begin{equation}\label{eq:nu}
\nu=\frac{1}{m}\sum_{y=1}^nd_y^{-}d_y^{+}.
\end{equation}

 \begin{theorem}\label{th:diameter}
Set ${\rm d}_\star=\log_\nu n$. There exists $\varepsilon_n=O\(\frac{\log\log(n)}{\log(n)}\)$ such that
\begin{equation}
\P\((1-\varepsilon_n)\,{\rm d}_\star\leq {\rm diam}(G)\leq (1+\varepsilon_n)\,{\rm d}_\star\)=1-o(1).
\end{equation}
Moreover, for any $x,y\in[n]$
\begin{equation}
\P\((1-\varepsilon_n)\,{\rm d}_\star\leq d(x,y)\leq (1+\varepsilon_n)\,{\rm d}_\star\)=1-o(1).
\end{equation}

\end{theorem}
The proof of Theorem \ref{th:diameter} is a directed version of a classical argument for undirected graphs \cite{BFdv}. It requires controlling the size of in- and out-neighborhoods of a node, which in turn follows  ideas from  \cite{AbBP} and \cite{BCS1}. The value ${\rm d}_\star=\log_\nu n$ can be interpreted as follows: both the in- and the out-neighborhood of a node are tree-like with average branching given by $\nu$, so that their boundary at depth $h$ has typically size $\nu^h$, see Lemma \ref{eq:lemmasize-in}; if the in-neighborhood of $y$ and the out-neighborhood of $x$ are exposed up to depth $h$, one finds that the value $h=\frac12\log_\nu(n)$ is critical for the formation of an arc connecting the two neighborhoods.

In particular, Theorem \ref{th:diameter} shows that w.h.p.\ the digraph is strongly connected, so there exists a unique stationary distribution $\pi$ characterized by the equation
\begin{equation}\label{eq:stat}
\pi(x)=\sum_{y=1}^n\pi(y)P(y,x)\,,\qquad x\in[n],
\end{equation}
with the normalization $\sum_{x=1}^n\pi(x)=1$. Here $P$ is the transition matrix of  the simple random walk on $G$, namely
\begin{equation}\label{eq:Pxy}
P(y,x)=\frac{m(y,x)}{d^+_y},
\end{equation}
and we write $m(y,x)$ for the multiplicity of the edge $(y,x)$ in the digraph $G$.  
If the sequences $\bd^\pm$ are such that $d_x^+=d_x^-$ for all $x\in[n]$, then the stationary distribution is given by 
\begin{equation}\label{eq:Eulerian}
\pi(x)=
\frac{d^\pm_x}{m}.
\end{equation} 
The digraph is called Eulerian in this case. In all other cases the stationary distribution is a nontrivial random variable. 
To discuss our results on the extremal values of $\pi$ it is convenient to introduce the following notation. 
\begin{definition}\label{def:types}
We say that a vertex $x\in [n]$ is of type $(i,j)$, and write $x\in \cV_{i,j}$, if $(d^-_x,d^+_x)=(i,j)$. 
We call $\cC=\cC(\bd^\pm)$ the set of all types that are present in the double sequence $\bd^\pm$, that is 
$\cC=\{(i,j):\, |\cV_{i,j}|>0\}.$ 
The assumption \eqref{degs} implies that the number of distinct types is bounded by a fixed constant $C$ independent of $n$, that is $|\cC|\leq C$.  We say that the type $(i,j)$ has linear size, 
if  
\begin{equation}
\liminf_{n\to\infty}\frac{|\cV_{i,j}|}{n}>0.
\end{equation}
We call $\cL\subset \cC$ the set of  types with linear size, and define the parameters
\begin{equation}\label{eq:def-gamma01}
\gamma_0:=\frac{\log\D_+}{\log \d_-}\,,\qquad\gamma_1:=\max_{(k,\ell)\in\cL}\frac{\log \ell}{\log k}\,,\qquad\kappa_1:=\min_{(k,\ell)\in\cL}\frac{\log \ell}{\log k}\,,\qquad \kappa_0:=\frac{\log\d_+}{\log \D_-}.
\end{equation}
\end{definition}

\begin{theorem}\label{th:pimin}
Set $\pimin=\min_{x\in[n]}\pi(x)$. 
There exists a constant $C>0$ such that 
\begin{equation}\label{eq:th-pimin-upperandlower}
\P\left(C^{-1}\log^{1-\gamma_0}(n)\leq n\pimin\leq C\,\log^{1-\gamma_1}(n)
\right)=1-o(1).
\end{equation}
Moreover, there exists $\b>0$ such that 
\begin{equation}\label{eq:th-pimin-size}
\P\Big(\exists S\subset[n],\:|S|\geq n^\b\,,\; n \max_{y\in S}\pi(y)\leq C\log^{1-\gamma_{1}}(n) \Big)=1-o(1).
\end{equation}

\end{theorem}

\begin{remark}\label{rem:optimal}
Notice that $\g_0\geq \g_1\geq 1$.  If the sequences $\bd^\pm$ are such that 
$(\d_-,\D_+)\in\cL$, then 
$\g_0=\g_1=:\g$, so in these cases Theorem \ref{th:pimin} implies that 
\begin{equation}\label{eq:pminopt}
\pimin \asymp \frac1n\log^{1-\gamma}(n)\,\qquad {\rm w.h.p.}
\end{equation}
In all other cases, the estimate \eqref{eq:th-pimin-upperandlower} can be strengthened by replacing $\g_0$ with $\g'_0$ where
\begin{equation}\label{eq:g0prime}
\g'_0:=\frac{\log\D'_+}{\log \d'_-}\,,\qquad \D'_+ := \max\{\ell:\; (k,\ell)\in\cL_0\} \,,\quad \d'_- := \min\{k:\; (k,\ell)\in\cL_0\},
\end{equation}
and $\cL_0\subset \cC$ is defined as the set of $(k,\ell)\in\cC$ such that
\begin{equation}\label{eq:g0p}
\limsup_{n\to\infty}\frac{|\cV_{k,\ell}|}{n^{1-a}}=+\infty\,,\qquad \forall a>0.
\end{equation}
We refer to Remark \ref{re:gamma0L} below for additional details on this improvement. 
\end{remark}


Concerning the maximal values of $\pi$ we establish the following estimates. 
\begin{theorem}\label{th:pimax}
Set $\pimax=\max_{x\in[n]}\pi(x)$. There exists a constant $C>0$ such that
\begin{equation}\label{eq:th-pimax-upperandlower}
\P\left(C^{-1}\log^{1-\kappa_1}(n)\leq n\pimax\leq \log^{1-\kappa_0}(n)
\right)=1-o(1).
\end{equation}
Moreover, there exists $\b>0$ such that 
\begin{equation}\label{eq:th-pimax-size}
\P\Big(\exists S\subset[n],\:|S|\geq n^\b\,,\; n \min_{y\in S}\pi(y)\geq C^{-1}\log^{1-\kappa_{1}}(n) \Big)=1-o(1).
\end{equation}
\end{theorem}

\begin{remark}
Notice that $\kappa_0\leq \kappa_1\leq 1$.  If the sequences $\bd^\pm$ are such that 
$(\D_-,\d_+)\in\cL$, then 
$\kappa_0=\kappa_1=:\kappa$, and in these cases Theorem \ref{th:pimax} implies 
\begin{equation}\label{eq:pmaxopt}
\pimax \asymp \frac1n\log^{1-\kappa}(n)\,\qquad {\rm w.h.p.}
\end{equation}
In analogy with Remark \eqref{rem:optimal}, if $(\D_-,\d_+)\notin\cL$, then \eqref{eq:th-pimax-upperandlower} can be improved by replacing 
$\kappa_0$ with $\kappa'_0$ where
\begin{equation}\label{eq:k0prime}
\kappa'_0:=\frac{\log\d'_+}{\log \D'_-}\,,\qquad \d'_+ := \min\{\ell:\; (k,\ell)\in\cL_0\} \,,\quad \D'_- := \max\{k:\; (k,\ell)\in\cL_0\},
\end{equation}

\end{remark}


We turn to a description of our results concerning the cover time. Let $X_t$, $t=0,1,2,\dots$, denote the simple random walk on the digraph $G$, that is the Markov chain with transition matrix $P$ defined in \eqref{eq:Pxy}. Consider the hitting times 
\begin{equation}
H_y=\inf\{ t\ge 0:\: X_t=y \}\,,\qquad \tau_{\rm cov}=\max_{y\in [n]}H_y.
\end{equation}
The cover time $T_{\rm cov}=T_{\rm cov}(G)$ is defined by
\begin{equation}
T_{\rm cov}=\max_{x\in[n]}\,\bE_x[\tau_{\rm cov}],
\end{equation}
where $\bE_x$ denotes the expectation with respect to the law of the random walk $(X_t)$ with initial point $X_0=x$ in a fixed realization of the digraph $G$. Let $\g_0,\g_1$ be as in Definition \ref{def:types}

\begin{theorem}\label{th:covertime}
There exists a constant $C>0$ such that 
\begin{equation}\label{eq:th-covertime}
\P\left(C^{-1}n\log^{\gamma_1}(n)\leq T_{\rm cov}\leq C\,n\log^{\gamma_0}(n) \right)=1-o(1).
\end{equation}
%
%
\end{theorem}
\begin{remark}\label{rem:optimal2}
For sequences $\bd^\pm$ such that $(\d_-,\D_+)\in\cL$ one has $\g_0=\g_1=\g$ and Theorem \ref{th:covertime} implies
\begin{equation}\label{eq:tcovopt}
T_{\rm cov}\asymp n\log^{\gamma}(n)\,,\qquad {\rm w.h.p.}
\end{equation}
As in Remark \ref{rem:optimal}, if $(\d_-,\D_+)\notin\cL$, then   Theorem \ref{th:covertime} can be strengthened by replacing $\g_0$ with the constant $\g_0'$ defined in \eqref{eq:g0prime}.  
\end{remark}

Finally, we observe that when the sequences $\bd^\pm$ are Eulerian, that is $d_x^+=d_x^-$ for all $x\in[n]$, then the estimates in Theorem \ref{th:covertime} can be refined considerably, and one obtains results that are at the same level of precision of those already established 
in the case of random undirected graphs \cite{CF3}.  

\begin{theorem}\label{th:covertime-eulerian}
Suppose $d_x^-=d_x^+=d_x$ for every $x\in[n]$. Call $\cV_{d}$ the set of vertices of degree $d$, and write $\bar d=m/n$
for the average degree. Assume
\begin{equation}
|\cV_{d}|= n^{\alpha_d+o(1)}
\end{equation}
for some constants $\alpha_d\in[0,1]$, for each type $d$.
Then, 
\begin{equation}\label{eq:th-cover-eulerian}
T_{\rm cov}=(\b+o(1)) \,n\log n\,,\qquad {\rm w.h.p.}
\end{equation}
where $\beta:=\bar d\,\max_{d}\frac{\a_d}{d}$.
\end{theorem}
In particular, if all present  types have linear size then $\a_d\in\{0,1\}$ for all $d$ and \eqref{eq:th-cover-eulerian}
holds with  $\b=\bar d/\d$, where $\d$ is the minimum degree. In any case it is not difficult to see that $\b\geq 1$, since $\bar d$ is determined only by types with linear size. For some general bounds on cover times of Eulerian graphs we refer to \cite{boczkowski2018sensitivity}.

The rest of the paper is divided into three sections. The first is a collection of preliminary structural facts about the directed configuration model. It also includes the proof of Theorem \ref{th:diameter}. The second section is the core of the paper. There we establish Theorem \ref{th:pimin} and  
 Theorem \ref{th:pimax}. The last section contains the proof of the cover time results Theorem \ref{th:covertime} and Theorem \ref{th:covertime-eulerian}.


\section{Neighborhoods and diameter}\label{sec:structure}
We start by recalling some simple facts about the directed configuration model. 

\subsection{Sequential generation}
Each vertex $x$ has $d^-_x$ labeled heads and $d^+_x$ labeled tails, and we call $E_x^-$ and $E_x^+$ the sets of heads and tails at $x$ respectively. The uniform bijection $\omega$ between heads $E^-=\cup_{x\in[n]}E_x^-$ and tails $E^+=\cup_{x\in[n]}E_x^+$, viewed as a matching, can be sampled by iterating the  following steps until there are no unmatched heads left:
\begin{enumerate}[1)]
\item pick an unmatched head $f\in E^-$ according to some priority rule;
\item pick an unmatched tail $e\in E^+$ uniformly at random;
\item match $f$ with $e$, i.e.\ set $\omega(f)=e$, and call $ef$ the resulting edge.
\end{enumerate}
This gives the desired uniform distribution over matchings $\omega :E^-\mapsto E^+$ regardless of the priority rule chosen at step 1. The digraph $G$ is obtained by adding a directed edge $(x,y)$ whenever $f\in E_y^-$ and $e\in E_x^+$ in step 3 above. 

\subsection{In-neighborhoods and out-neighborhoods}
We will use the notation 
\begin{equation}\label{eq:dD}
\d=\min\{\d_-,\d_+\}\,,\qquad \D=\max\{\D_-,\D_+\}.
\end{equation}
For any $h\in\bbN$, the $h$-in-neighborhood of a vertex $y$, denoted $\cB^-_{h}(y)$, is the digraph defined as the union of all directed paths  of length $\ell\leq h$ in $G$ which terminate at vertex $y$. In the sequel a path is always understood as a sequence of directed edges $(e_1f_1,\dots,e_kf_k)$ such that $v_{f_i}=v_{e_{i+1}}$ for all $i=1,\dots,k-1$, and we use the notation $v_e$ (resp.\ $v_f$) for the vertex $x$ such that $e\in E_x^+$ (resp.\ $f\in E_x^-$).

To generate the random variable $\cB^-_{h}(y)$, we use the following breadth-first procedure. Start at vertex $y$ and run the sequence of steps described above, by giving priority to those unmatched heads which have minimal distance to vertex $y$, until this minimal distance exceeds $h$, at which point the process stops. 
Similarly, for any $h\in\bbN$, the $h$-out-neighborhood of a vertex $x$, denoted $\cB^+_{h}(x)$ is defined as the subgraph induced by the set of directed paths of length $\ell\leq h$ which start at vertex $x$. To generate the random variable $\cB^+_{h}(x)$, we use the same breadth-first procedure described above except that we invert the role of heads and tails. With slight abuse of notation we sometimes write $\cB^\pm_{h}(x)$ for the vertex set of $\cB^\pm_{h}(x)$. We also warn the reader that to simplify the notation we often avoid taking explicitly the integer part of the various parameters entering our proofs. In particular, whenever we write $\cB_h^\pm(x)$ it is always understood that $h\in\bbN$.  

During 
the generation process of the in-neighborhood, say that a {\em collision} occurs whenever a tail gets chosen, whose end-point $x$ was already exposed, in the sense that some tail in $E^+_x$ or head in $E^-_x$ had already been matched. 
Since less than $2k$ vertices are exposed when the $k^{\textrm{th}}$ tail gets matched, less than $2\Delta k$ of the $m-k+1$ possible choices can result in a collision. Thus, the conditional chance that the $k^{\textrm{th}}$ step causes a collision, given the past, is less than
$p_k=\frac{2\Delta k}{m-k+1}$. It follows that the 
number $Z_k$ of collisions caused by the first $k$ arcs is stochastically dominated by the binomial random variable  Bin$\left(k,p_k\right)$. In particular, 
\begin{equation}\label{eq:bin1}
\P\left(Z_k\geq \ell\right) \leq  \frac{k^\ell p_k^\ell}{\ell!}\,,\qquad \ell\in\bbN.
\end{equation}
%
Notice that as long as  no collision occurs, the resulting digraph is a directed tree.  
The same applies to out-neighborhoods simply by inverting the role of heads and tails.

For any digraph $G$, define the tree excess of $G$ as 
$$\tx(G)=1+|E|-|V|,$$  
where $E$ is the set of directed edges and $V$ is the set of vertices of $G$. 
In particular, $\tx(\cB^\pm_{h}(x))=0$ iff $\cB^\pm_{h}(x)$ is a directed tree, and $\tx(\cB^\pm_{h}(x))\leq 1$ iff there is at most one collision during the generation of the neighborhood $\cB^\pm_{h}(x)$. Define the events
\begin{equation}\label{eq:eve}
\cG_x(h)=\left\{\tx(\cB_h^-(x))\leq 1 \;\text{and}\;\tx(\cB_h^+(x))\le 1\right\}\,,\qquad \cG(h)=\cap_{x\in[n]}\cG_x(h).
\end{equation}
Set also
\begin{equation}\label{eq:hslash}
\hslash= \frac{1}{5}\log_\Delta(n)\,.
\end{equation}
\begin{proposition}\label{pr:tx<=1}
There exists $\chi>0$ such that   $\P\(\cG_x(\hslash)\)
=1-O(n^{-1-\chi})$ for any $x\in[n]$. In particular,
\begin{equation}\label{eq:tx1}
\P\(\cG(\hslash)\)
=1-O(n^{-\chi}).
\end{equation}
\end{proposition}
\begin{proof}
During the generation of $\cB_h^-(x)$ one creates at most $\D^h$ edges. It follows from \eqref{eq:bin1} with $\ell=2$ that the probability of the complement of $\cG_x(\hslash)$ is $O(n^{-1-\chi})$ for all $x\in[n]$ for some absolute constant $\chi>0$:
\begin{equation}\label{eq:tx1t}
\P\(\cG_x(\hslash)\)
=1-O(n^{-1-\chi}).
\end{equation} 
The conclusion follows from the union bound.  
\end{proof}

We will need to control the size of the boundary of our neighborhoods. To this end, we introduce the notation $\partial\cB_t^-(y)$ for the set of vertices $x\in[n]$ such that $d(x,y)=t$. 
Similarly, $\partial\cB_t^+(x)$ is the set of vertices $y\in[n]$ such that $d(x,y)=t$. 
Clearly, $|\partial\cB_t^\pm(y)|\leq \D^{h}$ for any $y\in[n]$ and $h\in\bbN$.
\begin{lemma}
\label{lem:easyBh}
There exists $\chi>0$ such that for all $y\in[n]$,
 \begin{equation}\label{eq:easylow}
\P\(|\partial\cB_h^\pm(y)|\geq \tfrac12{\d_\pm^h},\:\forall h\in[1,\hslash]\)=1-O(n^{-1-\chi}).
\end{equation}
\end{lemma}
\begin{proof} 
By symmetry we may restrict to the case of in-neighborhoods. By \eqref{eq:tx1t}
it is sufficient to show that $|\partial\cB_h^\pm(y)|\geq \tfrac12{\d_\pm^h}$, for all $h\in[1,\hslash]$, if $\cG_y(\hslash)$ holds.
If the tree excess of the $h$-in-neighborhood $\cB_h^-(y)$ is at most $1$ then there is at most one collision in the generation of $\cB_h^-(y)$. This collision can be of two types: 

\begin{enumerate}
	\item\label{it:collision1} there exists some $1\leq t\leq h$ and a $v\in\partial\cB^-_{t}(y)$ s.t.\ $v$ has two out-neighbors $w,w'\in\partial\cB^-_{t-1}(y)$;
	
	\item\label{it:collision2} there exists some $0\leq t\leq h$ and a  $v\in\partial\cB^-_{t}(y)$ s.t. $v$ has an in-neighbor $w$ in $\cB^-_{t}(y)$.
\end{enumerate}
The first case can be further divided into two cases: a) $w=w'$, and b) $w\neq w'$; see Figure \ref{fig:1}.

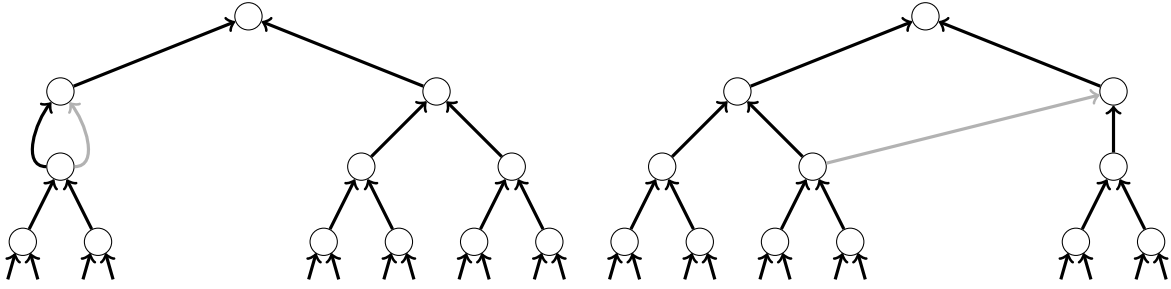
\begin{figure}
\centering

\begin{tikzpicture}
\draw (0-5,0) node(0)[circle,draw] {};
\draw (-2.5-5,-1) node(1)[circle,draw] {};
\draw (2.5-5,-1) node(2)[circle,draw] {};
\draw (-2.5-5,-2) node(3)[circle,draw]{};
\draw (1.5-5,-2) node(5)[circle,draw]{};
\draw (3.5-5,-2) node(6)[circle,draw]{};

\draw (-3-5,-3) node(7)[circle,draw]{};
\draw (-2-5,-3) node(8)[circle,draw]{};

\draw (4-5,-3) node(14)[circle,draw]{};
\draw (3-5,-3) node(13)[circle,draw]{};
\draw (2-5,-3) node(12)[circle,draw]{};
\draw (1-5,-3) node(11)[circle,draw]{};

\draw [->,line width=1.2pt] (1) -- (0) ;
\draw [->,line width=1.2pt] (2) -- (0) ;

\draw [->,line width=1.2pt] (3) to [out=180,in=230] (1);

\draw [->,line width=1.2pt,light-gray] (3) to [out=0,in=310] (1);

\draw [->,line width=1.2pt] (5) -- (2) ;
\draw [->,line width=1.2pt] (6) -- (2) ;

\draw [->,line width=1.2pt] (7) -- (3) ;
\draw [->,line width=1.2pt] (8) -- (3) ;

\draw [->,line width=1.2pt] (11) -- (5) ;
\draw [->,line width=1.2pt] (12) -- (5) ;

\draw [->,line width=1.2pt] (13) -- (6) ;
\draw [->,line width=1.2pt] (14) -- (6) ;

\draw[->,line width=1.2pt] (-2.8-5,-3.5)--(-2.9-5,-3.15);
\draw[->,line width=1.2pt] (-3.2-5,-3.5)--(-3.1-5,-3.15);

\draw[->,line width=1.2pt] (-1.8-5,-3.5)--(-1.9-5,-3.15);
\draw[->,line width=1.2pt] (-2.2-5,-3.5)--(-2.1-5,-3.15);

\draw[->,line width=1.2pt] (0.8-5,-3.5)--(0.9-5,-3.15);
\draw[->,line width=1.2pt] (1.2-5,-3.5)--(1.1-5,-3.15);

\draw[->,line width=1.2pt] (1.8-5,-3.5)--(1.9-5,-3.15);
\draw[->,line width=1.2pt] (2.2-5,-3.5)--(2.1-5,-3.15);

\draw[->,line width=1.2pt] (2.8-5,-3.5)--(2.9-5,-3.15);
\draw[->,line width=1.2pt] (3.2-5,-3.5)--(3.1-5,-3.15);

\draw[->,line width=1.2pt] (3.8-5,-3.5)--(3.9-5,-3.15);
\draw[->,line width=1.2pt] (4.2-5,-3.5)--(4.1-5,-3.15);

\draw (0+4,0) node(0a)[circle,draw] {};
\draw (-2.5+4,-1) node(1a)[circle,draw] {};
\draw (2.5+4,-1) node(2a)[circle,draw] {};
\draw (-3.5+4,-2) node(3a)[circle,draw]{};
\draw (-1.5+4,-2) node(4a)[circle,draw]{};

\draw (2.5+4,-2) node(6a)[circle,draw]{};

\draw (-4+4,-3) node(7a)[circle,draw]{};
\draw (-3+4,-3) node(8a)[circle,draw]{};
\draw (-2+4,-3) node(9a)[circle,draw]{};
\draw (-1+4,-3) node(10a)[circle,draw]{};

\draw (3+4,-3) node(14a)[circle,draw]{};
\draw (2+4,-3) node(13a)[circle,draw]{};

\draw [->,line width=1.2pt] (1a) -- (0a) ;
\draw [->,line width=1.2pt] (2a) -- (0a) ;

\draw [->,line width=1.2pt] (3a) -- (1a) ;
\draw [->,line width=1.2pt] (4a) -- (1a) ;

\draw [->,line width=1.2pt,light-gray] (4a) -- (2a) ;
\draw [->,line width=1.2pt] (6a) -- (2a) ;

\draw [->,line width=1.2pt] (7a) -- (3a) ;
\draw [->,line width=1.2pt] (8a) -- (3a) ;

\draw [->,line width=1.2pt] (9a) -- (4a) ;
\draw [->,line width=1.2pt] (10a) -- (4a) ;

\draw [->,line width=1.2pt] (13a) -- (6a) ;
\draw [->,line width=1.2pt] (14a) -- (6a) ;

\draw[->,line width=1.2pt] (-3.8+4,-3.5)--(-3.9+4,-3.15);
\draw[->,line width=1.2pt] (-4.2+4,-3.5)--(-4.1+4,-3.15);

\draw[->,line width=1.2pt] (-2.8+4,-3.5)--(-2.9+4,-3.15);
\draw[->,line width=1.2pt] (-3.2+4,-3.5)--(-3.1+4,-3.15);

\draw[->,line width=1.2pt] (-1.8+4,-3.5)--(-1.9+4,-3.15);
\draw[->,line width=1.2pt] (-2.2+4,-3.5)--(-2.1+4,-3.15);

\draw[->,line width=1.2pt] (-0.8+4,-3.5)--(-0.9+4,-3.15);
\draw[->,line width=1.2pt] (-1.2+4,-3.5)--(-1.1+4,-3.15);

\draw[->,line width=1.2pt] (1.8+4,-3.5)--(1.9+4,-3.15);
\draw[->,line width=1.2pt] (2.2+4,-3.5)--(2.1+4,-3.15);

\draw[->,line width=1.2pt] (2.8+4,-3.5)--(2.9+4,-3.15);
\draw[->,line width=1.2pt] (3.2+4,-3.5)--(3.1+4,-3.15);

\end{tikzpicture}

\caption{The light-coloured arrow represents a collision of type (1a) (left) and  a collision of type (1b) (right). 
}\label{fig:1}

\end{figure}

In case 1a) we note that  the $(h-t)$-in-neighborhood of $v$ must be a directed tree with at least $\d_-^{h-t}$ elements on its boundary and with no intersection with the $(h-t)$-in-neighborhoods of other $v'\in\partial\cB^-_{t}(y)$. Moreover, $\cB^-_{t-1}(y)$ must be a directed tree with  $|\partial\cB^-_{t-1}(y)|\geq \d_-^{t-1}$, and all elements of $ \partial\cB^-_{t-1}(y)$ except one have disjoint $(h-t+1)$-in-neighborhoods with $\d_-^{h-t+1}$ elements on their boundary. Therefore 
$$|\partial\cB^-_{h}(y)|\geq (\d_-^{t-1}-1) \d_-^{h-t+1} + (\d_--1)\d_-^{h-t}\geq \frac12\d_-^h.$$
In case 1b) one has that $t\geq 2$, $\cB^-_{t-1}(y)$ is a directed tree with  $|\partial\cB^-_{t-1}(y)|\geq \d_-^{t-1}$, and for all $z\in \partial\cB^-_{t}(y)$,  the $(h-t)$-in-neighborhoods of $z$ are disjoint directed trees with at least $\d_-^{h-t}$ elements on their boundary. Since $|\partial\cB^-_{t}(y)|\geq \d_-^{t}-1$ it follows that 
$$
|\partial\cB^-_{h}(y)|\geq (\d_-^t-1)\d_-^{h-t}\geq \frac12\d_-^h.
$$
Collisions of type 2 can be further divided into two types: a) $w\in\partial\cB^-_{s}(y)$ with $s< t$  and there is no path from $v$ to $w$ of length $t-s$, or $w\in\partial\cB^-_{t}(y)$ and $w\neq v$, and b) $w\in\partial\cB^-_{s}(y)$ with $s< t$  and there is a path from $v$ to $w$ of length $t-s$, or $w=v$. 
Note that in contrast with collisions of type 2a), a  collision of type 2b) creates a directed cycle within $\cB^-_{t}(y)$; see Figure \ref{fig:2a} and Figure \ref{fig:2b}.

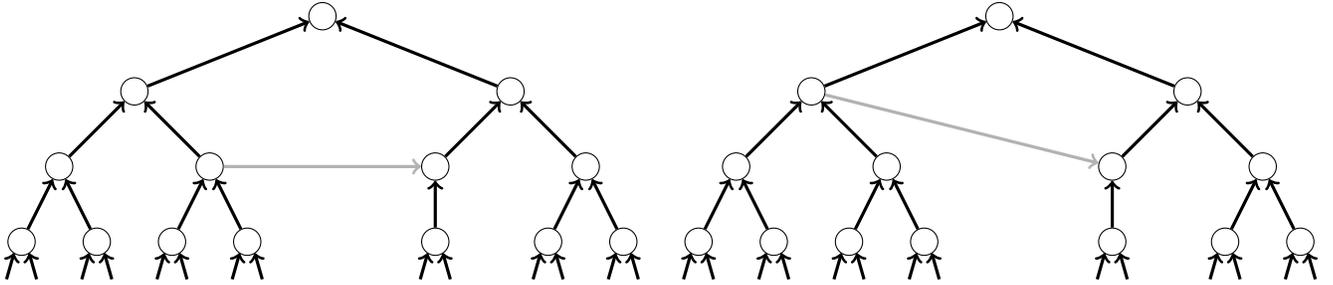
\begin{figure}
\centering
\begin{tikzpicture}
\draw (0-4,0) node(0a)[circle,draw] {};
\draw (-2.5-4,-1) node(1a)[circle,draw] {};
\draw (2.5-4,-1) node(2a)[circle,draw] {};
\draw (-3.5-4,-2) node(3a)[circle,draw]{};
\draw (-1.5-4,-2) node(4a)[circle,draw]{};
\draw (1.5-4,-2) node(5a)[circle,draw]{};
\draw (3.5-4,-2) node(6a)[circle,draw]{};

\draw (-4-4,-3) node(7a)[circle,draw]{};
\draw (-3-4,-3) node(8a)[circle,draw]{};
\draw (-2-4,-3) node(9a)[circle,draw]{};
\draw (-1-4,-3) node(10a)[circle,draw]{};

\draw (4-4,-3) node(14a)[circle,draw]{};
\draw (3-4,-3) node(13a)[circle,draw]{};
\draw (1.5-4,-3) node(11a)[circle,draw]{};

\draw [->,line width=1.2pt] (1a) -- (0a) ;
\draw [->,line width=1.2pt] (2a) -- (0a) ;

\draw [->,line width=1.2pt] (3a) -- (1a) ;
\draw [->,line width=1.2pt] (4a) -- (1a) ;
\draw [->,line width=1.2pt,light-gray] (4a) -- (5a) ;

\draw [->,line width=1.2pt] (5a) -- (2a) ;
\draw [->,line width=1.2pt] (6a) -- (2a) ;

\draw [->,line width=1.2pt] (7a) -- (3a) ;
\draw [->,line width=1.2pt] (8a) -- (3a) ;

\draw [->,line width=1.2pt] (9a) -- (4a) ;
\draw [->,line width=1.2pt] (10a) -- (4a) ;

\draw [->,line width=1.2pt] (11a) -- (5a) ;

\draw [->,line width=1.2pt] (13a) -- (6a) ;
\draw [->,line width=1.2pt] (14a) -- (6a) ;

\draw[->,line width=1.2pt] (-3.8-4,-3.5)--(-3.9-4,-3.15);
\draw[->,line width=1.2pt] (-4.2-4,-3.5)--(-4.1-4,-3.15);

\draw[->,line width=1.2pt] (-2.8-4,-3.5)--(-2.9-4,-3.15);
\draw[->,line width=1.2pt] (-3.2-4,-3.5)--(-3.1-4,-3.15);

\draw[->,line width=1.2pt] (-1.8-4,-3.5)--(-1.9-4,-3.15);
\draw[->,line width=1.2pt] (-2.2-4,-3.5)--(-2.1-4,-3.15);

\draw[->,line width=1.2pt] (-0.8-4,-3.5)--(-0.9-4,-3.15);
\draw[->,line width=1.2pt] (-1.2-4,-3.5)--(-1.1-4,-3.15);

\draw[->,line width=1.2pt] (1.3-4,-3.5)--(1.4-4,-3.15);
\draw[->,line width=1.2pt] (1.7-4,-3.5)--(1.6-4,-3.15);

\draw[->,line width=1.2pt] (2.8-4,-3.5)--(2.9-4,-3.15);
\draw[->,line width=1.2pt] (3.2-4,-3.5)--(3.1-4,-3.15);

\draw[->,line width=1.2pt] (3.8-4,-3.5)--(3.9-4,-3.15);
\draw[->,line width=1.2pt] (4.2-4,-3.5)--(4.1-4,-3.15);

\draw (0+5,0) node(0)[circle,draw] {};
\draw (-2.5+5,-1) node(1)[circle,draw] {};
\draw (2.5+5,-1) node(2)[circle,draw] {};
\draw (-3.5+5,-2) node(3)[circle,draw]{};
\draw (-1.5+5,-2) node(4)[circle,draw]{};
\draw (1.5+5,-2) node(5)[circle,draw]{};
\draw (3.5+5,-2) node(6)[circle,draw]{};

\draw (-4+5,-3) node(7)[circle,draw]{};
\draw (-3+5,-3) node(8)[circle,draw]{};
\draw (-2+5,-3) node(9)[circle,draw]{};
\draw (-1+5,-3) node(10)[circle,draw]{};

\draw (4+5,-3) node(14)[circle,draw]{};
\draw (3+5,-3) node(13)[circle,draw]{};
\draw (1.5+5,-3) node(11)[circle,draw]{};

\draw [->,line width=1.2pt] (1) -- (0) ;
\draw [->,line width=1.2pt] (2) -- (0) ;

\draw [->,line width=1.2pt] (3) -- (1) ;
\draw [->,line width=1.2pt] (4) -- (1) ;
\draw [->,line width=1.2pt,light-gray] (1) -- (5) ;

\draw [->,line width=1.2pt] (5) -- (2) ;
\draw [->,line width=1.2pt] (6) -- (2) ;

\draw [->,line width=1.2pt] (7) -- (3) ;
\draw [->,line width=1.2pt] (8) -- (3) ;

\draw [->,line width=1.2pt] (9) -- (4) ;
\draw [->,line width=1.2pt] (10) -- (4) ;

\draw [->,line width=1.2pt] (11) -- (5) ;

\draw [->,line width=1.2pt] (13) -- (6) ;
\draw [->,line width=1.2pt] (14) -- (6) ;

\draw[->,line width=1.2pt] (-3.8+5,-3.5)--(-3.9+5,-3.15);
\draw[->,line width=1.2pt] (-4.2+5,-3.5)--(-4.1+5,-3.15);

\draw[->,line width=1.2pt] (-2.8+5,-3.5)--(-2.9+5,-3.15);
\draw[->,line width=1.2pt] (-3.2+5,-3.5)--(-3.1+5,-3.15);

\draw[->,line width=1.2pt] (-1.8+5,-3.5)--(-1.9+5,-3.15);
\draw[->,line width=1.2pt] (-2.2+5,-3.5)--(-2.1+5,-3.15);

\draw[->,line width=1.2pt] (-0.8+5,-3.5)--(-0.9+5,-3.15);
\draw[->,line width=1.2pt] (-1.2+5,-3.5)--(-1.1+5,-3.15);

\draw[->,line width=1.2pt] (1.3+5,-3.5)--(1.4+5,-3.15);
\draw[->,line width=1.2pt] (1.7+5,-3.5)--(1.6+5,-3.15);

\draw[->,line width=1.2pt] (2.8+5,-3.5)--(2.9+5,-3.15);
\draw[->,line width=1.2pt] (3.2+5,-3.5)--(3.1+5,-3.15);

\draw[->,line width=1.2pt] (3.8+5,-3.5)--(3.9+5,-3.15);
\draw[->,line width=1.2pt] (4.2+5,-3.5)--(4.1+5,-3.15);
\end{tikzpicture}
\caption{Two examples of collision of type $(2a)$. 
}\label{fig:2a}
\end{figure}

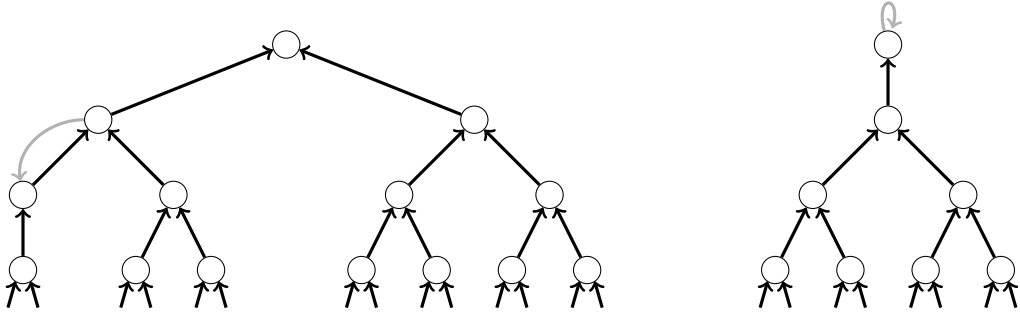
\begin{figure}
\centering
\begin{tikzpicture}
\draw (0-4,0) node(0)[circle,draw] {};
\draw (-2.5-4,-1) node(1)[circle,draw] {};
\draw (2.5-4,-1) node(2)[circle,draw] {};
\draw (-3.5-4,-2) node(3)[circle,draw]{};
\draw (-1.5-4,-2) node(4)[circle,draw]{};
\draw (1.5-4,-2) node(5)[circle,draw]{};
\draw (3.5-4,-2) node(6)[circle,draw]{};

\draw (-3.5-4,-3) node(7)[circle,draw]{};
\draw (-2-4,-3) node(9)[circle,draw]{};
\draw (-1-4,-3) node(10)[circle,draw]{};

\draw (4-4,-3) node(14)[circle,draw]{};
\draw (3-4,-3) node(13)[circle,draw]{};
\draw (2-4,-3) node(12)[circle,draw]{};
\draw (1-4,-3) node(11)[circle,draw]{};

\draw [->,line width=1.2pt] (1) -- (0) ;
\draw [->,line width=1.2pt] (2) -- (0) ;

\draw [->,line width=1.2pt] (3) -- (1) ;
\draw [->,line width=1.2pt] (4) -- (1) ;

\draw [->,line width=1.2pt] (5) -- (2) ;
\draw [->,line width=1.2pt] (6) -- (2) ;

\draw [->,line width=1.2pt] (7) -- (3) ;

\draw [->,line width=1.2pt] (9) -- (4) ;
\draw [->,line width=1.2pt] (10) -- (4) ;

\draw [->,line width=1.2pt] (11) -- (5) ;
\draw [->,line width=1.2pt] (12) -- (5) ;

\draw [->,line width=1.2pt] (13) -- (6) ;
\draw [->,line width=1.2pt] (14) -- (6) ;


\draw [->,line width=1.2pt,light-gray] (1) to [out=180,in=100] (3);

\draw[->,line width=1.2pt] (-3.3-4,-3.5)--(-3.4-4,-3.15);
\draw[->,line width=1.2pt] (-3.7-4,-3.5)--(-3.6-4,-3.15);

\draw[->,line width=1.2pt] (-1.8-4,-3.5)--(-1.9-4,-3.15);
\draw[->,line width=1.2pt] (-2.2-4,-3.5)--(-2.1-4,-3.15);

\draw[->,line width=1.2pt] (-0.8-4,-3.5)--(-0.9-4,-3.15);
\draw[->,line width=1.2pt] (-1.2-4,-3.5)--(-1.1-4,-3.15);

\draw[->,line width=1.2pt] (0.8-4,-3.5)--(0.9-4,-3.15);
\draw[->,line width=1.2pt] (1.2-4,-3.5)--(1.1-4,-3.15);

\draw[->,line width=1.2pt] (1.8-4,-3.5)--(1.9-4,-3.15);
\draw[->,line width=1.2pt] (2.2-4,-3.5)--(2.1-4,-3.15);

\draw[->,line width=1.2pt] (2.8-4,-3.5)--(2.9-4,-3.15);
\draw[->,line width=1.2pt] (3.2-4,-3.5)--(3.1-4,-3.15);

\draw[->,line width=1.2pt] (3.8-4,-3.5)--(3.9-4,-3.15);
\draw[->,line width=1.2pt] (4.2-4,-3.5)--(4.1-4,-3.15);

\draw (0+4,0) node(0a)[circle,draw] {};

\draw (0+4,-1) node(2a)[circle,draw] {};

\draw (-1+4,-2) node(5a)[circle,draw]{};
\draw (1+4,-2) node(6a)[circle,draw]{};

\draw (1.5+4,-3) node(14a)[circle,draw]{};
\draw (.5+4,-3) node(13a)[circle,draw]{};
\draw (-0.5+4,-3) node(12a)[circle,draw]{};
\draw (-1.5+4,-3) node(11a)[circle,draw]{};

\draw [->,line width=1.2pt] (2a) -- (0a) ;

\draw [->,line width=1.2pt] (5a) -- (2a) ;
\draw [->,line width=1.2pt] (6a) -- (2a) ;

\draw [->,line width=1.2pt] (11a) -- (5a) ;
\draw [->,line width=1.2pt] (12a) -- (5a) ;

\draw [->,line width=1.2pt] (13a) -- (6a) ;
\draw [->,line width=1.2pt] (14a) -- (6a) ;

\draw [->,line width=1.2pt,light-gray] (0a) edge[loop above] (0a);

\draw[->,line width=1.2pt] (-1.3+4,-3.5)--(-1.4+4,-3.15);
\draw[->,line width=1.2pt] (-1.7+4,-3.5)--(-1.6+4,-3.15);

\draw[->,line width=1.2pt] (-0.3+4,-3.5)--(-0.4+4,-3.15);
\draw[->,line width=1.2pt] (-0.7+4,-3.5)--(-0.6+4,-3.15);

\draw[->,line width=1.2pt] (0.7+4,-3.5)--(0.6+4,-3.15);
\draw[->,line width=1.2pt] (0.3+4,-3.5)--(0.4+4,-3.15);

\draw[->,line width=1.2pt] (1.7+4,-3.5)--(1.6+4,-3.15);
\draw[->,line width=1.2pt] (1.3+4,-3.5)--(1.4+4,-3.15);
\end{tikzpicture}
\caption{Two examples of collision of type (2b). 
}\label{fig:2b}
\end{figure}
We remark that in either case 2a) or case 2b),  $\partial\cB^-_{t}(y)$ has at least $\d_-^{t}$ elements, and the vertex $v\in\partial\cB^-_{t}(y)$ has at least $\d_--1$ in-neighbors whose $(h-t-1)$-in-neighborhoods are disjoint directed trees. 
All other $v'\in\partial\cB^-_{t}(y)$ have $(h-t)$-in-neighborhoods that are disjoint directed trees. Therefore, in case
2): $$
|\partial\cB^-_{h}(y)|\geq (\d_-^{t}-1) \d_-^{h-t} + (\d_--1)\d_-^{h-t-1}\geq \frac12\d_-^h.
$$

\end{proof}

%
%
%
%
%
%
 

We shall need a more precise control of the size of $\partial\cB_h^\pm(y)$, and for values of $h$ that are larger than $\hslash$.  
Recall the definition \eqref{eq:nu} of the parameter $\nu$. We use the following notation in the sequel:
\begin{equation}\label{eq:hs}
\ell_0=4\log_\delta\log(n),\qquad h_\eta=(1-\eta)\log_\nu(n).
\end{equation}

\begin{lemma}\label{le:lemmasize} 	
For every $\eta\in(0,1)$, 
there exist 
constants $c_1,c_2>0,\chi>0$ such that for all $y\in[n]$,

\begin{equation}\label{eq:lemmasize-in}
\P\(\nu^{h}\log^{-c_1}(n)\le|\partial\cB_h^\pm(y)|\le \nu^{h}\log^{c_2}(n)\,,\;\forall h\in\[\ell_0, h_\eta\] \)=1-O(n^{-1-\chi}).
\end{equation}

\end{lemma}
\begin{proof}
We run the proof for the in-neighborhood only since the case of the out-neighborhood is obtained in the same way. We generate $\cB^-_h(y)$, $h \in\[\ell_0, h_\eta\]$ sequentially in a breadth first fashion. After the depth $j$ neighborhood $\cB^-_j(y)$ has been sampled, we call $ \cF_{j}$ the set of all heads attached to vertices in $\partial \cB^-_j(y)$. Set $$u= \log^{-7/8}(n).$$
 For any $ h\geq \ell_0$ define
 \begin{equation}
\kappa_h:=[\nu(1-u)]^{h-\ell_0}\log^{7/2}(n),\qquad \widehat{\kappa}_h:=[\nu(1+u)]^{h-\ell_0}\Delta^{\ell_0}.
\end{equation}
We are going to prove 
 \begin{equation}\label{eq:lemmasize-in-t}
\P\(\kappa_h\le|\cF_{h}|\le \widehat{\kappa}_h\,,\;\forall h\in\[\ell_0, h_\eta\] \)=1-O(n^{-1-\chi}).
\end{equation}
Notice that, choosing suitable constants $c_1,c_2>0$,  \eqref{eq:lemmasize-in} is a consequence of
\eqref{eq:lemmasize-in-t}. 

Consider the events
\begin{equation}
A_j=\left\{ |\cF_i|\in\[\kappa_i,\widehat{\kappa}_i \]\,,\; \forall i\in[\ell_0,j] \right\}.
\end{equation}
Thus, we need to prove $\P(A_h)=1-O(n^{-1-\chi})$, for $h=h_\eta$. From Lemma \ref{lem:easyBh}
and the choice of $\ell_0$, it follows that
\begin{equation}\label{24}
\P(A_{\ell_0})=1-O(n^{-1-\chi}).
\end{equation}
For $h>\ell_0$ we write
\begin{equation}\label{244}
\P(A_h)= \P(A_{\ell_0})\prod_{j=\ell_0+1}^{h}\P(A_j|A_{j-1}).
\end{equation}
To estimate $\P(A_j|A_{j-1})$, note that $A_{j-1}$ depends only on the in-neighborhood $\cB^-_{j-1}(y)$, so if $\sigma_{j-1}$ denotes a realization of $\cB^-_{j-1}(y)$ with a slight abuse of notation we write $\sigma_{j-1}\in A_{j-1}$ if $A_{j-1}$ occurs for this given $\sigma_{j-1}$. Then
\begin{equation}\label{245}
\P(A_j|A_{j-1})=\frac{\sum_{\si_{j-1}}\P(\sigma_{j-1})\P(A_j|\sigma_{j-1})1_{\sigma_{j-1}\in A_{j-1}}}{\P(A_{j-1})}.
\end{equation}
Therefore, to prove a lower bound on $\P(A_j|A_{j-1})$ it is sufficient to prove a lower bound on $\P(A_j|\sigma_{j-1})$ that is uniform over all $\sigma_{j-1}\in A_{j-1}$. 

Suppose we have generated the neighborhood $\si_{j-1}$ up to depth $j-1$, for a $\si_{j-1}\in A_{j-1}$. In some arbitrary order we now generate the matchings of all heads $f\in\cF_{j-1}$. We define the random variable $X_f^{(j)}$, $f\in\cF_{j-1}$, which, for every $f$ evaluates to the in-degree $d_z^-$ of the vertex $z$ that is matched to $f$ 
if the vertex $z$ was not yet exposed, and evaluates to zero otherwise.
In this way
\begin{equation}\label{eq:sumxf}
|\cF_j|= \sum_{f\in\cF_{j-1}}X_f^{(j)}.
\end{equation}
Therefore,
%
\begin{align}\label{eq:conclu1}
&\P(A_{j}|\sigma_{j-1})=\P\(\nu(1-u)\kappa_{j-1}\le |\cF_j|\le \nu(1+u)\widehat\kappa_{j-1}\tc \sigma_{j-1}\)\\
&\quad=1-\P\Big( \sum_{f\in\cF_{j-1}}X_f^{(j)}<\nu(1-u)\kappa_{j-1}\tc \sigma_{j-1}\Big)-\P\Big(\sum_{f\in\cF_{j-1}}X_f^{(j)}> \nu(1+u)\widehat\kappa_{j-1}\tc \sigma_{j-1}\Big).
\nonumber 
\end{align}
To sample the variables $X_f^{(j)}$,
at each step we pick a tail uniformly at random among all unmatched tails and evaluate the in-degree of its end point if it is not yet exposed. Since $\si_{j-1}\in A_{j-1}$, at any such step the number of exposed vertices is at most $K=O(n^{1-\eta/2})$. 
In particular, for any $f\in\cF_{j-1}$ and any $d\in[\delta,\Delta]$, $\si_{j-1}\in A_{j-1}$:
\begin{equation*}
\P\(X_f^{(j)}= d \tc \si_{j-1}\)\geq \frac{\left[\(\sum_{k=1}^{n} d^+_{k}\ind_{d^-_{k}=d}\) -\D K\right]_+}{m}=:p(d),
\end{equation*}
where $[\cdot]_+$ denotes the positive part. 
This shows that $X_f^{(j)}$ stochastically dominates the random variable $Y^{(j)}$ and is stochastically dominated by the random variable $\widehat{Y}^{(j)}$, where $Y^{(j)}$ and $\widehat{Y}^{(j)}$ are defined by  \begin{gather*}
\forall d\in[\delta,\Delta],\qquad\P(Y^{(j)}=d)=\P(\widehat{Y}^{(j)}=d)= p(d)
\\
\P\(\widehat{Y}^{(j)}= \Delta+1 \)=\P\(Y^{(j)}= 0 \)=1-\sum_{d=\delta}^{\Delta}p(d) .
\end{gather*}
Notice that 
\begin{equation}\label{eq:eyj}
\E\(Y^{(j)}\)=\sum_{d=\delta}^{\Delta}dp(d)\geq \nu - \frac{\D^2K}m = \nu - O(n^{-\eta/2}).
\end{equation}
Similarly,
\begin{equation}\label{eq:eyj2}
\E\(\widehat{Y}^{(j)}\)\leq \nu + \frac{\D^2K}m = \nu + O(n^{-\eta/2}).
\end{equation}
Moreover, letting $Y_i^{(j)}$ and $\widehat{Y}_i^{(j)}$ denote i.i.d.\ copies of the random variables $Y^{(j)}$ and $\widehat{Y}^{(j)}$ respectively, since $\si_{j-1}\in A_{j-1}$, the sum in \eqref{eq:sumxf} stochastically dominates 
$\sum_{i=1}^{\kappa_{j-1}}Y_i^{(j)}$, and is stochastically dominated by 
$\sum_{i=1}^{\widehat \kappa_{j-1}}Y_i^{(j)}$. Therefore, $\sum_{f\in\cF_{j-1}}X_f^{(j)}<\nu(1-u)\kappa_{j-1}$ implies that 
\begin{equation}\label{eq:eyj3}
\sum_{i=1}^{\kappa_{j-1}}\left[Y_i^{(j)}-\E\(Y^{(j)}\)\right]\leq -\frac12\,u\kappa_{j-1},
\end{equation}
if $n$ is large enough. Similarly, $\sum_{f\in\cF_{j-1}}X_f^{(j)}>\nu(1+u)\widehat\kappa_{j-1}$ 
implies that 
\begin{equation}\label{eq:eyj4}
\sum_{i=1}^{\widehat\kappa_{j-1}}\left[\widehat{Y}_i^{(j)}-\E\(\widehat{Y}^{(j)}\)\right]\geq \frac12\,u\widehat\kappa_{j-1}.\end{equation}
An application of Hoeffding's inequality shows that the probability of the events \eqref{eq:eyj3} and \eqref{eq:eyj4} is bounded by 
$e^{-c u^2 \kappa_{j-1}}$ and $e^{-c u^2 \widehat\kappa_{j-1}}$ respectively, for some absolute constant $c>0$.
Hence, from \eqref{eq:conclu1} we conclude that for some constant $c>0$:
\begin{align*}
\P(A_{j}|\sigma_{j-1})\geq 1- e^{-c u^2 \kappa_{j-1}} - e^{-c u^2 \widehat\kappa_{j-1}}.
\end{align*}
Therefore, using $u^2 \widehat\kappa_{j-1}\geq  u^2\kappa_{j-1}\geq u^2\kappa_{0} \geq \log^{3/2}(n)$,
\begin{equation}\label{eq:ahe1}
\P(A_{j}|\sigma_{j-1})=1-O(n^{-3}),
\end{equation}
uniformly in $j\in[\ell_0,h_\eta]$ and $\sigma_{j-1}\in A_{j-1}$. By \eqref{245} the same bound applies to $\P(A_{j}|A_{j-1})$ and 
going back to \eqref{244}, for $h=h_\eta$ we have obtained 
\begin{equation*}\label{eq:ahe}
\P(A_{h})=1-O(n^{-1-\chi}).
\end{equation*}
\end{proof}
We shall also need the following refinement of Lemma \ref{le:lemmasize}.
Define the events
\begin{equation}\label{eq:evfy}
F_y^\pm=F_y^\pm(\eta,c_1,c_2)=\left\{\nu^{h}\log^{-c_1}(n)\le|\partial\cB_{h}^\pm(y)|\le \nu^{h}\log^{c_2}(n)\,,\;\forall h\in\[\ell_0, h_\eta\] \right\}.
\end{equation}
Lemma \ref{le:lemmasize} states that 
$$\P\((F_y^\pm)^c\)=O(n^{-1-\chi}).$$
Let $\cG(\hslash)$ be the event from Proposition \ref{pr:tx<=1}.
\begin{lemma}\label{le:lemmasize2} 	
For every $\eta\in(0,1)$, 
there exist 
constants $c_1,c_2>0,\chi>0$ such that for all $y\in[n]$,
\begin{equation}\label{eq:lemmasize-2in}
\P\((F_y^\pm)^c; \cG(\hslash)\)=O(n^{-2-\chi}).
\end{equation}
\end{lemma}
\begin{proof}
By symmetry we may prove the inequality for the event $F_y^-$ only. 
Consider the set $\cD^-_y$ of all possible $2$-in-neighborhoods of $y$ compatible with the event $\cG(\hslash)$, that is the set of labeled digraphs $D$ such that 
\begin{equation}\label{eq:defDy}
\P(\cB^-_2(y)=D\,;\,\cG(\hslash))>0.
\end{equation}
Then
\begin{equation}\label{eq:Dy1}
\P\((F_y^-)^c; \cG(\hslash)\)\leq \sup_{D\in\cD^-_y}\P\((F_y^\pm)^c\tc\cB^-_2(y)=D \).
\end{equation}
Thus it is sufficient to prove that 
\begin{equation}\label{eq:Dy2}
\P\((F_y^\pm)^c\tc\cB^-_2(y)=D \)=O(n^{-2-\chi}),
\end{equation}
uniformly in $D\in\cD^-_y$. 
%
To this end, we may repeat exactly the same argument as in the proof of Lemma \ref{le:lemmasize} with the difference that now we condition from the start on the event $ \cB^-_2(y)=D$ for a fixed $D\in\cD_y$. 
The key observation is that \eqref{24} can be strenghtened to $ O(n^{-2-\chi})$ if we condition on $\cB^-_2(y)=D$. That is, for some $\chi>0$, uniformly in $D\in\cD_y$, 
\begin{equation}\label{eq:cond-D}
\P\(A_{\ell_0}\tc\cB^-_2(y)=D \)=1-O(n^{-2-\chi}),
\end{equation}
To prove \eqref{eq:cond-D} notice that if the $2$-in-neighborhood of $y$ is given by $\cB^-_2(y)=D\in\cD_y$ then the set $\cF_{2}^-(y)$ has at least $4$ elements. Therefore, taking a sufficiently large constant $C$, for the event $|\cF_{\ell_0}^-(y)|\geq \d^{\ell_0}/C$ to fail it is necessary to have at least 3 collisions
in the generation of $\cB^-_t(y)$, $t\in\{3,\dots,\ell_0\}$. From the estimate \eqref{eq:bin1} the probability of this event is bounded by $p_k^3k^3$ with $k=\D^{\ell_0}$, which implies \eqref{eq:cond-D} if $\chi\in(0,1)$.
Once \eqref{eq:cond-D} is established, the rest of the proof is a repetition of the argument in \eqref{244}-\eqref{eq:ahe1}. 
\end{proof}

\subsection{Upper bound on the diameter}\label{sec:ubdiam}
The upper bound in Theorem \ref{th:diameter} is reformulated as follows.
\begin{lemma}\label{lem:ubdiam}
There exist constants $C,\chi>0$ such that if $\varepsilon_n=\frac{C\log\log(n)}{\log(n)}$, 
\begin{equation}
\P\({\rm diam}(G)>(1+\varepsilon_n)\,{\rm d}_\star\)=O(n^{-\chi}).
\end{equation}
\end{lemma}
\begin{proof}
From Proposition \ref{pr:tx<=1} we may restrict to the event $\cG(\hslash)$. From the union bound
\begin{equation}\label{eq:ubdiam}
\P\({\rm diam}(G)>(1+\varepsilon_n)\,{\rm d}_\star;\cG(\hslash)\)
\leq \sum_{x,y\in[n]}
\P\(d(x,y)>(1+\varepsilon_n){\rm d}_\star;\cG(\hslash)\).
\end{equation}
From Lemma \ref{le:lemmasize2}, for all $x,y\in[n]$
\begin{equation}\label{eq:ubdiam2}
\P\(d(x,y)>(1+\varepsilon_n){\rm d}_\star;\cG(\hslash)\)=\P\(d(x,y)>(1+\varepsilon_n){\rm d}_\star;F_x^+\cap F_y^-\) + O(n^{-2-\chi}).
\end{equation}
Fix 
$$k=\frac{1+\e_n}2\,\log_\nu n.$$

Let us use sequential generation to sample first $\cB^+_k(x)$ and then $\cB^-_{k-1}(y)$. Call $\si$ a realization of these two neighborhoods. Consider the event 
$$U_{x,y}=\{|\partial\cB_k^+(x)|\geq \nu^k\log^{-c_1}(n)\,;\;|\partial\cB_{k-1}^-(y)|\geq \nu^{k-1}\log^{-c_1}(n)\}.$$
Clearly, $ F_x^+\cap F_y^-\subset U_{x,y}$. Moreover  $U_{x,y}$ depends only on $\si$. Note also that $\{d(x,y)>(1+\varepsilon_n){\rm d}_\star\}\subset E_{x,y}$, where we define the event
\begin{equation}\label{defek}
E_{x,y}=\{\text{There is no path of length  $\leq 2k-1$ from $x$ to $y$} \}.
\end{equation}
The event $E_{x,y}$ depends only on $\si$. We say that $\si\in U_{x,y}\cap E_{x,y}$ if $\si$ is such that both $E_{x,y}$ and $U_{x,y}$ occur. 
Thus, we write
\begin{align}\label{eq:ubdiam3}
\P\(d(x,y)>(1+\varepsilon_n){\rm d}_\star;F_x^+\cap F_y^-\)&\leq 
\P\(d(x,y)>(1+\varepsilon_n){\rm d}_\star; U_{x,y}\cap E_{x,y}\)\nonumber\\
& \leq \sup_{\si\in U_{x,y}\cap E_{x,y}}\P\(d(x,y)>(1+\varepsilon_n){\rm d}_\star\tc \si\).
\end{align}
Fix a realization $\si\in U_{x,y}\cap E_{x,y}$. The event $E_{x,y}$ implies that 
all vertices on $\partial\cB_{k-1}^-(y)$ have all  their heads unmatched and the same holds for all the tails of vertices in $\partial\cB_k^+(x)$. Call $\cF_{k-1}$ the heads attached to vertices in $\partial\cB_{k-1}^-(y)$ and $\cE_{k}$ the tails attached to vertices in $\partial\cB_k^+(x)$. 
The event $d(x,y)>(1+\e_n){\rm d}_\star$ implies that there are no matchings between $\cF_{k-1}$ and $\cE_k$. The probability of this event is dominated by
$$
\(1-\frac{|\cE_{k}|}{m}\)^{|\cF_{k-1}|}\leq \(1-n^{-\frac12+\frac{\e_n}4}\)^{n^{\frac12+\frac{\e_n}4}}\leq \exp{(-n^{{\e_n}/2})}\,,
$$
if $n$ is large enough and $\e_n=C\log\log n/\log n$ with $C$ large enough. 
Therefore, uniformly in $\si\in U_{x,y}\cap E_{x,y}$, 
$$
\P\(d(x,y)>(1+\varepsilon_n){\rm d}_\star\tc \si\)\leq \exp{(-n^{{\e_n}/2})}=O(n^{-2-\chi}).
$$
Inserting this in \eqref{eq:ubdiam}-\eqref{eq:ubdiam2} completes the proof. 
\end{proof}

\subsection{Lower bound on the diameter}\label{sec:lbdiam}
We prove the following lower bound on the diameter. Note that Lemma \ref{lem:ubdiam} and Lemma \ref{lem:lbdiam} imply Theorem \ref{th:diameter}.
\begin{lemma}\label{lem:lbdiam}
There exists $C>0$ such that taking $\varepsilon_n=\frac{C\log\log(n)}{\log(n)}$, 
for any $x,y\in[n]$,
\begin{equation}
\P\(d(x,y)\leq 
(1-\e_n){\rm d}_\star\)=o(1).
\end{equation}
\end{lemma}
\begin{proof}
Define $$\ell=\frac{1-\e_n}2\,\log_\nu n. 
$$
We start by sampling the 
out-neighborhood of $x$ up to distance $\ell$. 
Consider the event 
$$J_x=\left\{| \cB^+_\ell(x)| 
\leq n^{\frac{1-\e_n}2}
\log^{c_2}(n) \right\}.$$
From Lemma \ref{le:lemmasize}, 	
$\P(J_x)=1-O(n^{-1-\chi})$ for suitable constants $c_2,\chi>0$, and therefore 
\begin{equation}\label{eq:lbd1}
\P(y\in \cB^+_\ell(x))= \P(y\in \cB^+_\ell(x); J_x)+O(n^{-1-\chi}).
\end{equation}
If $J_x$ holds, in the generation of $\cB^+_\ell(x)$ there are at most $K:=n^{\frac{1-\e_n}2}
\log^{c_2}(n) $ attempts to include $y$ in  $\cB^+_\ell(x)$, each with probability at most $d^-_y/(m-K) \leq 2\D/m$ of success, so that 
\begin{equation}\label{eq:lbd101}
\P(y\in \cB^+_\ell(x); J_x)\leq \frac{2\D}{m}\,K = O(n^{-\frac12}) .
\end{equation}
Once the out-neighborhood $\cB^+_\ell(x)$ has been generated, if $y\notin \cB^+_\ell(x)$, we generate the in-neighborhood $\cB^-_\ell(y)$. 
If $d(x,y)\le(1-\e_n){\rm d}_\star$ then there must be a collision with  $\partial\cB^+_\ell(x)$, and
\begin{equation}\label{eq:lbd2}
\P(d(x,y)\le(1-\e_n){\rm d}_\star\,;\:  y\notin \cB^+_\ell(x))= 
  \P(y\notin \cB^+_\ell(x)\,;\: \cB^{-}_\ell(y)\cap\partial\cB^{+}_\ell(x)\neq\emptyset).
\end{equation}
Consider the event
\begin{equation*}
J_y=\left\{| \cB^-_\ell(y)| <n^{\frac{1-\e_n}2}\log^{c_2}(n) \right\}.
\end{equation*}
From Lemma \ref{le:lemmasize} 	 it follows that $\P(J_y)=1-O(n^{-1-\chi})$ for suitable constants $c_2,\chi>0$. If $J_x$ and $J_y$ hold, in the generation of $\cB^{-}_\ell(y)$ there are at most $K=n^{\frac{1-\e_n}2}\log^{c_2}(n)$ attempts to collide with $\partial\cB^{+}_\ell(x)$, each of which with success probability at most $\Delta K/m$, and therefore 
\begin{equation}\label{eq:lbd3}
\P(y\notin \cB^+_\ell(x)\,;\: \cB^{-}_\ell(y)\cap\partial\cB^{+}_\ell(x)\neq\emptyset)
\leq \frac{\Delta K^2}m = O(n^{-\e_n/2})=o(1),
\end{equation}
where we take the constant $C$ in the definition of $\e_n$ sufficiently large. 
In conclusion,
\begin{align*}\label{eq:lbd4}
\P\(d(x,y)\leq (1-\e_n){\rm d}_\star\)\leq 
\P\(y\in \cB^+_\ell(x)\) + 
\P\(d(x,y)\le(1-\e_n){\rm d}_\star\,;\:  y\notin \cB^+_\ell(x)\), 
\end{align*}
and the inequalities \eqref{eq:lbd1}-\eqref{eq:lbd3} end the proof. 
\end{proof}

\section{Stationary distribution}\label{sec:stationary}
We start by recalling some key facts established in \cite{BCS1}.

\subsection{Convergence to stationarity}\label{sec:conv}
Let $P^t(x,\cdot)$ denote the distribution after $t$ steps of the random walk started at $x$.
The total variation distance between two probabilities $\mu,\nu$ on $[n]$ is defined as 
$$
\|\mu-\nu\|_\tv = \frac12\sum_{x\in[n]}|\mu(x)-\nu(x)|. 
$$ 
Let the entropy $H$ and the associated {\em entropic time} $\tent$ be defined by 
\begin{equation}\label{def:tent}
H=\sum_{x\in V}\frac{d_x^-}m\,\log d^+_x, \;\qquad \tent = \frac{\log n}H.
\end{equation}
Note that under our assumptions on $\bd^\pm$, the deterministic quantities $H,\tent$ satisfy $H=\Theta(1)$ and $\tent = \Theta(\log n)$. Theorem 1 of \cite{BCS1} states that 
\begin{equation}\label{cutoff}
\max_{x\in [n]}\left|\|P^{s\tent}(x,\cdot)-\pi\|_\tv - \vartheta(s)
\right|\overset{\P}{\longrightarrow}0\,,\qquad \forall s>0, \:s\neq 1,
\end{equation}
where $\vartheta$ denotes the step function $\vartheta(s)=1$ if $s<1$ and $\vartheta(s)=0$ if $s>1$, 
and we use the notation $\overset{\P}{\longrightarrow}$ for convergence in probability as $n\to\infty$. In words, convergence to stationarity for the random walk on the directed configuration model displays  with high probability a cutoff phenomenon, uniformly in the starting point, with mixing time given by the entropic time $\tent$. We remark that, by Jensen's inequality the mixing time $\tent=\frac{\log n}H$ is always larger than the diameter ${\rm d}_\star=\frac{\log n}{\log\nu}$ in Theorem \ref{th:diameter},
\begin{equation}\label{diamvsmixing}
H= \sum_{x=1}^n\frac{d_x^-}m\,\log d^+_x \leq \log \( \sum_{x=1}^n\frac{d_x^-}m\,d^+_x\)= \log \nu,
\end{equation}
with equality if and only if the sequence is out-regular, that is  $d_x^+\equiv d$. Thus, the analysis of convergence to stationarity  requires investigating the graph on a length scale that may well exceed the diameter. Considering all possible paths on this length scale is not practical, and we shall rely on a powerful construction of \cite{BCS1} that allows one to restrict to a subset of paths with a tree  structure, see Section \ref{sec:bcsconc} below for the details. 

\subsection{The local approximation}\label{sec:local}
A consequence of the arguments of \cite{BCS1} is that the unknown stationary distribution at a node $y$ admits an approximation in terms of the in-neighborhood of $y$ at a distance that is much smaller than the mixing time. More precisely, 
it follows from \cite[Theorem 3]{BCS1} that 
for any sequence $t_n\to\infty$
\begin{equation}\label{localapp}
\|\pi-\muin P^{t_n}\|_\tv 
\overset{\P}{\longrightarrow}0,
\end{equation}
where we use the notation $\muin$ for the in-degree distribution
\begin{equation}\label{muin}
\muin(x)=\frac{d_x^-}m,
\end{equation}
and for any probability $\mu$ on $[n]$, $\mu P^t$ is the distribution 
$$
\mu P^t(y) = \sum_{x\in[n]}\mu(x)P^t(x,y)\,,\qquad y\in[n].
$$
We refer to \cite[Lemma 1]{CQ:PageRank} for a stronger statement than \eqref{localapp} where $\muin$ is replaced by any sufficiently widespread probability on $[n]$. While these facts are
very useful to study the typical values of $\pi$, they give very poor information on its extremal values $\pimin$ and $\pimax$, and to prove Theorem \ref{th:pimin} and Theorem \ref{th:pimax} we need a stronger control of the local approximation of the stationary distribution. 

A key role in our analysis is played by the quantity $\G_h(y)$ defined as follows. Consider the set $\partial\cB^-_h(y)$ of all vertices $z\in[n]$ such that $d(z,y)=h$, 
and define
%
\begin{equation}\label{eq:gamma}
\Gamma_{h}(y):=\sum_{z\in\partial\cB^-_h(y)}\!\!d_z^- \,P^h(z,y).
\end{equation}
The definitions  \eqref{eq:gamma} and \eqref{eq:Pxy} are such that 
for any $y\in[n]$ and $h\in\bbN$
\begin{equation}\label{eq:gammavsp}
\Gamma_{h}(y)\leq m\,\muin P^h(y),
\end{equation}
where $\muin$ is defined in \eqref{muin}.
If $\cB^-_h(y)$ is a tree, then 
\eqref{eq:gammavsp} is an equality. 
In any case, $\Gamma_{h}(y)$ satisfies the following rough inequalities.
\begin{lemma}
\label{lem:gamma0}
With high probability, for all $y\in[n]$, for all $h\in[1,\hslash]$:
\begin{equation}\label{eq:gammabounds}
\left(\frac{\d_-}{\D_+}\right)^h\leq \Gamma_{h}(y)\leq 2\D_-\left(\frac{\D_-}{\d_+}\right)^h.
\end{equation}
\end{lemma}
\begin{proof}
From Proposition \ref{pr:tx<=1} we may assume that the event $\cG(\hslash)$ holds. From Lemma \ref{lem:easyBh} we know that $\frac12\d_-^h \leq |\partial\cB^-_h(y)|\leq \D_-^h$. 
Thus it suffices to show that for any $z\in\partial\cB^-_h(y)$, $h\in[1,\hslash]$:
\begin{equation}
\label{eq:0gammabounds}
\D_+^{-h}\leq P^h(z,y)\leq 2\d_+^{-h}.
\end{equation}
The bounds in \eqref{eq:0gammabounds} follow from the observation that any path of length $h$ from $z$ to $y$ has weight at least $\D_+^{-h}$ and at most $\d_+^{-h}$, and that there is at least one and at most two such paths if $z\in\partial\cB^-_h(y)$ and $\cG(\hslash)$ holds. The latter fact can be seen with the same argument used in the proof of Lemma \ref{lem:easyBh}. With reference to that proof: in case 1) there are at most two paths from $z$ to $y$,  see Figure \ref{fig:1}; in case 2) there is only one path from $z$ to $y$; see Figure \ref{fig:2a} and   Figure \ref{fig:2b}.
\end{proof}
Roughly speaking, in what follows the extremal values of $\pi$ will be controlled by approximating  $\pi(y)$ in terms of $\G_h(y)$ for values of $h$ of order $\log \log n$, for every node $y$. 
The next two results allow us to control $\G_h(y)$ in terms of $\G_{h_0}(y)$ for all $h\in\[h_0,\hslash\]$
where $h_0$ is of order $\log\log n$. 

\begin{lemma}\label{le:claim2}
There exist constants $c>0$ and $C>0$ such that: 
\begin{equation}\label{eq:claim2}
\P\left(\forall y\in[n],\:\forall h\in\[h_0,\hslash\],\:\Gamma_{h}(y)\geq c\log^{1-\gamma_0}(n)\right)=1-o(1),
\end{equation}
where $\g_0$ is the constant from Theorem \ref{th:pimin} and $ h_0:=\log_{\delta_-}\!\!\log(n) + C$. 
\end{lemma}
\begin{proof}
From Lemma \ref{lem:easyBh}
we may assume that 
$ |\partial\cB^-_{h_0}(y)|\geq \frac12\d_-^{h_0}=:R$ for all $y\in[n]$, where $h_0$ is as in the statement above with $C$ to be fixed later. 
Once we have the in-neighborhood $\cB^{^-}_{h_0}(y)$ we proceed with the generation of the $(h-h_0)$-in-neghborhoods
of all $z\in   \partial\cB^-_{h_0}(y)$. Consider the first $R$ elements of $\partial\cB^-_{h_0}(y)$, and order them as $(z_1,\dots,z_R)$ in some arbitrary way. We sample sequentially $ \cB^{-}_{h-h_0}(z_1)$, then $ \cB^{-}_{h-h_0}(z_2)$, and so on. 
We want to couple the random variables $Z_i:= \cB^{-}_{h-h_0}(z_i)$, $i=1,\dots,R$ with a sequence of independent rooted directed random trees $W_i$, $i=1,\dots,R$, defined as follows. The tree $W_i$ is defined as the first $h-h_0$ generations of the marked random tree $\cT_i$ produced by the following  instructions:
\begin{itemize}
\item
the root is given the mark $z_i$;
\item
every vertex with mark $j$ has $d^-_j$ children, each of which is given independently the mark $k\in[n]$ with probability $d^+_k/m$.
\end{itemize}
Consider the generation of the $i$-th variable   $Z_i$. This is achieved by the breadth-first sequential procedure, where at each step a head is matched with a tail chosen uniformly at random from all unmatched tails; see Section \ref{sec:structure}. If instead we pick the tail uniformly at random from all possible tails, then we need to reject the outcome if the chosen tail belongs to the set of tails that have been already matched. Since the total number of tails matched at any step of this generation is at most $K:=\D^\hslash=O(n^{1/5})$, it follows that the probability of a rejection is bounded by $p:=K/m = O(n^{-4/5})$.  Let us now consider the event of a collision, that is when the chosen tail belongs to a vertex that has already been exposed during the previous steps, including the generation of $\cB^{^-}_{h_0}(y)$ and of the  $Z_j$, $j\leq i$. Notice that the total number of exposed vertices is at most $K$ and therefore the probability of a collision is bounded by $p'=\D K/m = O(n^{-4/5})$. 
Since the generation of $Z_i$ requires at most $K$ matchings, we see that conditionally on the past, a $Z_i$ with no rejections and no collisions is created with probability uniformly bounded from below by $1-q$, where $q=O(n^{-3/5})$.  
We say that $Z_i$ is {\em bad}  if its generation produced a  rejection or a collision. 
Once the $Z_i$'s have been sampled we define a set $\cI$ such that $i\in \cI$ if and only if either $Z_i$ is bad or there is a bad $Z_j$ such that the generation of $Z_j$ produced a collision with a vertex from $Z_i$. 
With this notation, $W_i=Z_i$ for all $i\notin \cI$ and 
\begin{equation}\label{eq:gastar2}
\Gamma_{h}(y)\geq\D_+^{-h_0}\sum_{i\notin \cI} \Gamma_{h-h_0}(z_i).
\end{equation}
The above construction shows that the  cardinality of the set $\cI$ is stochastically dominated by twice the binomial  ${\rm Bin}(R,q)$. Therefore, 
 \begin{equation}\label{eq:hoeff1}
  \bbP(|\cI|\geq 10)\leq \bbP({\rm Bin}(R,q)\geq 5) \leq (Rq)^5 = o(n^{-2}).
\end{equation}
On the other hand, notice that for all $i\notin\cI$:
\begin{equation}\label{eq:gastar3}
\Gamma_{h-h_0}(z_i) = M^{i}_{h-h_0},
\end{equation}
where $M^{i}_t$, $t\in\bbN$, is defined as follows. Let $\cT_{t,i}$ denote the set of vertices forming generation $t$ of the tree $\cT_i$ rooted at $z_i$, and for $x\in \cT_{t,i}$, write \begin{equation}\label{eq:gastar5}
\w(x):=\w\(x\mapsto z_i; \cT_i\)=
\prod_{u=1}^t\frac1{d^+_{x_u}},
	\end{equation}
	for the weight of the path $(x_t=x,x_{t-1},\dots,x_1,x_0=z_i)$ from $x$ to $z_i$ along $\cT_i$. Then $M^{i}_t $ is defined by 
\begin{equation}\label{eq:gastar4}
M^{i}_t = 
 \sum_{x\in\cT_{t,i}}d_x^-\w(x),\qquad M^{i}_0 = d_{z_i}^-.
	\end{equation}
It is not hard to check (see e.g.\ \cite[Proposition 4]{CQ:PageRank}) that for fixed $n$, $(M^{i}_t)_{t\geq 0}$ is a martingale with $$\bbE[M^{i}_t]=M^{i}_0 = d_{z_i}^-.$$ 
In particular, by truncating at a sufficiently large constant $C_1>0$ one has $ M^{i}_{h-h_0} \geq X_i$, where $$X_i:=\min\{M^{i}_{h-h_0},C_1\}$$  are independent random variables with $0\leq X_i\leq C_1$ and $\bbE[X_i]\geq 1$ for all $i$. Therefore, Hoeffding's inequality gives, for any $k\in\bbN$:
\begin{align}\label{eq:gastar06}
\P\Big(\sum_{i=1}^k M^{i}_{h-h_0}
 \leq k/2\Big)& \leq e^{-c_1 k},
\end{align}
where $c_1>0$ is a suitable constant. 

Divide the integers $\{1,\dots,R\}$ into 10 disjoint intervals $I_1,\dots,I_{10}$, each containing $R/10$ elements. If $|\cI|<10$ then there must be one of the intervals, say  $I_{j_*}$, such that $I_{j_*}\cap \cI=\emptyset$. It follows that if $|\cI|<10$, then 
\begin{equation}\label{eq:gastar20}
\sum_{i\notin \cI} \Gamma_{h-h_0}(z_i)\geq \sum_{i\in I_{j^*}} M^{i}_{h-h_0} \geq \min_{\ell=1,\dots,10} \sum_{i\in I_\ell} M^{i}_{h-h_0}.
\end{equation}
Using \eqref{eq:hoeff1}, and \eqref{eq:gastar06}-\eqref{eq:gastar20} we conclude that, for a suitable constant $c_2>0$: 
\begin{align}\label{eq:gastar6}
\P\Big(\sum_{i\notin \cI} \Gamma_{h-h_0}(z_i) \leq c_2 R\Big)& \leq 
\P\Big(\min_{\ell=1,\dots,10} \sum_{i\in I_\ell} M^{i}_{h-h_0} \leq c_2 R\Big) + \bbP(|\cI|\geq 10)\nonumber\\
&\leq 10 \exp{\(-c_1 R/10\)}+ o(n^{-2}). 
\end{align}
Since $R=\frac12 \d_-^{h_0}=\frac12{\d_-^C}\log n$, the probability in \eqref{eq:gastar6} is $o(n^{-2})$ if $C$ is large enough.   
From \eqref{eq:gastar2}, on the event $\sum_{i\notin \cI} \Gamma_{h-h_0}(z_i) > c_2R$ one has
\begin{equation}\label{eq:gastar200}
\Gamma_{h}(y)\geq\tfrac12c_2\d_-^{h_0}\D_+^{-h_0} = c \,\log ^{1-\g_0}(n),
\end{equation}
where $c=\tfrac12c_2(\d_-/\D_+)^{C}$.
Thus the event \eqref{eq:gastar200} has probability $1-o(n^{-2})$, and the desired conclusion follows by taking a union bound over $y\in[n]$ and $h\in[h_0,\hslash]$. 
\end{proof}

\begin{lemma}\label{lem:CFcon}
There exists a constant $K>0$ such that for all $\e>0$, with high probability:
\begin{equation}
\max_{y\in[n]}\max_{h\in[h_1,\hslash]}\Big|\frac{\Gamma_h(y)}{\Gamma_{h_1}(y)}-1
\Big|\leq \e,
\end{equation}
where $h_1:=K\log\log(n)$.
\end{lemma}
\begin{proof}
For any $h\geq h_1$, let $\si_h$ denote a realization of the in-neighborhood $\cB^-_{h}(y)$, obtained with the usual breadth-first sequential generation. From Proposition \ref{pr:tx<=1} we may assume that the tree excess of $\cB^-_{h}(y)$ is at most 1, as long as $h\leq \hslash$. Call $\cE_{tot,h},\cF_{tot,h}$ the set of unmatched tails and unmatched heads, respectively, after the generation of $\si_h$. Let also $\cE_h\subset\cE_{tot,h}$ denote the set of unmatched tails belonging to vertices not yet exposed, and let $\cF_h$ be the subset of 
heads attached to   $\partial\cB^-_h(y)$.  
By construction, all heads attached to $\partial\cB^-_h(y)$ must be unmatched at this stage so that $\cF_h\subset \cF_{tot,h}$. Moreover,    
\begin{equation}
\Gamma_h(y)=\sum_{f\in\cF_h}P^h(v_f,y),
\end{equation}
where $v_f$ denotes the vertex to which the head $f$ belongs. 
To compute $\G_{h+1}$ given $\si_h$ we let $\omega:\cE_{tot,h}\mapsto\cF_{tot,h}$ denote a uniform random matching of $\cE_{tot,h}$ and $\cF_{tot,h}$, 
and notice that a vertex $z$ is in $\partial\cB^-_{h+1}(y)$ if and only if $z$ is revealed by matching one of the heads $f\in \cF_h$ with one of the tails $e\in \cE_h$. Therefore,
 \begin{align}
\Gamma_{h+1}(y)&=\sum_{e\in\cE_h}\frac{d^-_e}{d^+_e}\sum_{f\in\cF_{h}}
P^h(v_f,y)\ind_{\omega(e)=f}\nonumber\\
&=\sum_{e\in\cE_{tot,h}}c(e,\omega(e)), \label{eq:chacha}
\end{align}
where we use the notation $d^\pm_e$ for the degrees of the vertex to which the tail $e$ belongs, and the function $c$ is defined by
 \begin{equation}\label{eq:chatte0}
c(e,f)=\frac{d^-_e}{d^+_e}P^h(v_f,y)\ind_{e\in\cE_h,f\in\cF_h}.
\end{equation}	
Since $\si_h$ is such that $\tx(\cB^-_{h}(y))\leq 1$, we may estimate $P^h(v_f,y)$ as in \eqref{eq:0gammabounds},   so that
\begin{equation}\label{eq:chatte00}
\|c\|_\infty = \max_{e,f}c(e,f)\leq 2\D\,\d^{-h-1}.
\end{equation}
We now use a version of Bernstein's inequality proved by Chatterjee (\cite[Proposition 1.1]{chatterjee2007stein}) which applies to any function of a uniform random matching of the form \eqref{eq:chacha}.
It follows that for any fixed $\si_h$, for any $s>0$:
\begin{equation}\label{eq:chatteo1}
\P\(|\Gamma_{h+1}(y)-\E\left[\Gamma_{h+1}(y)\tc \si_h\right]|\geq s\tc \si_h\) 
\leq 2\exp\(-\frac{s^2}{2\left\|c \right\|_{\infty}(2\E\left[\Gamma_{h+1}(y)\tc \si_h\right]+s)} \).
\end{equation}
Taking $s=a\E\left[\Gamma_{h+1}(y)\tc \si_h\right]$, $a\in(0,1)$, one has
\begin{equation}\label{eq:chatterjeeo2}
\P\(|\Gamma_{h+1}(y)-\E\left[\Gamma_{h+1}(y)\tc \si_h\right]|\geq s\tc \si_h\) 
\leq  2\exp\(-\frac{a^2\E\left[\Gamma_{h+1}(y)\tc \si_h\right]}{6\left\|c \right\|_{\infty}} \).
\end{equation}
Since  the probability of the event $\omega(e)=f$ conditioned on $\si_h$ is $\frac1{|\cE_{tot,h}|}=\frac1m(1+O(\D^h/m))$, we have
\begin{align}\label{eq:gas}
\E\left[\Gamma_{h+1}(y)\tc \si_h\right]&=
\frac1{|\cE_{tot,h}|}\sum_{e\in\cE_h}\frac{d^-_e}{d^+_e}\G_h(y) \nonumber\\&=
\frac1m\(1+O(\D^h/m)\)\(m-\sum_{e\notin\cE_h}\frac{d^-_e}{d^+_e}\)\G_h(y) \nonumber \\&=
\(1+O(\D^h/m)\)\G_h(y)=
\(1+O(n^{-1/2})\)\G_h(y),
\end{align}
for all $h\in[h_1,\hslash]$, where we use the fact that the sum over all tails $e$ (matched or unmatched) of $d^-_e/d^+_e$ equals $m$. In particular, from Lemma \ref{le:claim2} it follows that for some constant $c>0$:
\begin{equation}\label{eq:need}
\E\left[\Gamma_{h+1}(y)\tc \si_h\right]\geq c\log^{-\g_0 +1}(n),
\end{equation}
and therefore, using \eqref{eq:chatte00}, one finds 
\begin{equation}\label{eq:need0}
\|c\|_\infty^{-1}\E\left[\Gamma_{h+1}(y)\tc \si_h\right]\geq \log^6(n),
\end{equation} 
for all $h\geq h_1$, if the constant $K$ in the definition of $h_1$ is large enough.
From \eqref{eq:chatterjeeo2}, \eqref{eq:gas} and \eqref{eq:need0} it follows that, letting
$$
\cA:=\left\{|\Gamma_{h+1}(y)-\Gamma_{h}(y)|\leq a  \Gamma_{h}(y)\,, \;\forall h\in[h_1,\hslash]\right\},
$$
with $a:=\log^{-2}(n)$, then 
 \begin{equation}\label{eq:chatterjeeo3}
\P\(
\cA\) =1-o(1).
\end{equation}
Moreover, on the event $\cA$, for all $h\in[h_1,\hslash]$:
$$
|\Gamma_{h}(y)-\Gamma_{h_1}(y)|\leq \sum_{j=h_1}^{h-1}\left|\Gamma_{j+1}(y)-\Gamma_{j}(y)\right|\leq \e \Gamma_{h_1}(y).
$$
\end{proof}


\subsection{Lower bound on $\pi_{\min}$}\label{suse:pimin-lb}
If for some $t\in\bbN$ and $a>0$ one has $P^t(x,y)\geq a$ for all $x,y\in[n]$, then 
\begin{equation}\label{eq:pih}
\pi(z) = \sum_{x=1}^n
\pi(x)P^t(x,z)\geq a,
\end{equation}
and therefore $\pimin\geq a$. We will prove the lower bound on $P^t(x,y)$ by choosing $t$ of the form $t=(1+\e)\tent$, for some small enough $\e>0$; see \eqref{def:tent} for the definition of $\tent$. More precisely, fix a constant $\eta>0$, set $\eta'= 3\eta\frac{H}{\log\d}$, 
and define 
\begin{equation}\label{eq:tmixe}
t_\star=
h_x+h_y+1\,,\quad 
h_x=(1-\eta)\tent\,,\quad
h_y=\eta'\tent.
\end{equation}
Note that $\eta'\geq 3\eta$ and thus $t_\star=t_\star(\eta)\geq (1+2\eta)\tent$. 

\begin{lemma}\label{lem:lower-nice-gamma}
There exists $\eta_0>0$ such that for all $\eta\in(0,\eta_0)$:
		\begin{equation}\label{eq:www}
	\P\(\forall x,y\in[n],\:\: P^{t_\star+1}(x,y)\geq \tfrac{c}{n}\,\Gamma_{h_y}(y) \)=1-o(1),
	\end{equation}
	for some constant $c=c(\eta,\D)>0$. 
	\end{lemma}
From  \eqref{eq:pih} and
Lemma \ref{lem:lower-nice-gamma} it follows that w.h.p.\ for all $y$ 
		\begin{equation}\label{eq:pigamma}
	\pi(y)\geq \tfrac{c}{n}\,\Gamma_{h_y}(y) .
	\end{equation}
Lemma \ref{le:claim2} thus implies, for some new constant $c>0$
		\begin{equation}\label{eq:piminlowbo}
	\P\(\pimin\geq \tfrac{c}{n}\log^{1-\gamma_0}(n)\right)=1-o(1),
	\end{equation}
which settles the lower bound in Theorem \ref{th:pimin}. 

To prove Lemma \ref{lem:lower-nice-gamma} we will restrict to a subset of {\em nice} paths from $x$ to $y$. This will allow us to obtain a concentration result for the probability to reach $y$ from $x$ in $t_\star$ steps.

\subsubsection{A concentration result for 
nice paths}\label{sec:bcsconc}
The definition of the nice paths follows a construction introduced in \cite{BCS1}, which we now recall. In contrast with \cite{BCS1} however, here we need a lower bound on $P^{t_\star}(x,y)$ and thus the argument is somewhat different. 

Following \cite[Section 6.2]{BCS1} and \cite[Section 4.1]{BCS2}, we introduce the rooted directed tree $\cT(x)$,  namely the subgraph of the $h_x$-out-neighborhood of $x$ defined by the following process: initially all tails and heads are unmatched and $\cT(x)$ is identified with its root, $x$; throughout the process, we let $\partial_+\cT(x)$ (resp. $\partial_-\cT(x)$) denote the set of unmatched tails (resp. heads) whose endpoint belongs to $\cT(x)$; the height $\h(e)$ of a tail $e\in\partial_+\cT(x)$ is defined as $1$ plus the number of edges in the unique path in $\cT(x)$ from $x$ to the endpoint of $e$; the weight of $e\in\partial_+\cT(x)$ is defined as 
\begin{equation}\label{eq:weighte}
\w(e) = \prod_{i=0}^{\h(e)-1}\frac1{d_{x_i}^+}\,,
\end{equation}
where $(x=x_0,x_1,\dots,x_{\h(e)-1})$ denotes the path in $\cT(x)$ from $x$ to the endpoint of $e$;
we then iterate the following steps:
	\begin{itemize}
		\item a tail $e\in \partial_+\cT(x)$ is selected with maximal weight among all $e\in\partial_+\cT(x)$ with 
		 $\h(e) \leq h_x-1$ and 
		  $\w(e) \geq \w_{min}:=n^{-1+\eta^2}$ (using an arbitrary ordering of the tails to break ties); 
		
		\item $e$ is matched to a uniformly chosen unmatched head $f$, forming the edge $ef$;
		
		\item if $f$ was not in $\partial_-\cT(x)$, then its endpoint and the edge $ef$ are added to  $\cT(x)$.
	\end{itemize}
	The process stops when there are no tails $e\in\partial_+\cT(x)$ with height $\h(e) \leq h_x-1$ and weight $\w(e)\geq \w_{min}$.  
Note that $\cT(x)$ remains a directed tree at each step. The final value of $\cT(x)$
represents the desired directed tree. 
After the generation of the tree $\cT(x)$ a total number $\kappa$ of edges has been revealed, some of which may not belong to $\cT(x)$. As in	\cite[Lemma 7]{BCS2}, it is not difficult to see that when exploring the out-neighborhood of $x$ in this way the random variable $\kappa$ is deterministically bounded as 
	\begin{equation}\label{eq:kappaw}
	\kappa\leq n^{1-\frac{\eta^2}2}.
	\end{equation}
At this stage, let us call $\cE^*(x)$ 
	the set of unmatched tails $e\in \partial_+\cT(x)$ such that $\h(e)=h_x$.

\begin{definition}
	A path ${\bf p}=(x_0=x,x_1,\dots,x_{t_\star}=y)$ of length $t_\star$ starting at $x$ and ending at $y$ is called \emph{nice} if it satisfies:	
	\begin{enumerate}
		\item\label{it:nicepaths1} The first $h_x$ steps of ${\bf p}$
		are contained in $\cT(x)$, and satisfy 
		$$
		\prod_{i=0}^{h_x}\frac1{d_{x_i}^+}\leq n^{2\eta-1};
		$$
		\item\label{it:nicepaths2} 
		$x_{h_x+1}\in\partial\cB^-_{h_y}(y)$.
	\end{enumerate}
\end{definition}
To obtain a useful expression for the probability of going from $x$ to $y$ along a nice path, we need to  generate $\cB^-_{h_y}(y)$, the $h_y$-in-neighborhood of $y$. 
To this end, assume that $\kappa$ edges in the $h_x$-out-neighborhood of $x$ have been already sampled according to the procedure described above, and then sample $\cB^-_{h_y}(y)$ according to the sequential generation described in Section \ref{sec:structure}. Some of the matchings producing $\cB^-_{h_y}(y)$ may have already been revealed during the previous stage. In any case, this second stage creates an additional random number $\tau$ of edges, satisfying the crude bound $\tau\leq\Delta^{h_y+1}$. 
	We call $\cF_{tot}$ the set of unmatched heads, and $\cE_{tot}$ the set of unmatched tails after the sampling of these $\kappa+\tau$ edges. 
Consider the set $\cF^0:= \cF_{h_y}\cap \cF_{tot}$, where $\cF_{h_y}$ denotes the set of all heads (matched or unmatched) attached to vertices in $\partial \cB^{-}_{h_y}(y)$. 
	Moreover, call $\cE^0:= \cE^*(x)\cap\cE_{tot}$ the subset of unmatched tails which are attached to vertices at height $h_x$ in $\cT(x)$.
	Finally,  complete the generation of the digraph by matching the $m-\kappa-\tau$ unmatched tails $\cE_{tot}$ to the $m-\kappa-\tau$ unmatched heads $\cF_{tot}$ using a uniformly random bijection $\omega:\cE_{tot}\mapsto\cF_{tot}$. For any $f\in\cF_{h_y}$ we introduce the notation
	\begin{equation}\label{eq:weightf}
		\w(f):=P^{h_y}(v_f,y),
		\end{equation}
		where $v_f$ denotes the vertex  $v\in \partial \cB^{-}_{h_y}(y)$ such that $f\in E_v^-$.
With the notation introduced above, the probability to go from $x$ to $y$ in $t_\star$ steps following a nice path can now be written as
	\begin{equation}\label{eq:nice-xy}
	P_{0,t_\star}(x,y):=\sum_{e\in\cE^0}\sum_{f\in\cF^0}\w(e)\w(f) \ind_{\omega(e)=f}\ind_{\w(e)\le n^{2\eta-1}}. 
	\end{equation}
	Note that, conditionally on the construction of the first $\kappa+\tau$ edges described above, each Bernoulli random variable $ \ind_{\omega(e)=f}$ appearing in the above sum has probability of success at least $1/m$. 
	In particular, if $\si$ denotes a fixed realization of the $\kappa+\tau$ edges, then 
\begin{equation}\label{eq:nice-e}
\E\left[P_{0,t_\star}(x,y)\tc \si\right]
\geq \frac1m\,A_{x,y}(\si)B_{x,y}(\si)\,,
\end{equation}
where
\begin{equation}\label{eq:defaxby}
A_{x,y}(\si):=\sum_{e\in\cE^0}
\ind_{\w(e)\le n^{2\eta-1}}\w(e)\,,\quad B_{x,y}(\si):=\sum_{f\in\cF^0}\w(f).
\end{equation}
Moreover, the probability of $\omega(e)=f$ for any fixed $e\in\cE^0,f\in\cF^0$ is at most $1/(m-\kappa-\tau)$, so that 
\begin{equation}\label{eq:nice-ep}
\E\left[P_{0,t_\star}(x,y)\tc \si\right]
\leq \frac{(1+o(1))}m\,A_{x,y}(\si)B_{x,y}(\si)\leq \frac{(1+o(1))}m\,\G_{h_y}(y),
\end{equation}
where we use $A_{x,y}\leq 1$ and $B_{x,y}\leq \G_{h_y}(y)$.
Consider the event 
\begin{equation}\label{eq:nice-event}
\cY_{x,y}=\Big\{\si:\; A_{x,y}(\si) 
\geq \tfrac 12\,,\;
B_{x,y}(\si)
\geq \log^{-\g_0}(n)\,, \:\tx(\cB_{h_y}^-(y))\leq 1
\Big\},
\end{equation}
where the exponent $-\g_0$ is chosen for convenience only and any exponent $-c$ with $c>\g_0-1$ would be as good. 
\begin{lemma}
\label{lem:chatt}
There exists $\eta_0>0$ such that for all $\eta\in(0,\eta_0)$, for any $\si\in \cY_{x,y}$, any $a\in(0,1)$:
\begin{equation}\label{eq:chatt1}
\P\(|P_{0,t_\star}(x,y)-\E\left[P_{0,t_\star}(x,y)\tc \si\right]|\geq a\E\left[P_{0,t_\star}(x,y)\tc \si\right]\tc \si\) 
\leq 2\exp\(-a^2n^{\eta/2}\)
\end{equation}
\end{lemma}
	\begin{proof}
Conditioned on $\si$, $P_{0,t_\star}(x,y)$ is a function 
of the uniform random permutation  $\omega:\cE_{tot}\mapsto\cF_{tot}$,
  \begin{equation}\label{eq:chatt0}
P_{0,t_\star}(x,y)
=\sum_{e\in\cE_{tot}} c(e,\omega(e))\,,\quad c(e,f)=\w(e)\w(f)\ind_{\w(e)\le n^{2\eta-1}}\ind_{e\in\cE^0,f\in\cF^0}.
\end{equation}	
Since we are assuming $\tx(\cB_{h_y}^-(y))\leq 1$, we can use \eqref{eq:0gammabounds} to estimate $\w(f)\leq2 \d^{-h_y}=n^{-3\eta}$ 
for any $f\in\cF^0$. Therefore
 \begin{equation}\label{eq:chatt0ne}
\|c \|_{\infty}=\max_{e,f}c(e,f)\leq 2n^{-1-\eta}.
\end{equation}
As in Lemma \ref{lem:CFcon}, and as in \cite{BCS1}, we use Chatterjee's concentration inequality for uniform random matchings \cite[Proposition 1.1]{chatterjee2007stein}  to obtain
	for any $s>0$:
\begin{equation}\label{eq:chatte1}
\P\(|P_{0,t_\star}(x,y)-\E\left[P_{0,t_\star}(x,y)\tc \si\right]|\geq s\tc \si\) 
\leq 2\exp\(-\frac{s^2}{2\left\|c \right\|_{\infty}(2\E\left[P_{0,t_\star}(x,y)\tc \si\right]+s)} \).
\end{equation}
Taking $s=a\E\left[P_{0,t_\star}(x,y)\tc \si\right]$, $a\in(0,1)$, one has
\begin{equation}\label{eq:chatterjee2}
\P\(|P_{0,t_\star}(x,y)-\E\left[P_{0,t_\star}(x,y)\tc \si\right]|\geq s\tc \si\) 
\leq  2\exp\(-\frac{a^2\E\left[P_{0,t_\star}(x,y)\tc \si\right]}{6\left\|c \right\|_{\infty}} \).
\end{equation}
Using \eqref{eq:nice-e},  \eqref{eq:nice-event}, and \eqref{eq:chatt0ne} one concludes that  \eqref{eq:chatt1}
holds for all $\si\in\cY_{x,y}$ and for all $n$ large enough.
	\end{proof}
	
\subsubsection{Proof of Lemma \ref{lem:lower-nice-gamma}}
Let $V_*$ denote the set of all $z\in [n]$ such that $\cB_\hslash^+(z)$ is a directed tree. As observed in \cite[Proposition 6]{BCS1}, it is an immediate consequence of Proposition \ref{pr:tx<=1} that with high probability, for all $x\in[n]$:
\begin{equation}\label{eq:pxvstar0}
P(x,V_*)=\sum_{z\in V_*}P(x,z) \geq \tfrac12.
\end{equation}
Therefore, 
\begin{equation}\label{eq:pxvstar2}
P^{t_\star +1}(x,y)\geq \tfrac12\min_{x\in V_*}P^{t_\star }(x,y).
\end{equation}
Since $P^{t_\star }(x,y)\geq P_{0,t_\star}(x,y)$ it is sufficient to prove 
\begin{equation}\label{eq:wow2}
	\P\(\forall x\in V_*, \forall y\in[n],\:\: P_{0,t_\star}(x,y)\geq \tfrac{c}{n}\,\Gamma_{h_y}(y) \)=1-o(1),
	\end{equation}
	for some constant $c=c(\eta,\D)>0$. The proof of \eqref{eq:wow2} is based on  Lemma \ref{lem:chatt} and the following estimates which allow us to make sure the events $\cY_{x,y}$ in Lemma \ref{lem:chatt} have large probability.  
\begin{lemma}\label{lem:eveA1}
The event
$\cA_1= \{\forall x\in V_*, \forall y\in[n]: A_{x,y}
\geq \tfrac 12\}$ %
has probability
$$\P\(\cA_1 \)=1-o(1)\,.$$ 
\end{lemma}
\begin{proof}
Let us first note that the event
$\widehat\cA_1= \{\forall x\in V_*: \sum_{e\in\cE^*(x)}\w(e)\ind_{\w(e)\le n^{2\eta-1}}\geq 0.9\}$
satisfies
$$\P\(\widehat\cA_1 \)=1-o(1).$$ 
Indeed, this fact  is  a consequence of \cite{BCS1,BCS2}, which established that 
for any $\e>0$, with high probability
\begin{equation}\label{eq:wow3}
\min_{x\in V_*}\sum_{e\in\cE^*(x)}\w(e)\ind_{\w(e)\le n^{2\eta-1}}\geq 1-\e,
\end{equation}
see e.g.\ \cite[Theorem 4 and Lemma 11]{BCS2}.
Thus, it remains to show that replacing $\cE^*(x)$ with $\cE^0$ does not alter much the sum. Suppose the $\kappa$ edges generating $\cT(x)$ have been revealed and then sample the $\t$ edges generating the neighborhood $\cB_{h_y}^-(y)$. Let $K$ denote the number of collisions between $\cT(x)$ and $\cB_{h_y}^-(y)$. 
There are at most $N:=\D^{h_y}=n^{3\eta\log\D/\log \d}$ attempts each with success probability at most $p:=\kappa/(m-\kappa)$. Thus $K$ is stochastically dominated by a binomial ${\rm Bin}(N,p)$, and therefore by Hoeffding's inequality 
 $$
 \P(K>Np+N)\leq \exp{\(-2N\)} \leq \exp{\(-n^{3\eta}\)}. 
 $$
 Thus by a union bound we may assume that all $x,y$ are such that the corresponding collision count $K$ satisfies $K\leq Np+N\leq 2N$. Therefore, on the event $\widehat\cA_1$ 
  $$
  \sum_{e\in\cE^0}\w(e)\ind_{\w(e)\le n^{2\eta-1}} \geq 0.9- 2N  \,n^{2\eta-1} \geq\frac12,
  $$
  if $\eta$ is small enough.
\end{proof}
\begin{lemma}\label{lem:eveA2}
Fix a constant $c>0$ and consider the event
$\cA_2= \{\forall x,y\in [n]: B_{x,y}
\geq c\,\G_{h_y}(y)\}$.
If $c>0$ is small enough
$$\P\(\cA_2 \)=1-o(1)\,.$$ 
\end{lemma}
\begin{proof}
By definition, $\sum_{f\in\cF_{h_y}}\w(f)=\G_{h_y}(y)$. Thus, we need to show that 
if we replace $\cF^0$ by $\cF_{h_y}$ the sum defining $B_{x,y}$ is still comparable to $\G_{h_y}(y)$.
For any constant $T>0$, 
for each $z\in \partial\cB_{h_y-T}^-(y)$, let $V_z$ denote the set of $w\in \partial\cB_{h_y}^-(y)$ such that $d(w,z)=T$. Notice that if the event $\cG(\hslash)$ from Proposition \ref{pr:tx<=1} holds then for each $z\in \partial\cB_{h_y-T}^-(y)$ one has $|V_z|\geq \frac12\d^T$. Consider the generation of the $\kappa +\tau$ edges as above, and call a vertex $z\in \partial\cB_{h_y-T}^-(y)$ {\em bad} if all heads attached to $V_z$ are matched, or equivalently if none of these heads is in $\cF_{tot}$.  Given a $z\in \partial\cB_{h_y-T}^-(y)$, we want to estimate the probability that it is bad. To this end, we use the same construction given in Section \ref{sec:bcsconc} but this time we first generate  the in-neighborhood $\cB_{h_y}^-(y)$ and then the tree $\cT(x)$. 
Let $K$ denote the number of collisions between $\cT(x)$ and the set $V_z$. Notice that $|V_z|\leq \D^T$ and that $|\cT(x)|\leq n^{1-\eta^2/2}$, so that $K$ is stochastically dominated by the binomial ${\rm Bin}(N,p)$ where $N=n^{1-\eta^2/2}$ and $p=\D^{T+1}/n$. 
Therefore, 
 $$
 \P\(K>\tfrac12\d^T\)\leq  (Np)^{\frac12\d^T}\leq \(\D^{T+1}n^{-\eta^2/2}\)^{\frac12\d^T}. 
 $$
Since $|V_z|\geq \frac12\d^T$, if $z$ is bad then $K>\frac12\d^T$ and thus the probability of the event that $z$ is bad is at most $O(n^{-\d^T\eta^2/4})$. The probability that there exists a bad $z\in \partial\cB_{h_y-T}^-(y)$ is then bounded by $O(\D^{h_y}n^{-\d^T\eta^2/4})$. In conclusion, if $T=T(\eta)$ is a large enough constant, we can ensure that for any $y\in[n]$ the probability that there exists a bad $z\in \partial\cB_{h_y-T}^-(y)$ is $o(n^{-2})$, and therefore, by a union bound, with high probability there are no bad $z\in \partial\cB_{h_y-T}^-(y)$, for all $x,y\in[n]$.  On this event, for all $z$ we may select one vertex $w\in V_z$ with at least one head $f\in\cF^0$ attached to it. Notice that $\w(f)\geq \D^{-T-1}P^{h_y-T}(z,y)$.
   Therefore, assuming that there are no bad $z\in \partial\cB_{h_y-T}^-(y)$:
\begin{align*}
B_{x,y}(\si)&=\sum_{f\in\cF^0}\w(f) \\&
\geq \D^{-T}\!\!\!\sum_{z\in \partial\cB_{h_y-T}^-(y)}P^{h_y-T}(z,y) 
\geq\D^{-T-1}\G_{h_y-T}(y).
\end{align*}
From Lemma \ref{lem:CFcon} we may finish with the estimate $\G_{h_y-T}(y)\geq \frac12 \G_{h_y}(y)$. 
\end{proof}

We can now conclude the proof of \eqref{eq:wow2}.
Consider the event 
 \begin{equation}
\cA=\cA_1\cap\cA_2\cap\cG(\hslash).
\end{equation}
For any $s>0$, 
\begin{align}\label{eq:eveG2}
\P\(\forall x,y\in[n],\:\: P_{0,t_\star}(x,y)\geq \tfrac{s}{n}\,\Gamma_{h_y}(y) \)
\geq \P(\cA) - \sum_{x,y\in[n]}\P\(P_{0,t_\star}(x,y)<\tfrac{s}{n}\,\Gamma_{h_y}(y); \cA\).
\end{align}
From Lemma \ref{lem:eveA1}, Lemma \ref{lem:eveA2}, and Proposition \ref{pr:tx<=1} it follows that $\P(\cA) =1-o(1)$.
Let $\cW_{x,y}$ denote the event 
\begin{equation}\label{eq:nice-ex}
\E\left[P_{0,t_\star}(x,y)\tc \si\right]
\geq \tfrac{c}{2m}\,\G_{h_y}(y),
\end{equation}
where $c$ is the constant from Lemma \ref{lem:eveA2}. From Lemma \ref{le:claim2} we infer that 
$$
\cA\subset \cW_{x,y}\cap\cY_{x,y}, 
$$
for all $x,y$, and for all $n$ large enough. 
Therefore, 
\begin{align}\label{eq:eveG6}
\P\(P_{0,t_\star}(x,y)< \tfrac{s}{n}\,\Gamma_{h_y}(y); \cA\) 
\leq \sup_{\si\in\cW_{x,y}\cap\cY_{x,y}}  \P\(P_{0,t_\star}(x,y)< \tfrac{s}{n}\,\Gamma_{h_y}(y)\tc \si\). 
\end{align}
Taking $s>0$ a small enough constant and using \eqref{eq:nice-ep} and \eqref{eq:nice-ex},   
we see that $P_{0,t_\star}(x,y)< \frac{s}{n}\,\Gamma_{h_y}(y)$ implies 
$$
|P_{0,t_\star}(x,y)-\E\left[P_{0,t_\star}(x,y)\tc \si\right]|\geq a\,\E\left[P_{0,t_\star}(x,y)\tc \si\right],
$$
for some constant $a>0$, and therefore from Lemma \ref{lem:chatt}
\begin{align}\label{eq:eveG7}
\sup_{\si\in\cW_{x,y}\cap\cY_{x,y}} 
\P\(P_{0,t_\star}(x,y)< \tfrac{s}{n}\,\Gamma_{h_y}(y)\tc \si\)=o(n^{-2}).
\end{align}
The bounds \eqref{eq:eveG2} and \eqref{eq:eveG7} end the proof of \eqref{eq:wow2}.
This ends the proof of Lemma \ref{lem:lower-nice-gamma}. 

\begin{remark}\label{re:gamma0L}
Let us show that if the type $(\d_-,\D_+)$ is not in the set of linear types $\cL$ one can improve the lower bound on $\pimin$ as mentioned in Remark \ref{rem:optimal}.
The proof given above shows that it is sufficient to replace $\g_0$ by $\g'_0$ in Lemma \ref{le:claim2}, where $\g'_0$ is defined by \eqref{eq:g0prime}.
To this end, for any $\e>0$, let $\cL_\e$ denote the set of types 
$(k,\ell)\in\cC$ such that
\begin{equation}\label{eq:g0pe}
\limsup_{n\to\infty}\frac{|\cV_{k,\ell}|}{n^{1-\e}}=+\infty\,,
\end{equation}
where $\cV_{k,\ell}$ denotes the set of vertices of type $(k,\ell)$, and define 
\begin{equation}\label{eq:g0primee}
\g'_\e:=\frac{\log\D'_{\e,+}}{\log \d'_{\e,-}}\,,\qquad \D'_{\e,+} := \max\{\ell:\; (k,\ell)\in\cL_\e\} \,,\quad \d'_{\e,-} := \min\{k:\; (k,\ell)\in\cL_\e\}.
\end{equation}
The main observation is that if $(k,\ell)\notin\cL_\e$, then w.h.p.\ there are at most a finite number of vertices of type $(k,\ell)$ in all in-neighborhoods $\cB^-_{h_0}(y)$, $y\in[n]$,  for any $h_0=O(\log\log n)$. Indeed, for a fixed $y\in[n]$ the number of $v\in\cV_{k,\ell}\cap\cB^-_{h_0}(y) $
is stochastically dominated by the binomial ${\rm Bin}\(\Delta^{h_0},n^{-\e/2} \)$, and therefore if $K=K(\e)$ is a sufficiently large constant then the probability of having more than $K$ such vertices is bounded by $(\Delta^{h_0}n^{-\e/2})^K=o(n^{-1})$. Taking a union bound over $y\in[n]$ shows that w.h.p.\ all  $\cB^-_{h_0}(y)$, $y\in[n]$ have at most $K$ vertices with type $(k,\ell)$. Then we may repeat the argument of 
Lemma \ref{le:claim2} with this constraint, to obtain that for all $\e>0$, w.h.p.\ 
$\G_{h_y}(y)\geq c(\e) \log^{1-\g'_\e}(n)$. Since the number of types is finite one concludes that if $\e$ is small enough then $\g'_0=\g'_\e$ and the desired conclusion follows. 
\end{remark}

\subsection{Upper bound on $\pi_{\min}$}\label{suse:pimin-ub}
In this section we prove the upper bound on $\pimin$ given in \eqref{eq:th-pimin-size}. We  first show that we can replace $\pi(y)$ in  \eqref{eq:th-pimin-size} by a more convenient quantity. 
Define the distances
\begin{equation}\label{distances}
d(s)=\max_{x\in [n]}\|P^{s}(x,\cdot)-\pi\|_\tv\,,\quad 
\bar d(s) = \max_{x,y\in [n]}\|P^{s}(x,\cdot)-P^{s}(y,\cdot)\|_\tv.
\end{equation}
It is standard that, for all $k,s\in\bbN$,
\begin{equation}\label{distances1}
d(ks)\leq \bar d(ks) \leq \bar d(s)^k \leq 2^kd(s)^k,
\end{equation}
see e.g.\ \cite{LevPer:AMS2017}. In particular, defining 
\begin{equation}\label{lamt}
\l_t(y)=\frac1n\sum_{x\in[n]}P^t(x,y)\,,
\end{equation}
for any $k\in\bbN$, setting $t=2k\tent$, one has
\begin{equation}\label{distances2}
\max_{y\in[n]}|\l_{t}(y)- \pi(y)|\leq d(2k\tent)\leq 2^{k}d(2\tent)^k.
\end{equation}
From \eqref{cutoff} we know that w.h.p. $d(2\tent)\leq \frac1{2e}$ so that the right hand side above is at most $e^{-k}$. If $k=\Theta(\log^2(n))$ we can safely replace $\pi(y)$ with $\l_t(y)$ in  \eqref{eq:th-pimin-size}. Thus, it suffices to prove the following statement. 
\begin{lemma}\label{lalemma}
For some constants $\b>0$, $C>0$, 
and for any $t=t_n=\Theta(\log^3(n))$:
 \begin{equation}\label{eq:lale1}
\P\Big(\exists S\subset[n],\:|S|\geq n^\b\,,\; n \max_{y\in S}\l_{t}(y)\leq  C\,\log^{1-\gamma_1}(n)
\Big)=1-o(1).
\end{equation}

\end{lemma}
\begin{proof}
Let $(\d_*,\D_*)\in\cL$ denote the type realizing the maximum in the definition of $\g_1$; see \eqref{eq:def-gamma01}. Let $V_*=\cV_{\d_*,\D_*}$ denote the set of vertices of this type, and let $\a_*\in(0,1)$ be a constant such that  $|V_*|\geq \a_* n$, for all $n$ large enough. Let us fix a constant $\b_1\in(0,\tfrac14)$. 
This will be related to the constant $\b$, but we shall not look for the optimal exponent $\b$ in the  statement \eqref{eq:lale1}. Consider the first $N_1:=n^{\b_1}$ vertices in the set $V_*$, and call them $y_1,\dots,y_{N_1}$. Next,  generate sequentially the in-neighborhoods $\cB^-_{h_0}(y_i)$, $i=1,\dots,N_1$, where 
\begin{equation}\label{eq:lale01}
h_0=\log_{\d_*}\!\log n - C_0,
\end{equation}
for some constant $C_0$ to be fixed later.  As in the proof of Lemma \ref{le:claim2} we couple the $\cB^-_{h_0}(y_i)$ with independent random trees $Y_i$ rooted at $y_i$. For each $\cB^-_{h_0}(y_i)$ the probability of failing to equal $Y_i$, conditionally on the previous generations, is uniformly bounded above by $p:=N_1\D^{2h_0}/m$. Let $\cA$ denote the event that 
 all $\cB^-_{h_0}(y_i)$
are successfully coupled to the $Y_i$'s and that they have no intersections.  Therefore, 
\begin{equation}\label{eq:lale2}
\P(\cA)\geq 1-O(N_1p)\geq 1-O(n^{3\b_1-1})=1-o(1).
\end{equation}
Consider now a single random tree $Y_1$. We say that $Y_1$ is {\em unlucky} if all labels of the vertices in the tree are of type $(\d_*,\D_*)$. The probability that $Y_1$ is unlucky is at least 
$$
q=\(\frac{\a_* n \D_*}{m}\)^{\d_*^{h_0}}\geq n^{-\eta},
$$ 
where $\eta=\d_*^{-C_0}\log(\D/2\a_*)$ if $C_0$ is the constant in \eqref{eq:lale01}. We choose $C_0$ so large that $0<\eta\leq \b_1/4$. 
Call $S_1$ the set of $y\in\{y_1,\dots,y_{N_1}\}$ such that $Y_i$ is unlucky. Since the $Y_i$ are i.i.d.\ the probability that $|S_1|<n^{\b_1/2}$ is bounded by the probability that  ${\rm Bin}(N_1,q)< n^{\b_1/2}$, which by Hoeffding's inequality is at most
\begin{equation}\label{eq:lale02}
\exp{\(-n^{\b_1/3}\)}
\end{equation}
Fix  a realization $\si$ of the    in-neighborhoods $\cB^-_{h_0}(y_i)$, $i=1,\dots,N_1$. Say that $y_i$ is unlucky if all vertices in $\cB^-_{h_0}(y_i)$ are of type $(\d_*,\D_*)$. Thanks to \eqref{eq:lale2} we may assume that $\si\in\cA$, i.e.\ $ \cB^-_{h_0}(y_i)=Y_i$ for all $i$ so that the set of unlucky $y_i$ coincides with  $S_1$, and thanks to \eqref{eq:lale02} we may also assume that $\si$ is such that $|S_1|\geq \bar N:=n^{\b_1/2}$. We call $\cA'\subset \cA$ the set of all $\si\in\cA$ satisfying the latter requirement. Let $\bar S$ denote the first $\bar N$ elements in $S_1$. We are going to show that uniformly in $\si\in\cA'$, for a sufficiently large constant $C>0$,
any $t=\Theta(\log^3(n))$,
\begin{equation}\label{eq:lale3}
\P\Big(\sum_{y\in\bar S}\l_t(y) > \tfrac{C\bar N}{2n}\log^{1-\gamma_1}(n) \,\Big|\, \si\Big)
=o(1).
\end{equation}
Notice that \eqref{eq:lale3} says that, conditionally on a fixed $\si\in\cA'$,  with high probability
$$ \sum_{y\in\bar S}\l_t(y) \leq\tfrac{C\bar N}{2n}\log^{1-\gamma_1}(n),$$ which implies that there are at most $\bar N/2$ vertices $y\in\bar S$ with the property that $  \l_t(y) >  \frac{C}{n}\log^{1-\gamma_1}(n)$. Summarizing, the above arguments and \eqref{eq:lale3} allow one to conclude the unconditional statement that 
with high probability there are at least $\frac12n^{\b_1/2}$ vertices $y\in[n]$ such that
 $$
 \l_t(y) \leq  \tfrac{C}{n}\log^{1-\gamma_1}(n),
 $$
which implies the desired claim \eqref{eq:lale1}, taking e.g.\ $\b=\b_1/3$. 

To prove \eqref{eq:lale3}, consider the sum 
$$
\cX = \sum_{y\in\bar S}\l_t(y). 
$$
We first establish that,  uniformly in $\si\in\cA'$, for any $t=\Theta(\log^3(n))$,
\begin{equation}\label{eq:lale4}
\E\(\cX \tc \si\) =(1+o(1)) \frac{\d_*}m \bar N  \D_*^{-h_0}\d_*^{h_0}.
\end{equation}
If $y$ is unlucky then $P^{h_0}(z,y)=\D_*^{-h_0}$ for any $z\in\partial\cB^-_{h_0}(y)$. Hence, for any $y\in \bar S$: 
$$
\l_t(y) = \frac{\D_*^{-h_0}}{n}\sum_{x\in[n]}\sum_{z\in\partial\cB^-_{h_0}(y)}P^{t-h_0}(x,z) = 
\D_*^{-h_0} \sum_{z\in\partial\cB^-_{h_0}(y)}\l_{t-h_0}(z).
$$
Since $|\partial\cB^-_{h_0}(y)|=\d_*^{h_0}$, and since all $z\in\partial\cB^-_{h_0}(y)$ have the same in-degree $d_z^-=\d_*$,  using symmetry the proof of \eqref{eq:lale4} is reduced to showing that for any $z\in\partial\cB^-_{h_0}(y)$, $t=\Theta(\log^3n)$, 
\begin{equation}\label{eq:lale6}
\E\(\l_t(z) \tc \si\) =(1+o(1)) \frac{d_z^-}m.
\end{equation}
%
 To compute the expected value in \eqref{eq:lale6} we use the so called {\em annealed} process. Namely, observe that
\begin{equation}\label{eq:lale5}
\E\(\l_t(z) \tc \si\) =\frac{1}{n}\sum_{x\in[n]}
\E\(P^t(x,z)\tc\si\)=\frac{1}{n}\sum_{x\in[n]}\P^{a,\si}_x\(X_t=z\),
\end{equation}
where $X_t$ is the annealed walk with initial environment $\si$, and initial position $x$, and $\P^{a,\si}_x$ denotes its law. This process can be described as follows. At time 0 the environment consists of the edges from $\si$ alone, and $X_0=x$; at every step, given the current environment and position, the walker picks a uniformly random tail $e$ from its current position, if it is still unmatched then it picks a uniformly random unmatched head $f$, the edge $ef$ is added to the environment and the position is moved to the vertex of $f$, while if $e$ is already matched then the position is moved to the vertex of the head to which $e$ was matched.   Let us show that uniformly in $x\neq z\in\partial\cB^-_{h_0}(y)$, uniformly in $\si\in\cA'$:
\begin{equation}\label{eq:uni1}
\P^{a,\si}_x\(X_{t}=z\)=(1+o(1))\frac{d_z^-}m.
\end{equation}
 Say that a collision occurs if the walk lands on a vertex that was already visited by using a freshly matched edge. At each time step the probability of a collision is at most $O(t/m)$, and therefore the probability of more than one collision in the first $t$ steps is at most $O(t^4/m^2)=o(m^{-1})$. Thus we may assume that there is at most one cycle in the path of the walk up to time $t$. 
There are two cases to consider: 1) there is no cycle in the path up to time $t$ or there is one cycle that  does not pass through the vertex $z$; 
2) there is a cycle and it passes through  $z$.  In case 1) since $X_{t}=z$ the walker must necessarily pick one of the heads of $z$ at the very last step. Since all heads of $z$ are unmatched by construction, and since the total number of unmatched heads at that time is at least $m-n^{\b_1}\D^{h_0}-t= (1-o(1))m$, this event has probability $(1+o(1))d_z^-/m$.
In case 2) since $x\neq z$ we argue that
in order to have a cycle that passes through  $z$, the walk has to visit $z$ at some time before $t$, which is an event of probability $O(t/m)$, and then must hit back the previous part of the path, which is an event of probability $O(t^2/m)$. This shows that we can upper bound the probability of scenario 2) by $O(t^3/m^2)=o(m^{-1})$. This concludes the proof of \eqref{eq:uni1}. Next, observe that if $x=z$, then the previous argument gives $\P^{a,\si}_z\(X_{t}=z\)=O(t/m)$ which is a bound on  the probability that the walk hits again $z$ at some point within time $t$. In conclusion, \eqref{eq:lale5} and \eqref{eq:uni1} imply \eqref{eq:lale6} which establishes \eqref{eq:lale4}.

Let us now show that 
\begin{equation}\label{eq:lale40}
\E\(\cX^2 \tc \si\) =(1+o(1))\E\(\cX \tc \si\)^2. 
\end{equation}
Once we have \eqref{eq:lale40} we can conclude \eqref{eq:lale3} by using Chebyshev's inequality
together with \eqref{eq:lale4} and the fact that $\d_*^{h_0}\D_*^{-h_0} \le C_2\log^{1-\g_1}(n)$ for some constant $C_2>0$. 
We write 
\begin{equation}\label{eq:lale50}
\E\(\cX^2 \tc \si\) =\sum_{y,y'\in \bar S} \D_*^{-2h_0} \frac1{n^2}\sum_{x,x'\in[n]} \sum_{z\in\partial\cB^-_{h_0}(y)}\sum_{z'\in\partial\cB^-_{h_0}(y')}\bbP^{a,\si}_{x,x'}(X_{t-h_0}=z,X'_{t-h_0}=z'),
\end{equation}
where $\bbP^{a,\si}_{x,x'}$ is the law of two trajectories $(X_s,X'_s)$, $s=0,\dots,t$, that can be sampled as follows. Let $X$ be sampled up to time $t$ according to the previously described annealed measure $\bbP_x^{a,\si}$, call $\si'$ the environment obtained by adding to $\si$ all the edges discovered during the sampling of $X$ and then sample $X'$ up to time $t$ independently, according to $  \bbP_{x'}^{a,\si'}$. 

Let also $\bbP^{a,\si}_{\rm u}$ be defined by 
$$
\bbP^{a,\si}_{\rm u}=\frac1{n^2}\sum_{x,x'\in[n]}\bbP^{a,\si}_{x,x'}.
$$
Thus, under $\bbP^{a,\si}_{\rm u}$ the two trajectories have independent uniformly distributed starting points $x,x'$. With this notation we write
\begin{equation}\label{eq:lale5w0}
\E\(\cX^2 \tc \si\) =
\sum_{y,y'\in \bar S} 
\D_*^{-2h_0}  \sum_{z\in\partial\cB^-_{h_0}(y)}\sum_{z'\in\partial\cB^-_{h_0}(y')}\bbP^{a,\si}_{\rm u}(X_{t-h_0}=z,X'_{t-h_0}=z').
\end{equation}

Let us show that if $z\neq z'$, $t=\Theta(\log^3(n))$:
\begin{equation}\label{eq:lale51}
\bbP^{a,\si}_{\rm u}(X_t=z,X'_t=z') = (1+o(1))\frac{d_z^-d_{z'}^-}{m^2}.
\end{equation}
Indeed, let $A$ be the event that the first trajectory hits $z$ at time $t$ and visits $z'$ at some time before that. Then reasoning as in \eqref{eq:uni1} 
  the event $A$ has probability $O(t/m^2)$. Given any realization $X$ of the first trajectory satisfying this event, the probability of $X'_t=z'$
is at most the probability of colliding with the trajectory $X$ within time $t$, which is $O(t/m)$. On the other hand, if the first trajectory hits $z$ at time $t$ and does  visit $z'$ at any time before that, then the conditional probability of $X'_t=z$, as in  \eqref{eq:uni1} is given by $(1+o(1))d_{z'}^-/m$. This proves  \eqref {eq:lale51} when $z\neq z'$.

If $z=z'$, $t=\Theta(\log^3(n))$, let us show that 
\begin{equation}\label{eq:lale52}
\bbP^{a,\si}_{\rm u}(X_t=z,X'_t=z) = O(1/m^2).
\end{equation}
Consider the event $A$ that the first trajectory $X$ has at most one collision. The complementary event $A^c$ has probability at most $O(t^4/m^2)$. If $A^c$ occurs, then the conditional probability of $X'_t=z$ is at most the probability that $X'$ collides with the first trajectory at some time $s\leq t$, that is $O(t/m)$. Hence,
\begin{equation}\label{eq:lale5a2}
\bbP^{a,\si}_{\rm u}(X_t=z,X'_t=z; A^c) = O(t^5/m^3)=O(1/m^2).
\end{equation}
To prove \eqref{eq:lale52}, notice that  to realize $X'_t=z$ there must be a time $s=0,\dots,t$ such that $X'$ collides with the first trajectory $X$ at time $s$, then $X'$ stays in the digraph $D_1$ defined by the first trajectory for the remaining $t-s$ units of time, and $X'$ hits $z$ at time $t$.  On the event $A$ the probability of spending $h$ units of time in $D_1$ is at most $2\d^{-h}$, and for any $h\in[0,t]$ there are at most $h+1$ points $x$ which have a path of length $h$ from $x$ to $z$ in $D_1$. Therefore
\begin{equation}\label{eq:lale5b2}
\bbP^{a,\si}_{\rm u}(X_t=z,X'_t=z; A) \leq (1+o(1))\frac{d_z^-}{m}\sum_{h=0}^t\frac{2(h+1)}{m}\,2\d^{-h}= O(1/m^2).\end{equation}
Hence, \eqref{eq:lale52} follows from \eqref{eq:lale5a2} and \eqref{eq:lale5b2}.

%
%
%

In conclusion, using \eqref{eq:lale51} and \eqref{eq:lale52} in \eqref{eq:lale5w0}, and recalling  \eqref{eq:lale4}, we have obtained \eqref{eq:lale40}. 
\end{proof}


\subsection{Upper bound on $\pi_{\max}$}\label{suse:pi-max-ub}
As in Section \ref{suse:pimin-ub} we start by replacing $\pi(y)$ with $\l_t(y)=\frac1n \sum_x P^t(x,y)$. In \eqref{distances2} we have seen that if $t=2k\tent$, then w.h.p.
\begin{equation}\label{distances20}
\max_{y\in[n]}|\l_{t}(y)- \pi(y)|\leq e^{-k}.
\end{equation}
Thus, using a union bound over $y\in[n]$, the upper bound in Theorem \ref{th:pimax} follows from the next statement. 
\begin{lemma}\label{lem:maxlat}
There exists $C>0$ such that for any $t=t_n=\Theta(\log^3(n))$, uniformly in $y\in[n]$
 \begin{equation}\label{eq:maxlat1}
\P\(\l_{t}(y)\geq \tfrac{C}n\log^{1-\kappa_0}(n)\)
= o(n^{-1}).
\end{equation}
\end{lemma}
\begin{proof}
Fix $$h_0=\log_{\D_-}\!\log n,$$ 
and call $\sigma$ a realization of the in-neighborhood $\cB_{h_0}^-(y)$. 
Clearly,
$$\lambda_{t+h_0}(y)=\sum_{z\in\cB^-_{h_0}(y)}\lambda_t(z)P^{h_0}(z,y).$$
From \eqref{eq:0gammabounds}, under the event $\cG_y(\hslash)$ from Proposition \ref{pr:tx<=1}, we have $P^{h_0}(z,y)\le 2\delta_+^{-h_0}=2\log^{-\kappa_0}(n)$ for every $z\in\cB^-_{h_0}(y)$. Define 
$$\cX:=\sum_{z\in\cB^-_{h_0}(y)}\lambda_t(z)=\l_t(\cB^-_{h_0}(y)).$$
Then it 
is sufficient to prove that for some constant $C$, uniformly in $\si$  and $y\in[n]$:
\begin{equation}\label{eq:c1}
\P\(\cX>\tfrac{C}{n}\,\log n
\,;\;\cG_y(\hslash) \:\tc\: \si\)=o(n^{-1}).
\end{equation}

By Markov's inequality, for any $K\in\bbN$ and any constant $C>0$:
\begin{equation}\label{eq:markovc1}
\P\(\cX>\tfrac{C}{n}\,\log(n);\cG_y(\hslash) \tc \si\)\leq
\frac{\E\[\cX^K;\cG_y(\hslash)\tc \si\]}{\(\tfrac{C}{n}\log  n\)^K}.
\end{equation}
We fix $K=\log n$, and claim that 
there exists an absolute  constant $C_1>0$ such that
\begin{equation}\label{eq:c2}
\E\[\cX^K;\cG_y(\hslash)\tc \si\]
\le \(\tfrac{C_1}{n}\log n\)^K.
\end{equation}
The desired estimate \eqref{eq:c1} follows from \eqref{eq:c2} and \eqref{eq:markovc1} by taking $C$ large enough. 

We compute the $K$-th moment $\E\[\cX^K;\cG_y(\hslash)\tc \si\]$ by using the annealed process as in \eqref{eq:lale50}. This time we have $K$ trajectories instead of $2$: 
\begin{align}\label{eq:lale051}
\E\[\cX^K;\cG_y(\hslash)\tc \si\]
& =
\frac{1}{n^K}\sum_{x_1,\dots,x_K}\E\[P^t(x_1,\cB^-_{h_0}(y))\cdots
P^t(x_K,\cB^-_{h_0}(y))\,;\,\cG_y(\hslash)\tc \si\]
\nonumber\\&
= \frac{1}{n^K}\sum_{x_1,\dots,x_K} \P^{a,\si}_{x_1,\dots,x_K}\(X_t^{(1)}\in\cB^-_{h_0}(y),\dots,X_t^{(K)}\in\cB^-_{h_0}(y)\,;\, \cG_y(\hslash)
\),
\end{align}
where $X^{(j)}:=\{X_s^{(j)}, s\in [0,t]\}$, $j=1,\dots,K$ denote $K$ annealed walks each with initial point $x_j$, and 
$\P^{a,\si}_{x_1,\dots,x_K}$ denotes the joint law of the trajectories $X^{(j)}$, $j=1,\dots,K$, and the environment, defined as follows. Start with the environment $\si$, and then run the first random walk $X^{(1)}$ up to time $t$ as described after \eqref{eq:lale5}. After that  run the walk $X^{(2)}$ up to time $t$ with initial environment given by the union of edges from $\si$ and the first trajectory, as described in \eqref{eq:lale50}. Proceed recursively until all trajectories up to time $t$ have been sampled. This produces a new environment, namely the digraph given by the union of $\si$ and all the $K$ trajectories. At this stage there are still  many unmatched heads and tails, and we complete the environment by using a uniformly random matching of the unmatched heads and tails.  This defines the coupling $\P^{a,\si}_{x_1,\dots,x_K}$ between the environment and $K$ independent walks in that environment, which justifies the expression in \eqref{eq:lale051}. It is convenient to introduce the notation
$$
\P^{a,\si}_{\rm u}= \frac{1}{n^K}\sum_{x_1,\dots,x_K} \P^{a,\si}_{x_1,\dots,x_K},
$$
for the annealed law of the $K$ trajectories such that independently each trajectory starts at a uniformly random point $X_0^{(j)}=x_j$. Let $D_0=\si$ and let $D_\ell$, for $\ell=1,\dots,K$,  denote the digraph defined by the union of $\si=\cB^-_{h_0}(y)$ with the first $\ell$ paths 
$$\{X_s^{(j)}, 0\leq s\leq t\}, \quad j=1,\dots,\ell.$$ 
Call $D_\ell(\hslash)$ the subgraph of $D_\ell$ consisting of all directed paths in $D_\ell$ ending at $y$ with length at most $\hslash$.
We define $\cG^\ell_y(\hslash)$ as the event $\tx(D_\ell(\hslash))\leq 1$. Notice that 
if the final environment has to satisfy $\cG_y(\hslash)$, then necessarily for every $\ell$ the digraph $D_\ell$ must satisfy $\cG^\ell_y(\hslash)$. Therefore, 
\begin{align}\label{eq:lale052}
\E\[\cX^K;\cG_y(\hslash)\tc \si\]
\leq\P^{a,\si}_{\rm u} \(X_t^{(1)}\in\cB^-_{h_0}(y),\dots,X_t^{(K)}\in\cB^-_{h_0}(y)\,;\, \cG^K_y(\hslash)\).
\end{align}

Define 
\begin{equation}\label{eq:c3}
\cW_\ell = \sum_{x\in V(D_\ell)}[d_x^-(D_\ell)-1]_+,
\end{equation}
where $V(D_\ell)$ denotes the vertex set of $D_\ell$ and $d_x^-(D_\ell)$ is the in-degree of $x$ in the digraph $D_\ell$. Define also the $(\ell,s)$ cluster $\cC_{\ell}^s$ as the digraph given by the union of  
$D_{\ell-1}$ and the truncated path $\{X_u^{(\ell)}, 0\leq  u\leq s\}$. 
We say that the $\ell$-th trajectory $X^{(\ell)}$ has a {\em collision} at time $s\geq 1$ if the edge $(X^{(\ell)}_{s-1},X^{(\ell)}_s)\notin \cC_{\ell}^{s-1}$ and $X^{(\ell)}_s\in  \cC_{\ell}^{s-1}$. 
We say that a collision occurs at time zero if $X^{(\ell)}_0\in D_{\ell-1}$. 
Notice that at least $$\sum_{x\notin\cB^-_{h_0}(y)} [d_x^-(D_\ell)-1]_+$$ collisions must have occurred after the generation of the first $\ell$ trajectories.

Let $\cQ_\ell$ denote the total number of collisions after the  generation of the first    $\ell$ trajectories.
Since $|\cB^-_{h_0}(y)|\leq \D\log n$ one must have 
\begin{equation}\label{eq:c03}
\cW_\ell \leq  \D\log n + \cQ_\ell.
\end{equation}
Notice that the probability of a collision at any given time by any given trajectory is bounded above by $p:=2\D(Kt + \D_-^{h_0})/m=O(\log^4(n)/n)$ and therefore $\cQ_\ell$ is stochastically dominated by the binomial ${\rm Bin}(Kt,p)$. In particular, 
for any $k\in\bbN$:
\begin{equation}\label{eq:c4}
\P\(\cQ_K\geq k
\)\leq (Kt p)^k \leq C_2^k\frac{\log^{8k}(n)}{n^k},
\end{equation}
for some constant $C_2>0$. If $A>0$ is a large enough constant, then
\begin{equation}\label{eq:c5}
\P\(\cQ_K\geq A\log n
\)\leq e^{-\tfrac{A}2 \log^2(n)}.
\end{equation}
If $A\geq 2$ then \eqref{eq:c5} is smaller than the right hand side of \eqref{eq:c2} with e.g.\ $C_1=1$, and therefore from now on we may restrict to proving the upper bound 
\begin{align}\label{eq:lale5000}
\P^{a,\si}_{\rm u} \(X_t^{(1)}\in\cB^-_{h_0}(y),\dots,X_t^{(K)}\in\cB^-_{h_0}(y)\,;\,\cQ_K\leq A\log n \,;\,\cG^K_y(\hslash)\)\leq \(\tfrac{C_1}{n}\log n\)^K,
\end{align}
for some constant $C_1=C_1(A)>0$.
To prove \eqref{eq:lale5000}, define the events
\begin{align}\label{eq:laleb51}
B_\ell=\{X_t^{(1)}\in\cB^-_{h_0}(y),\dots,X_t^{(\ell)}\in\cB^-_{h_0}(y)\,;\,\cQ_\ell\leq A\log n \,;\,\cG^\ell_y(\hslash)\},
\end{align}
for $\ell=1,\dots,K$.
Since $B_{\ell+1}\subset B_{\ell}$, the left hand side in \eqref{eq:lale5000} is equal to 
\begin{equation}
\P^{a,\sigma}_{\rm u}
\(B_1\)\prod_{\ell=2}^{K}\P^{a,\sigma}_{\rm u}\(
B_\ell\tc B_{\ell-1}
\)
\end{equation}
Thus, it is sufficient to show that for some constant $C_1$:
 \begin{equation}\label{eq:lalea52}
\P^{a,\sigma}_{\rm u}\(
B_\ell\tc B_{\ell-1}
\)\leq \tfrac{C_1}{n}\log n\,,
\end{equation}
for all $\ell =1,\dots, K$, where it is understood that $\P^{a,\sigma}_{\rm u}\(
B_1\tc B_{0}
\)=\P^{a,\sigma}_{\rm u}\(
B_1
\).$

Let us partition the event $\{X_t^{(\ell)}\in\cB_{h_0}^-(y) \}$  by specifying the last time in which the walk $X^{(\ell)}$ enters the neighborhood $\cB_{h_0}^-(y)$. Unless the walk starts in $\cB_{h_0}^-(y)$, at that time it must enter from $\partial\cB^-_{h_0}(y)$. Since the tree excess of $\cB_{h_0}^-(y)$ is at most $1$, once the walker is in $\cB_{h_0}^-(y)$, we can bound the chance that it remains in $\cB_{h_0}^-(y)$ for $k$ steps by $2\delta_+^{-k}$. Therefore,
\begin{align*}
\P^{a,\sigma}_{\rm u}\(
B_\ell\tc B_{\ell-1}
\)&\leq \P^{a,\sigma}_{\rm u}\(X^{(\ell)}_{t}\in\cB^-_{h_0}(y)\tc B_{\ell-1}
\)\\&
\leq 2\d_+^{-t}\P^{a,\sigma}_{\rm u}\(X^{(\ell)}_{0}\in\cB^-_{h_0}(y)\tc B_{\ell-1}
\)+ \sum_{j=1}^{t} 2\delta_+^{-(t-j)}
\P^{a,\sigma}_{\rm u}\(X^{(\ell)}_{j}\in\partial\cB^-_{h_0}(y)\tc B_{\ell-1}
\)\\&
\le2t\delta_+^{-t/2}+\sum_{j=t/2+1}^t2\delta^{-(t-j)}\P^{a,\sigma}_{\rm u}\(X^{(\ell)}_{j}\in\partial\cB^-_{h_0}(y)\tc B_{\ell-1}
\)
\end{align*} 
Since $t=\Theta(\log^3(n))$, it is enough to show 
 \begin{equation}\label{eq:lale522}
\P^{a,\sigma}_{\rm u}\(X^{(\ell)}_{j}\in\partial\cB^-_{h_0}(y)\tc B_{\ell-1}
\)\leq \tfrac{C_1}{n}\log n,
\end{equation}
uniformly in $j\in(t/2,t)$ and $1\leq \ell\le K$. 

Let $\cH^{\ell}_0$ denote the event that the $\ell$-th walk makes its first visit to the digraph $D_{\ell-1}$  at the very last time $j$, when it enters $\partial\cB^-_{h_0}(y)$. Uniformly in the trajectories of the first $\ell-1$ walks, at any time there are at most $\D_-| \partial\cB^-_{h_0}(y)|\leq \D_-^{h_0+1}=\D_-\log n$ unmatched heads attached to $\partial\cB^-_{h_0}(y)$, and therefore
   \begin{equation}\label{eq:lale523}
\P^{a,\sigma}_{\rm u}\(X^{(\ell)}_{j}\in\partial\cB^-_{h_0}(y)\,;\, \cH^{\ell}_0\tc B_{\ell-1}
\)=O(| \partial\cB^-_{h_0}(y)|/m)\leq \tfrac{C_1}{n}\log n.
\end{equation}
Let $\cH^{\ell}_2$ denote the event that the $\ell$-th walk makes a first visit to $D_{\ell-1}$ at some time $s_1<j$, then at some time $s_2>s_1$ it exits $D_{\ell-1}$, and then at a later time $s_3\leq j$ enters again the digraph $D_{\ell-1}$. Since each time the walk is outside $D_{\ell-1}$ the probability of entering $D_{\ell-1}$ at the next step is  $O(Kt/m)$, it follows that 
   \begin{equation}\label{eq:laleb53}
\P^{a,\sigma}_{\rm u}\(X^{(\ell)}_{j}\in\partial\cB^-_{h_0}(y)\,;\, \cH^{\ell}_2\tc B_{\ell-1}
\)=O(K^2t^4/m^2)\leq \tfrac{C_1}{n}\log n.
\end{equation}
It remains to consider the case where the $\ell$-th walk enters only once the digraph $D_{\ell-1}$ at some time $s\leq j-1$, and then stays in $D_{\ell-1}$ for the remaining $j-s$ units of time. Calling $\cH^\ell_{1,s}$ this event, and summing over all possible values of $s$, we need to show that
   \begin{equation}\label{eq:lalec53}
\sum_{s=0}^{j-1}
\P^{a,\sigma}_{\rm u}\(X^{(\ell)}_{j}\in\partial\cB^-_{h_0}(y)\,;\, \cH^{\ell}_{1,s}\tc B_{\ell-1}
\)\leq \tfrac{C_1}{n}\log n.
\end{equation}
We divide the sum in two parts: $s\in[0, j-\hslash+h_0]$ and $s\in(j-\hslash+h_0,j)$. For the first part, note that  the walk  
must spend at least  $\hslash-h_0\geq \hslash/2$ units of time in $D_{\ell-1}(\hslash)$, which has probability at most $2\d_+^{-\hslash/2}=O(n^{-\e})$ for some constant $\e>0$, because of the condition $\cG^{\ell-1}_y(\hslash)$ included in the event $B_{\ell-1}$. Since the probability of hitting $D_{\ell-1}$ at time $s$ is  $O(Kt/m)$ we obtain 
    \begin{equation}\label{eq:lale54}
\sum_{s=0}^{j-\hslash+h_0}
\P^{a,\sigma}_{\rm u}\(X^{(\ell)}_{j}\in\partial\cB^-_{h_0}(y)\,;\, \cH^{\ell}_{1,s}\tc B_{\ell-1}
\)=O(Kt^2n^{-\e}/m)\leq \tfrac{C_1}{n}\log n.
\end{equation}
To estimate the  sum over $s\in(j-\hslash+h_0,j)$, notice that the walk has to enter $D_{\ell-1}$ by hitting a point $z\in D_{\ell-1}$ at time $s$ such that 
there exists a path of length $h=j-s$ from $z$ to $\partial\cB^-_{h_0}(y)$ within the digraph $D_{\ell-1}$.  Call $L_{h}$ the set of such points in $D_{\ell-1}$. Hitting this set at any given time $s$ coming from outside the digraph $D_{\ell-1}$ has probability at most $2\D|L_h|/m$, and the path followed once it has entered $D_{\ell-1}$ is necessarily in $D_{\ell-1}(\hslash)$ and therefore has weight at most $2\delta_+^{-h}$. Then,     
\begin{equation}\label{eq:lale55}
\sum_{s=j-\hslash+h_0+1}^{j-1}
\P^{a,\sigma}_{\rm u}\(X^{(\ell)}_{j}\in\partial\cB^-_{h_0}(y)\,;\, \cH^{\ell}_{1,s}\tc B_{\ell-1}
\)\leq \sum_{h=1}^{\hslash-h_0-1}\frac{2\D|L_{h}|}{m} 2\delta_+^{-h},
\end{equation}
Let $A_h\subset L_h$ denote the set of points exactly at distance $h$ from $\partial\cB^-_{h_0}(y)$ in $D_{\ell-1}$. 
We have
\begin{align*}
|A_h| &\leq  \sum_{x\in A_{h-1}} d_x^-(D_{\ell-1})
\\&
     \leq  |A_{h-1}| + \sum_{x\in A_{h-1}} [d_x^-(D_{\ell-1})- 1]_+ 
\\&      \leq |A_{h-2}| + \sum_{x\in A_{h-1}\cup A_{h-2}}  [d_x^-(D_{\ell-1}) - 1]_+
\\&\leq  \dots  \leq  |A_0|+\sum_{x\in A_{0}\cup .... \cup A_{h-1}}  [d_x^-(D_{\ell-1}) - 1]_+
\\& \leq |\partial\cB^-_{h_0}(y)|  + \cW_{\ell-1}. 
\end{align*}
Since $h\leq \hslash=O( \log n)$ and $|\partial\cB^-_{h_0}(y)|\leq \log n$, using 
\eqref{eq:c03} we have obtained 
 \begin{equation}\label{eq:lale56}
|A_h|\leq C_2\log n + \cQ_{\ell-1}.
\end{equation}
On the event $B_{\ell-1}$ we know that $\cQ_{\ell-1}\leq A\log n$, and therefore
$|A_h|\leq C_3\log n$ for some absolute constant $C_3>0$. In conclusion, for all $h\in(0,\hslash-h_0)$
 \begin{equation}\label{eq:lale57}
|L_h|\leq \sum_{\ell=0}^{h}|A_\ell|\leq C_3h\log n.
\end{equation}
Inserting this estimate in \eqref{eq:lale55},
\begin{equation}\label{eq:lale58}
\sum_{s=j-\hslash+1}^{j-1}
\P^{a,\sigma}_{\rm u}\(X^{(\ell)}_{j}\in\partial\cB^-_{h_0}(y)\,;\, \cH^{\ell}_{1,s}\tc B_{\ell-1}
\)
\leq \tfrac{C_4}n \log n.
\end{equation}
Combining \eqref{eq:lale54} and \eqref{eq:lale58} we have proved \eqref{eq:lalec53} for a suitable constant $C_1$.
\end{proof}


\subsection{Lower bound on $\pi_{\max}$}\label{sec:lbpimax}
\label{sec:lowbopimax}
\begin{lemma}\label{lalemmab}
There exist constants $\e,c>0$ such that
 \begin{equation}\label{eq:lale1b}
\P\Big(\exists S\subset[n],\:|S|\geq n^\e\,,\; n \min_{y\in S}\pi(y)\geq c\log^{1-\kappa_1}(n)
\Big)=1-o(1).
\end{equation}\end{lemma}
\begin{proof}
We argue as in the first part of the proof of Lemma \ref{lalemma}.
Namely, let $(\D_*,\d_*)\in\cL$ denote the type realizing the minimum in the definition of $\kappa_1$; see \eqref{eq:def-gamma01}. Let $V_*=\cV_{\D_*,\d_*}$ denote the set of vertices of this type, and let $\a_*\in(0,1)$ be a constant such that  $|V_*|\geq \a_* n$, for all $n$ large enough. Fix a constant $\b_1\in(0,\tfrac14)$ and call $y_1,\dots,y_{N_1}$ the first $N_1:=n^{\b_1}$ vertices in the set $V_*$. Then sample the in-neighborhoods $\cB^-_{h_0}(y_i)$ where 
\begin{equation}\label{eq:lale01b}
h_0=\log_{\D_*}\!\log n - C,
\end{equation}
and call $\si$ a realization of all these neighborhoods. 
As in the proof of Lemma \ref{lalemma}, we may assume that all $\cB^-_{h_0}(y_i)$ are successfully coupled with i.i.d.\ random trees $Y_i$. Next define a $y_i$ {\em lucky} if  $\cB^-_{h_0}(y_i)$  has all its vertices of type $(\D_*,\d_*)$. Then, if $C$ in \eqref{eq:lale01b} is large enough  we may assume that at least $n^{\b_1/2}$ vertices $y_i$ are lucky; see \eqref{eq:lale02}.  As before, we call $\cA'$ the set of $\si$ realizing these constraints. Given a realization $\si\in\cA'$, and some $\e\in(0,\b_1/2)$ we fix the first $n^\e$ lucky vertices $y_{*,i}$, $i=1,\dots,n^\e$. 
Since $\bbP(\cA')=1-o(1)$, letting $S=\{y_{*,i}, i=1,\dots,n^\e\}$, it is sufficient to prove
that for some constant $c>0$ 
\begin{equation}\label{eq:laleob}
\max_{\si\in\cA'}\:\P\left(\min_{i=1,\dots,n^\e}n \pi(y_{*,i})<  c\log^{1-\kappa_1}(n)\tc \si
\right) =o(1).
\end{equation}
To prove \eqref{eq:laleob} we first observe that by \eqref{eq:pigamma} and Lemma \ref{lem:CFcon} it is sufficient to prove the same estimate with $n \pi(y_{*,i})$ replaced by $\G_{h_1}(y_{*,i})$, where $h_1=K\log\log n$ for some large but fixed constant $K$. Therefore, by using symmetry and a union bound it suffices 
to show 
\begin{equation}\label{eq:laleobb}
\max_{\si\in\cA'}\:\P\left(\G_{h_1}(y_{*})<  c\log^{1-\kappa_1}(n)\tc \si
\right) \leq n^{-2\e},
\end{equation}
where $y_*=y_{*,1}$ is the first lucky vertex. 
By definition of lucky vertex,  $\partial\cB^-_{h_0}(y_*)$ has exactly $\D_*^{h_0}$ elements. For each $z\in\partial\cB^-_{h_0}(y_*)$ we sample the in-neighborhood $\cB^-_{h_1-h_0}(z)$. The same argument of the proof of Lemma \ref{le:claim2} shows that the probability that all these neighborhoods  are successfully coupled to i.i.d.\ random directed trees is at least $1-  O(\D^{2h_1}/n)$. On this event we have 
\begin{equation}\label{eq:gastar2b}
\Gamma_{h_1}(y_*)=\d_*^{-h_0}\sum_{i=1}^{\D_*^{h_0}} X_i,
\end{equation}
where $X_i=M^i_{h_1-h_0}$ is defined by 
\eqref{eq:gastar4}.   Then \eqref{eq:gastar06} shows that 
\begin{equation}\label{eq:gastar3b}
\P\(\Gamma_{h_1}(y_*)<\tfrac12\D_*^{h_0}\d_*^{-h_0}\)\leq \exp{\(-c_1\D_*^{h_0}\)},
\end{equation}
for some constant $c_1>0$. Since $\D_*^{h_0}=\D_*^{-C}\log n$ and $\D_*^{h_0}\d_*^{-h_0}=(\d_*/\D_*)^C\log^{1-\kappa_1}(n)$,
this shows that 
\begin{equation}\label{eq:laleobab}
\max_{\si\in\cA'}\:\P\left(\Gamma_{h_1}(y_*)<  c_2\log^{1-\kappa_1}(n)\tc \si
\right) \leq n^{-2\e},
\end{equation}
for some new constant $c_2>0$ and for $\e=c_1\D_*^{-C}/4$. 
This ends the proof of \eqref{eq:laleobb}.
\end{proof}

\section{Bounds on the cover time}\label{se:cover}
In this section we show how the control of the extremal values of the stationary distribution obtained in previous sections can be turned into the bounds on the cover time presented in Theorem \ref{th:covertime}. To this end
we exploit the full strength of the strategy developed by Cooper and Frieze \cite{CF1,cooper2007cover,CF4,CF2}.

\subsection{The key lemma}\label{suse:fvtl}
Given a digraph $G$, write $X_t$ for the position of the random walk at time $t$ and 
write $\bP_x$ for the law of $\{X_t, t\geq 0\}$ with initial value $X_0=x$. In particular, $\bP_x(X_t=y )=P^t(x,y)$ denotes the transition probability. Fix a time $T>0$ and define the event that the walk does not visit $y$ in the time interval $[T,t]$, for $t>T$:
\begin{equation}
\cA^T_y(t)=\{X_s\not=y,\:\forall s\in[T,t] \}.
\end{equation}
Moreover, define the generating function
\begin{equation}
R_y^T(z)=\sum_{t=0}^{T} z^t\,\bP_y(X_t=y ),\qquad z\in\bbC.
\end{equation}
Thus, $R_y^T(1)\geq 1$ is the expected number of returns to $y$ within time $T$, if started at $y$. 
The following statement 
is proved in \cite{CF4}, see also \cite[Lemma 3]{CF2}. 
\begin{lemma}\label{exptail}
Assume that $G=G_n$ is a sequence of digraphs with vertex set $[n]$ and stationary distribution $\pi=\pi_n$, and let $T=T_n$ be a sequence of times such that

\begin{enumerate}[(i)]
\item $\max_{x,y\in[n]}|P^T(x,y)-\pi(y)|\leq n^{-3}$. \label{it:fvtl-1}
\item $T^2\pimax=o(1)$ and $T\pimin\geq n^{-2}$. \label{it:fvtl-3}
\end{enumerate}
Suppose that $y\in[n]$ satisfies:
\begin{enumerate}[(iii)]
\item 
 there exist $K,\psi>0$ independent of $n$ such that 
$$\min_{|z|\le 1+\frac1{K T}}|R^T_y(z)|\ge\psi.$$ 
 \label{it:fvtl-2}
\end{enumerate}
Then there exist  $\xi_1,\xi_2=O(T\pimax)$ such that for all $t\geq T$:
\begin{equation}\label{eq:cfexp}
\max_{x\in[n]}\left|\bP_x\(\cA^T_y(t)\)-\frac{1+\xi_1}{(1+p_y)^{t+1}}\right|\leq e^{-\frac{t}{2K T}},
\end{equation}
where 
\begin{equation}
p_y=(1+\xi_2)\frac{\pi(y)}{R_y^T(1)}.
\end{equation}
\end{lemma}
We want to apply the above lemma to digraphs from our configuration model. 
Thus, our first task is to make sure that the assumptions  of Lemma \ref{exptail} are satisfied. 
From now on we fix the sequence $T=T_n$ as
\begin{equation}\label{def:Tn}
T = \log^3(n).
\end{equation}
From \eqref{cutoff} and the argument in \eqref{distances1} it follows that item (i)
of Lemma \ref{exptail} is satisfied with high probability. Moreover, Theorem \ref{th:pimin} and Theorem \ref{th:pimax} imply that  item (ii) of Lemma \ref{exptail} is also satisfied with high probability. Next, following \cite{CF1}, we define a class of vertices $y\in[n]$ which satisfy item (iii) of Lemma \ref{exptail}. We use the convenient notation
 \begin{equation}\label{def:sigma}
\vartheta = \log\log\log(n).
\end{equation}


\begin{definition}\label{def:ltl} 
We call \emph{small cycle} a collection of $\ell\le 3\vartheta$ edges 
such that their undirected projection forms a simple undirected cycle of length $\ell$. We say that $v\in[n]$ is {\em locally tree-like (LTL)} if its in- and out-neighborhoods up to depth $\vartheta$ are both directed trees  and  they intersect only at $x$. We denote by $V_1$ the set of LTL vertices, and write $V_2=[n]\setminus V_1$ for the complementary set.
\end{definition}
The next proposition can be proved as in \cite[Section 3]{CF1}.
\begin{proposition}\label{pr:nice}
The following  holds with high probability:
\begin{enumerate}
\item
The number of small cycles is at most $ \Delta^{9\vartheta}$. \label{fsm}
		
\item
The number of vertices which are not LTL satisfies  $|V_2|\le\Delta^{15\vartheta}$. \label{mltl}
		
\item
There are no small cycles which are less than $9\vartheta$ undirected steps away. \label{scfa}
\end{enumerate}
\end{proposition}

%
%
%

\begin{proposition}\label{pr:rt-ltl}
 With high probability, uniformly in $y\in
V_1$:
 \begin{equation}\label{eq:rtbo}
 R_y^T(1)=1+O(2^{-\vartheta}).
\end{equation}
 Moreover, there exist constants $K,\psi>0$ such that with high probability, every $y\in V_1$ satisfies item (iii) of Lemma \ref{exptail}. In particular, \eqref{eq:cfexp} holds uniformly in $y\in V_1$.
\end{proposition}
\begin{proof}
We first prove \eqref{eq:rtbo}. Fix $y\in V_1$ and consider the neighborhoods $\cB_{\vartheta}^\pm(y)$ and $\cB_{\hslash}^-(y)$. By Proposition \ref{pr:tx<=1} we may assume that $\cB_{\hslash}^-(y)$ and $\cB_{\vartheta}^+(y)$ are both directed trees except for at most one extra edge. By the assumption $y\in V_1$ we know that $\cB_{\vartheta}^-(y), \cB_{\vartheta}^+(y)$ are both directed trees with no intersection except $y$, so that the extra edge in $\cB_{\hslash}^-(y)\cup\cB_{\vartheta}^+(y)$ cannot be in $\cB_{\vartheta}^-(y)\cup \cB_{\vartheta}^+(y)$. Thus, the following cases only need to be considered: 
\begin{enumerate}
\item there is no extra edge in $\cB_{\hslash}^-(y)\cup \cB_{\vartheta}^+(y)$;
\item the extra edge connects $\cB_{\hslash}^-(y)\setminus \cB_{\vartheta}^-(y)$ to itself 
\item the extra edge connects $\cB_{\vartheta}^-(y)$ to $ \cB_{\hslash}^-(y)\setminus \cB_{\vartheta}^-(y)$;
\item the extra edge connects $\cB_{\vartheta}^+(y)$ to $ \cB_{\hslash}^-(y)\setminus \cB_{\vartheta}^-(y)$.

\end{enumerate} 
In all cases but the last, if a walk started at $y$ returns at $y$ at time $t>0$ then it must exit $\partial\cB_{\vartheta}^+(y)$ and enter  $\partial\cB_{\hslash}^-(y)$, and from any vertex of $\partial\cB_{\hslash}^-(y)$ the probability to reach $y$ before exiting $\cB_{\hslash}^-(y)$ is at most $2\d^{-\hslash}$. Therefore, 
in these cases the number of visits to $y$ up to $T$ is stochastically dominated by  $1+{\rm Bin}(T,2\d^{-\hslash})$ and 
$$
1\leq R_y^T(1) \leq 1 + 2T\d^{-\hslash} = 1 + O(n^{-a}),
$$
for some $a>0$. 
In the last case instead it is possible for the walk to jump from $\cB_{\vartheta}^+(y)$ to $ \cB_{\hslash}^-(y)\setminus \cB_{\vartheta}^-(y)$. Let $E_k$ denote the event that the walk visits $y$ exactly $k$ times in the interval $[1,T]$.
Let $B$ denote the event that the walk visits $y$ exactly $\vartheta$ units of time after its first visit to $\partial\cB_{\vartheta}^-(y)$. 
Then $\bP_y(B)\leq \d^{-\vartheta}$. On the complementary event $B^c$ the walk must enter $\partial\cB_{\hslash}^-(y)$ before visiting $y$, and each time it visits $\partial\cB_{\hslash}^-(y)$ it has probability at most $2\d^{-\hslash}$ to visit $y$ before the next visit to $\partial\cB_{\hslash}^-(y)$. Since the number of attempts is at most $T$ one finds
$$
\bP_y(E_1) \leq \bP_y(B) + \bP_y(E_1,B^c)\leq 
\d^{-\vartheta} + 2T\d^{-\hslash} \leq 2\d^{-\vartheta}.
$$
  By the strong Markov property, 
   $$
\bP_y(E_k) \leq \bP_y(E_1)^k.
$$
Therefore 
$$
R_y^T(1) = 1 + \sum_{k=1}^\infty k \bP_y(E_k) =1+O(\d^{-\vartheta}).
$$
To see that $y\in V_1$ satisfies item (iii) of Lemma \ref{exptail}, take $z\in\bbC$ with $|z|\leq 1+ 1/KT$ and  write 
\begin{align*}
|R^T_y(z)|&\ge 1-\sum_{t=1}^T \bP_y(X_t=y)|z|^t \geq 1- e^{1/K}(R^T_y(1)-1)= 1 - O(\d^{-\vartheta}).
\end{align*}

\end{proof}

\subsection{Upper bound on the cover time}\label{suse:cover-ub}
We prove the following estimate relating the cover time to $\pimin$. From Theorem \ref{th:pimin} this implies the upper bound on the cover time in Theorem \ref{th:covertime}.

\begin{lemma}\label{lem:upcov}
For any constant $\e>0$, with high probability
\begin{equation}\label{eq:upcov1}
\max_{x\in[n]}\,\bE_x\!\(\t_{\rm cov}\)\leq (1+\e)\frac{\log n}{\pimin}.
\end{equation}
\end{lemma}
\begin{proof}
Let $U_s$ denote the set of vertices that are not visited in the time interval $[0,s]$. By Markov's inequality, for all $t_*\geq T$:
\begin{align}
\bE_x[\tau_{\rm cov}]&=\sum_{s\ge 0}\bP_x(\tau_{\rm cov}>s)=\sum_{s\ge 0}\bP_x(U_s\neq\emptyset)\nonumber\\
&\le t_* + \sum_{s\ge t_*}\bE_x\left[|U_s|\right]
=t_*+\sum_{s\ge t_*}\sum_{y\in[n]}\bP_x(y\in U_s)\nonumber\\
&\leq  t_*+\sum_{s\ge t_*}\sum_{y\in[n]}\bP_x(\cA^T_y(s)).
\label{eq:tcove}
\end{align}
Choose $$t_*:= \frac{(1+\e)\log n}\pimin,$$ for $\e>0$ fixed. 
It is sufficient to prove that the last term in \eqref{eq:tcove} is $o(t_*)$ uniformly in $x\in[n]$. 

From 
Proposition \ref{pr:rt-ltl}  we can estimate
\begin{align}\label{eq:guess}
\bP_x(\cA^T_y(s))= \frac{(1+\xi')}{(1+\bar p_y)^{s+1}}\,, 
\end{align}
where $\bar p_y:= (1+\xi)\pi(y)$ with $\xi,\xi'=O(T\pimax)+O(\d^{-\vartheta})=o(1)$ uniformly in $x\in[n],y\in V_1$. Therefore,
\begin{align}\label{eq:guessa}
\sum_{s\ge t_*}\sum_{y\in V_1}\bP_x(\cA^T_y(s))=(1+o(1))\sum_{y\in V_1} \frac{1}{\bar p_y(1+\bar p_y)^{t_*}}.  
\end{align}
Using $\pi(y)\geq \pimin$, \eqref{eq:guessa} is bounded by
\begin{align*}
\frac{(1+o(1))n}{\bar p_y(1+\bar p_y)^{t_*}}& \leq \frac{2n}\pimin \exp {\(-\pimin t_*(1+o(1))\)} 
\leq \frac{1}\pimin =o(t_*),
\end{align*}
for all fixed $\e>0$ in the definition of $t_*$. 

It remains to control the contribution of $y\in V_2$ to the sum in \eqref{eq:tcove}. From Proposition \ref{pr:nice} we may assume that $|V_2|=O(\D^{15\vartheta})$. In particular, it is sufficient to show that with high probability uniformly in $x\in[n]$ and $y\in V_2$:
\begin{align}\label{eq:v2si}
\sum_{s\ge t_*}\bP_x(\cA^T_y(s)) = o(t_*\D^{-15\vartheta}). 
\end{align}
To prove \eqref{eq:v2si}, fix $y\in V_2$ and notice that by Proposition \ref{pr:nice} (3), we may assume that there exists $u\in V_1$ s.t. $d(u,y)<10\vartheta$. If $t_1=t_0+10\vartheta$, $t_0:=4/\pimin$, then
\begin{align*}
\bP_x(\cA^T_y(t_1)^c)&=\bP_x( y\in\{X_T,X_{T+1},\dots,X_{t_1}\})\\
&\geq\bP_x( u\in\{X_T,X_{T+1},\dots,X_{t_0}\})\bP_u(y\in\{X_1,\dots,X_{10\vartheta}\})\\
& \geq  \(1-\bP_x(\cA^T_u(t_0))\)\D^{-10\vartheta}.
\end{align*}
Since $u\in V_1$, as in \eqref{eq:guess}, for $n$ large enough,
\begin{align}
\bP_x(\cA^T_u(t_0))\leq \frac{2}{(1+\bar p_y)^{t_0+1}}\leq \frac12. 
\end{align}
Setting $\g:=\frac12\D^{-10\vartheta}$, we have shown that $\bP_x(\cA^T_y(t_1)^c)\geq \g$. 
Since this bound is uniform over $x$, the
 Markov property implies, for all  
$ k\in\bbN$, 
\begin{equation}\label{eq:assol}
\bP_x(\cA^T_y(s))\le(1-\gamma)^k,\:\:s>k(T+t_1).
\end{equation}
Therefore,
\begin{align*}
\sum_{s\ge t_*}\bP_x(\cA^T_y(s))&\le\sum_{s\ge t_*}(1-\gamma)^{\lfloor s/(T+t_1)\rfloor}
\le \sum_{s\ge t_*}(1-\gamma)^{s/2t_1}		\\
&\leq \frac{\exp{\(-\g t_*/2t_1\)}}{1-\exp{\(-\g/2t_1\)}}= O(t_1/\g) =
 o(t_*\D^{-15\vartheta}). 
\end{align*}
\end{proof}

\subsection{Lower bound on the cover time}\label{suse:cover-lb}
We prove the following stronger statement. 
\begin{lemma}\label{lem:lowbocover}
For some constant $c>0$, with high probability 
\begin{align}\label{eq:lowbocover1}
\min_{x\in[n]}\,\bP_x\!\(\t_{\rm cov} \geq c\, n\log^{\g_1} n\)=1-o(1).
\end{align}
\end{lemma}
Clearly, this implies the lower bound on $T_{\rm cov}=\max_{x\in[n]}\bE_x\!\(\t_{\rm cov}\)$ in Theorem \ref{th:covertime}.
The proof of Lemma \ref{lem:lowbocover} is based on the second moment method as in \cite{CF2}. 
If $W\subset [n]$ is a set of vertices, let $W_t$ be the set 
\begin{equation}\label{eq:Wtset}
W_t=\{y\in W:\,y \text{  is not visited in } [0,t]\}
\end{equation}
Then 
\begin{align}\label{eq:lowbocover2}
\bP_x\!\(\t_{\rm cov}>t\)\geq \bP_x\!\(|W_t|>0\)\geq \frac{\bE_x\!\left[|W_t|\right]^2}{\bE_x\!\left[|W_t|^2\right]}.
\end{align}
Therefore, Lemma \ref{lem:lowbocover} is a consequence of the following estimate. 

\begin{lemma}\label{lem:secmom2}
For some constant $c>0$, with high probability 
 there exists a nonempty set $W\subset [n]$ such that  
\begin{align}\label{it:lbct3b}
\max_{x\in[n]}\,\frac{\bE_x\!\left[|W_{t}|^2\right]}{\bE_x\!\left[|W_t|\right]^2}=1+o(1),\qquad t=c \,n\log^{\g_1} n. 
\end{align}
\end{lemma}

%
%
%
%
%
%
%
%
%
%
%
%
%
We start the proof of Lemma \ref{lem:secmom2} by exhibiting a candidate for the set $W$. 
\begin{proposition}\label{de:candidate}
For any constant $K>0$, with high probability there exists a set $W$ such that
\begin{enumerate}
\item\label{it:cand1} $W\subset V_1$, where $V_1$ is the LTL set from Definition \ref{def:ltl}, and $|W|\geq n^\a$ for some constant $\a>0$.

\item\label{it:cand2} For some constant $C>0$, for all $y\in W$,  
\begin{equation}\label{eq:thebounds}
\pi(y)\leq \tfrac{C}n\,\log^{1-\gamma_1}(n).
\end{equation}

\item\label{it:cand3} For all $x,y\in W$:
\begin{equation}\label{eq:thebounds1}
\left|\pi(x)-\pi(y)\right|\le\pimin\log^{-K}(n).
\end{equation}

\item\label{it:cand4} For all $x,y\in W$: $\min\{d(x,y), d(y,x)\}>2\vartheta$.
\end{enumerate}

\end{proposition}
\begin{proof}
From 
Theorem \ref{th:pimin} 
we know that w.h.p.\ there exists a set $S\subset [n]$ with $|S|>n^\b$ such that \eqref{eq:thebounds} holds. Moreover, a minor modification of the proof of Lemma \ref{lalemma} shows that we may also assume that $S\subset V_1$ and that $\min\{d(x,y), d(y,x)\}>2\vartheta$ for every $x,y\in W$. Indeed, it suffices to generate the out-neighborhoods $\cB^+_\vartheta(y_i)$ for every $i=1,\dots,N_1$ and the argument for \eqref{eq:lale2} shows that these are disjoint trees with high probability.   To conclude,  we observe that there is a $W\subset S$ such that $|W|>n^{\b/2}$ and such that \eqref{eq:thebounds1} holds. Indeed, using $\pimin \geq n^{-1}\log^{-K_1}(n)$ for some constant $K_1$, for any constant $K>0$ we may partition the interval $$[n^{-1}\log^{-K_1}(n),Cn^{-1}\log^{1-\gamma_1}(n)]$$ 
in $\log^{2K}(n)$ intervals of equal length and there must be at least one of them  containing $n^{\b} \log^{-2K}(n)\geq n^{\b/2}$ elements which, if $K$ is sufficiently large, satisfy \eqref{eq:thebounds1}. 
\end{proof}
\begin{proof}[Proof of Lemma \ref{lem:secmom2}]
Consider the first moment $\bE_x\!\left[|W_t|\right]$, where $W$ is the set from Proposition \ref{de:candidate} and $t$ is fixed as $t=c\,n\log^{\g_1}(n)$. For $y\in W\subset V_1$ we use Lemma \ref{exptail} and Proposition \ref{pr:rt-ltl}. As in \eqref{eq:guess} we have
\begin{align}\label{eq:secmom01}
\bP_x(\cA^T_y(t))=(1+o(1))(1+\bar p_y)^{-(t+1)},
\end{align}
where $\bar p_y= (1+o(1))\pi(y)\leq p_W:=2C\,n^{-1}\log^{1-\gamma_1}(n)$, where $C$ is as in \eqref{eq:thebounds}.
Therefore,
\begin{align*}
\bE_x\[|W_t|\]&=\sum_{y\in W}
\bP_x\(y \text{ not visited in } [0,t] \) \\
&\ge-T+\sum_{y\in W}\P(\cA^T_y(t))
\ge-T+(1+o(1))|W|(1+p_W)^{-t}. 
\end{align*}
Taking the constant $c$ in the definition of $t$ sufficiently small, one has $p_W t\leq \a/2 \log n$ and therefore
\begin{align}\label{eq:secmom3}
\bE_x\[|W_t|\]\ge-T+(1+o(1))|W|n^{-\a/2} \geq 
\tfrac12\,n^{\a/2},
\end{align}
where we use $T=\log^3(n)$ and $|W|\geq n^\a$.  
In particular, since $T=\log^3(n)$, \eqref{eq:secmom3} shows that 
\begin{align}\label{eq:secmom4}
 \sum_{y\in W}\P(\cA^T_y(t))  =(1+o(1))\bE_x\[|W_t|\].
\end{align}
Concerning  the second moment $\bE_x\!\left[|W_t|^2\right]$, we have
\begin{align*}
\bE_x\[|W_t|^2\]&=\sum_{y,y'\in W}
\bP_x\(y \text{ and } y'  \text{ not visited in } [0,t] \) \\
&\le\sum_{y,y'\in W}\bP_x\(\cA^T_y(t)\cap \cA^T_{y'}(t)\).
\end{align*}
From this and \eqref{eq:secmom4}, the proof of Lemma \ref{lem:secmom2} is completed by showing, uniformly in $x\in[n],y,y'\in W$:
\begin{align}\label{eq:secmom5}
\bP_x\(\cA^T_y(t)\cap \cA^T_{y'}(t)\)= (1+o(1))\bP_x\(\cA^T_y(t)\) \bP_x\(\cA^T_{y'}(t)\).
\end{align}
We follow the idea of \cite{CF2}. Let $G^*$ denote the digraph obtained from our digraph $G$ by merging the two vertices $y,y'$ into the single vertex $y_*=\{y,y'\}$. Notice that $y^*$ is LTL in the graph $G^*$ in the sense of Definition \ref{def:ltl}. Moreover, $G^*$ has the law of a directed configuration model with the same degree sequence of $G$ except that at $y_*$ it has $d_{y_*}^\pm=d^\pm_y+d^\pm_{y'}$. It follows that we may apply Lemma \ref{exptail} and Proposition \ref{pr:rt-ltl}. Therefore, if $\bP^*_x$ denotes the law of the random walk on $G^*$ started at $x$, as in \eqref{eq:secmom01} we have 
 \begin{align}\label{eq:secmom00}
\bP^*_x(\cA^T_{y_*}(t))=(1+o(1))(1+\bar p_{y_*})^{-t},
\end{align}
uniformly in $x\in[n],y,y'\in W$, where $\bar p_{y_*}=(1+o(1))\pi^*(y_*)$, and $\pi^*$ is the stationary distribution of $G^*$.
 In Lemma \ref{lem:pistar} below we prove that 
 \begin{align}\label{eq:secmom001}
\maxtwo{v\in[n]:}{v\neq y,y'}|\pi(v)-\pi^*(v)|\leq a,\qquad |\pi(y)+\pi(y')-\pi^*(y_*)|\leq a,
\end{align}
where $a:= \pimin \log^{-1}(n)$. 
Assuming \eqref{eq:secmom001}, we can conclude the proof of  \eqref{eq:secmom5}. 
Indeed, letting $P_*$ denote the transition matrix of the graph $G^*$, 
\begin{align*}
\bP_x^*(\cA^T_{y_*}(t))&=\sum_{v\not=y,y'}P^T_*(x,v)\bP^*_v(X_s\not=y_*,\:\forall s\in[1,t-T])\\
&=\sum_{v\not=y,y'}\(\pi^*(v)+O(n^{-3}) \)\bP^*_v(X_s\not=y_*,\:\forall s\in[1,t-T])
\end{align*}
On the other hand,
\begin{align*}
\bP_x(\cA^T_y(t)\cap \cA^T_{y'}(t))&=
\sum_{v\not=y,y'}P^T(x,v)\bP_v(X_s\notin\{y,y'\},\:\forall s\in[1,t-T])\\
&=\sum_{v\not=y,y'}\(\pi(v)+O(n^{-3}) \)
\bP_v(X_s\notin\{y,y'\},\:\forall s\in[1,t-T])
\end{align*}
For all $v\not=y,y'$,
\begin{align*}
\bP^*_v(X_s\neq y_*,\:\forall s\in[1,t-T])&=\bP_v(X_s\not\in\{y,y'\},\:\forall s\in[1,t-T]) \\&\leq \frac{(1+o(1))}{\pimin}P^T(x,v)\bP_v(X_s\not\in\{y,y'\},\:\forall s\in[1,t-T]),
\end{align*}
uniformly in $x\in[n]$, where we have used  condition (i) in Lemma \ref{exptail}.
Therefore, using \eqref{eq:secmom001}
\begin{align*}
&\left|\bP_x\(\cA^T_y(t)\cap \cA^T_{y'}(t)\)-\bP_x^*\(\cA^T_{y_*}(t)\)\right|\\
\qquad &\leq\sum_{v\neq y,y'}
|\pi(v)-\pi_*(v)+O(n^{-3})|\,\bP_v(X_s\not\in\{y,y'\},\:\forall s\in[1,t-T]) \\
&\leq (a+O(n^{-3}))\frac{(1+o(1))}{\pimin}
\sum_{v\neq y,y'}P^T(x,v)\bP_v(X_s\not\in\{y,y'\},\:\forall s\in[1,t-T])\\
&\leq \frac{2a}{\pimin}\,\bP_x(\cA_y(t)\cap\cA_{y'}(t)).
\end{align*}
By definition of $a$ we have $a/\pimin\to 0$ so that 
\begin{equation}\label{eq:aass}
\bP_x(\cA^T_y(t)\cap \cA^T_{y'}(t))=(1+o(1))\bP_x^*(\cA^T_{y_*}(t)).
\end{equation}
 Using \eqref{eq:secmom01},\eqref{eq:secmom00} and \eqref{eq:secmom001} we conclude that 
 \begin{align*}
\bP_x\(\cA^T_y(t)\cap \cA^T_{y'}(t)\)&=(1+o(1))\exp{\(-(1+o(1))(\pi(y)+\pi(y'))t\)} \\&= (1+o(1))\bP_x\(\cA^T_y(t)\)\bP_x\(\cA^T_{y'}(t)\).
\end{align*}
\end{proof}

\begin{lemma}\label{lem:pistar}
The stationary distributions $\pi,\pi^*$ satisfy \eqref{eq:secmom001}.
 \end{lemma}
\begin{proof}
We follow the proof of Eq.\ (107) in \cite{CF2}. The stochastic matrix of the simple random walk on $G^*$ is given by
\begin{equation*}
P_*(v,w)=\begin{cases}
P(v,w)&\text{if }v,w\not=y_*\\
\frac{1}{2}\(P(y,w)+P(y',w) \)&\text{if }v=y_*\\
P(v,y)+P(v,y')&\text{if }w=y_*.
\end{cases}
\end{equation*}
Let $V^* $ denote the vertices of $G^*$. Define the vector $\zeta(v)$, $v\in V^*$ via 
$$\zeta(v)=\begin{cases}\pi_*(v)-\pi(v) & v\neq y_*\\
\pi_*(y_*) - (\pi(x)+\pi(y)) & v=y_*
\end{cases}
$$
We are going to show that 
\begin{equation}\label{eq:ass2}
\max_{v\in V^*}|\zeta(v)| = o(\pimin\log^{-1}(n)),
\end{equation}
which implies \eqref{eq:secmom001}.
A computation shows that
\begin{equation*}
\zeta P_*(w)=\sum_{v\in V^*}\zeta(v)P_*(v,w)=\begin{cases}
\zeta(w)&\text{if }w\not\in\cB_1^+(y)\cup\cB^+_1(y')\\
\zeta(w)+\frac{\pi(y')-\pi(y)}{2}P(y,w)&\text{if }w\in\cB_1^+(y)\\
\zeta(w)+\frac{\pi(y)-\pi(y')}{2}P(y',w)&\text{if }w\in\cB^+_1(y').
\end{cases}
\end{equation*}
Therefore, the vector $\phi:=\zeta(I-P_*)$ satisfies 
\begin{equation*}
|\phi(w)|\le\begin{cases}
0&\text{if }w\not\in\cB_1^+(y)\cup\cB^+_1(y')\\
\frac{|\pi(y)-\pi(y')|}{2\Delta}&\text{otherwise }.
\end{cases}
\end{equation*}
Hence $\phi(v)=0$ for all but at most $2\D$ vertices $v$, and recalling \eqref{eq:thebounds1} 
we have 
\begin{equation}\label{eq:ass1}
|\phi(w)|\le (2\D)^{-1}\pimin\log^{-K}(n).
\end{equation}
Next, consider the matrix 
\begin{equation*}
M=\sum_{s=0}^{T-1}P_*^s,
\end{equation*}
and notice that
\begin{equation*}
\zeta(I-P_*^T)=\phi M.
\end{equation*}
Since $P_*$ and $\pi^*$ satisfy condition (i) in Lemma \ref{exptail}, 
\begin{equation}
P_*^T=\Pi_*+E,\qquad\text{with}\qquad |E(u,v)|\leq n^{-3},\:\:\forall u,v\in V^*,
\end{equation}
where $\Pi_*$ denotes the matrix with all rows equal to $\pi_*$.
We rewrite the vector $\zeta$ as
\begin{equation*}
\zeta=\alpha\pi_*+\rho,
\end{equation*}
where $\alpha\in\bbR$ and $\rho$ is orthogonal to $\pi_*$, that is $$\scalar{\rho}{\pi_*}=\sum_{v\in V^*}\rho(v)\pi_*(v)=0.$$ 
Therefore,
\begin{align*}
\scalar{\phi M}{\rho}=
\scalar{\rho}{(I-E) \rho}.
\end{align*}
Moreover, 
\begin{equation}\label{eq:ass01}
|\scalar{\phi M}{\rho}|\le\sum_{s=0}^{T-1}|\scalar{\phi }{P_*^s\rho}|\le T\frac{\pimax^*}{\pimin^*}\|\phi\|_2\|\rho\|_2,
\end{equation}
where we use 
\begin{align*}
\scalar{P_*^s\psi }{P_*^s\psi}&\leq \frac1{\pimin^*}\sum_{v}\pi^*(v)(P_*^s\psi)^2(v)
\\& \leq \frac1{\pimin^*}\sum_{u,v}\pi^*(v)P_*^s(v,u)\psi^2(u) =\frac1{\pimin^*}\sum_{u}\pi^*(u)\psi^2(u)\leq 
\frac{\pimax^*}{\pimin^*}\|\psi\|_2^2,
\end{align*}
for any vector $\psi:V^*\mapsto \bbR$.
On the other hand,
\begin{equation}\label{eq:ass02}
|\scalar{\rho}{(I-E) \rho}|\ge \|\rho\|_2^2-n^{-3}\(\sum_v |\rho(v)|\)^2 \geq \|\rho\|_2^2(1-n^{-2}).\end{equation}
Using \eqref{eq:ass1}, from \eqref{eq:ass01} and \eqref{eq:ass02} we conclude that
\begin{equation*}
\|\rho\|_2\le 2T\frac{\pimax^*}{\pimin^*}\|\phi\|_2=2T\frac{\pimax^*}{\pimin^*}\times O(\pimin\log^{-K}(n)).
\end{equation*}
From Theorem \ref{th:pimin} applied to $G^*$ we can assume that $ \frac{\pimax^*}{\pimin^*}=O(\log^{K/3}(n))$ if $K$ is a large enough constant. Since $T=\log^3(n)$, with $K$ sufficiently large one has 
\begin{equation*}\label{eq:Klarge}
\|\rho\|_2\leq\pimin\log^{-K/2}(n).
\end{equation*}
Next, notice that
\begin{align*}
0=
\scalar{\zeta}{1}=\scalar{\alpha\pi_*+\rho}{1}=\alpha+\scalar{\rho}{1}.
\end{align*}
Hence
\begin{equation*}
|\alpha|=|\scalar{\rho}{1}|
\le\sqrt n\,\|\rho\|_2\leq \sqrt n\,\pimin\log^{-K/2}(n).
\end{equation*}
In  conclusion,
\begin{align*}
\zeta(v)^2&\le 2\alpha^2\pi_*(v)^2+2\rho(v)^2 
\leq  2n\pimin^2\log^{-K}(n)(\pimax^*)^2+ 2\|\rho\|_2^2 \\&
\leq 2n\pimin^2\log^{-K}(n)(\pimax^*)^2 + 2\pimin^2\log^{-K}(n) \leq 4 \pimin^2\log^{-K}(n),
\end{align*}
which implies 
\eqref{eq:ass2}.
\end{proof}

\subsection{The Eulerian case}\label{suse:cover-eulerian}
We prove Theorem \ref{th:covertime-eulerian}. The strategy is the same as for the proof of Theorem \ref{th:covertime}, with some significant simplifications due to the explicit knowledge of the invariant measure $\pi(x)=d_x/m$.  For the upper bound, it is then sufficient  to prove that, setting $t_*=(1+\e)\b n\log n$,  
\begin{align}\label{eq:topro}
\sum_{y\in V_1}\sum_{s\ge t_*}\bP_x(\cA^T_y(s))+
\sum_{y\in V_2}\sum_{s\ge t_*}\bP_x(\cA^T_y(s))=o(n\log n). 
\end{align}
Letting $\cV_d$ denote the set of vertices with degree $d$, reasoning as in \eqref{eq:guessa} we have
\begin{align*}
\sum_{y\in V_1}\sum_{s\ge t_*}\bP_x(\cA^T_y(s))\leq  (1+o(1))\sum_{d=\d}^\D |\cV_d|
\frac{m}{d(1+(1+o(1))d/m)^{t_*}}\end{align*}
Since $ |\cV_d|=n^{\a_d+o(1)}$, $m=\bar d n$, for any fixed $\e>0$ we obtain
\begin{align}\label{eq:topro1}
\sum_{y\in V_1}\sum_{s\ge t_*}\bP_x(\cA^T_y(s))\leq  
 \frac{2m}{\d}
\,\sum_{d=\d}^\D
\exp {\(-\(\tfrac{d\b}{\bar d} - \a_d\)\log n\)} = O(n),
\end{align}
since by definition $\frac{d\b}{\bar d} - \a_d\geq 0$.
Concerning the vertices $y\in V_2$ one may repeat the argument in \eqref{eq:assol} without modifications, to obtain 
\begin{align}\label{eq:topro2}
\sum_{y\in V_2}\sum_{s\ge t_*}\bP_x(\cA^T_y(s))=o(n\log n). 
\end{align}
Thus, \eqref{eq:topro} follows from \eqref{eq:topro1} and \eqref{eq:topro2}.

It remains to prove the lower bound. We shall prove that for any fixed $d$ such that $|\cV_d|=n^{\a_d+o(1)}$, $\a_d\in(0,1]$, for any $\e>0$, 
\begin{align}\label{eq:lowboeuler1}
\min_{x\in[n]}\,\bP_x\!\(\t_{\rm cov} \geq (1-\e)\frac{\bar d\a_d}d\, n\log^{\g_1} n\)=1-o(1).
\end{align}
We proceed as in the proof of Lemma \ref{lem:secmom2}. Here we choose $W$ as the subset of $\cV_d$ consisting of LTL vertices in the sense of Definition \ref{def:ltl} and such that for all $x,y\in W$ one has $\min\{d(x,y),d(y,x)\}>2\vartheta$. Let us check that this set satisfies 
\begin{align}\label{eq:loel}
|W|\geq n^{\a_d +o(1)}.
\end{align}
Indeed, the vertices that are not LTL are at most $\D^{9\vartheta}$ by Proposition \ref{pr:nice}. Therefore there are at least $|\cV_d|-\D^{9\vartheta}=n^{\a_d +o(1)}$ LTL vertices in $\cV_d$. Moreover, since there are at most $\D^{2\vartheta}$ vertices at undirected distance $2\vartheta$ from any vertex, we can take a subset $W$ of LTL vertices of $\cV_d$ satisfying the requirement that $\min\{d(x,y),d(y,x)\}>2\vartheta$ for all $x,y\in W$ and such that $|W|\geq (|\cV_d|-\D^{9\vartheta})\D^{-2\vartheta}=n^{\a_d +o(1)}.$ 
From here on 
all arguments can be repeated without modifications, with the simplification that we no longer need a proof of Lemma \ref{lem:pistar} since $a$ can be taken to be zero in \eqref{eq:secmom001} in the Eulerian case. The only thing to control is the validity of the bound \eqref{eq:secmom4} with the choice $$t=(1-\e)\frac{\bar d\a_d}d\, n\log n.$$
As in \eqref{eq:secmom4}, it suffices to check that with high probability
\begin{align}\label{eq:lowboeuler10}
\sum_{y\in W}\P(\cA^T_y(t))-T\to \infty.
\end{align}
From \eqref{eq:secmom01} we obtain
\begin{align}\label{eq:lowboeuler101}
\sum_{y\in W}\P\(\cA^T_y(t)\) = (1+o(1))|W|\exp{\(-\tfrac{(1+o(1))d}{m}\,t\)}.
\end{align}
Using \eqref{eq:loel} and $dt/m = (1-\e)\a_d\log n$, \eqref{eq:lowboeuler101} is at least $n^{\e\a_d/2}$ for all $n$ large enough. Since $T=\log^3(n)$ this proves \eqref{eq:lowboeuler10}.

\bigskip

\subsection*{Acknowledgments}
We thank Charles Bordenave and Justin Salez for useful discussions. We acknowledge support of PRIN 2015 5PAWZB ``Large Scale Random Structures", and of INdAM-GNAMPA Project 2019 ``Markov chains and games on networks''.
\bigskip

 	
\bibliographystyle{plain}
\bibliography{bibRWRD}
 	
 \end{document}